\documentclass[11pt,oneside,reqno]{amsart}
\usepackage[english]{babel}
\usepackage{amsmath,amsfonts,amsthm,amssymb,amsbsy,upref,color,graphicx,amscd,hyperref,enumerate,bbm,comment, algorithm, algorithmic,dsfont}
\hypersetup{colorlinks,breaklinks,
            linkcolor=[rgb]{0,0,0},
            citecolor=[rgb]{0,0,0},
            urlcolor=[rgb]{0,0,0}}

\usepackage{xfrac,xcolor}
\usepackage{mathtools,url}
\usepackage[active]{srcltx}
\usepackage{cases} 
\usepackage[shortlabels]{enumitem}
\usepackage{cancel}
\usepackage[capitalise, nameinlink, noabbrev]{cleveref} 

\newcommand{\norm}[2]{\left\Vert {#1} \right\Vert_{#2}}

\newcommand{\mathbbmm}[1]{\text{\usefont{U}{bbm}{m}{n}#1}}
\newcommand{\ind}{\mathbbmm{1}}

\DeclareMathOperator{\dive}{div}

\def\ds{\displaystyle}
\def\eps{{\varepsilon}}
\def\N{\mathbb{N}}
\def\O{\Omega}
\def\R{\mathbb{R}}

\def\A{\mathcal{A}}

\def\HH{\mathcal{H}}

\newcommand{\be}{\begin{equation}}
\newcommand{\ee}{\end{equation}}

\numberwithin{equation}{section}
\theoremstyle{plain}
\newtheorem{theo}{Theorem}[section]

\newtheorem{lemma}[theo]{Lemma}
\newtheorem{coro}[theo]{Corollary}
\newtheorem{prop}[theo]{Proposition}
\newtheorem{defi}[theo]{Definition}
\newtheorem{definition}[theo]{Definition}

\theoremstyle{remark}
\newtheorem{remark}[theo]{Remark}
\newtheorem{oss}[theo]{Remark}

\def\XXint#1#2#3{{\setbox0=\hbox{$#1{#2#3}{\int}$ }
\vcenter{\hbox{$#2#3$ }}\kern-.6\wd0}}

\title[Regularity  of the optimal sets]{Regularity of the optimal sets for\\ a class of integral shape functionals}

\author[G.~Buttazzo, F.~P.~Maiale, D.~Mazzoleni, G.~Tortone, B.~Velichkov]{Giuseppe Buttazzo, Francesco Paolo Maiale,\\ Dario Mazzoleni, Giorgio Tortone, Bozhidar Velichkov }

\address {Giuseppe Buttazzo\newline \indent
	Dipartimento di Matematica, Universit\`a di Pisa \newline \indent
	Largo Bruno Pontecorvo, 5, I--56127 Pisa, Italy}
\email{giuseppe.buttazzo@unipi.it}

\address {Francesco Paolo Maiale \newline \indent
	Scuola Normale Superiore\newline \indent
	Piazza dei Cavalieri 7, 56126 Pisa, Italy}
\email{francesco.maiale@sns.it}

\address {Dario Mazzoleni \newline \indent
Dipartimento di Matematica, University of Pavia \newline \indent
Via Ferrata 5, 27100 Pavia, Italy}
\email{dario.mazzoleni@unipv.it}

\address {Giorgio Tortone \newline \indent
	Dipartimento di Matematica, Universit\`a di Pisa \newline \indent
	Largo Bruno Pontecorvo, 5, 56127 Pisa, Italy}
\email{giorgio.tortone@dm.unipi.it}

\address {Bozhidar Velichkov \newline \indent
Dipartimento di Matematica, Universit\`a di Pisa \newline \indent
Largo Bruno Pontecorvo, 5, 56127 Pisa, Italy}
\email{bozhidar.velichkov@unipi.it}

\begin{document}

\thanks{{\bf Acknowledgments.}
	G.T. and B.V. are supported by the European Research Council's (ERC) project n.853404 ERC VaReg - \it Variational approach to the regularity of the free boundaries\rm, financed by the program Horizon 2020.
	G.B., F.P.M., D.M. and G.T. are members of INDAM-GNAMPA.
	D.M. has been partially supported by the MIUR-PRIN Grant 2020F3NCPX
``Mathematics for industry 4.0 (Math4I4)''}


\begin{abstract}
We prove {the first} regularity theorem for the free boundary of solutions to shape optimization problems involving integral  functionals, for which the energy of a domain $\Omega$ is obtained as the integral of a cost function $j(u,x)$ depending on the solution $u$ of a certain PDE problem on $\Omega$. The main feature of these functionals is that the minimality of a domain $\Omega$ cannot be translated into a variational problem for a single (real or vector valued) state function.

In this paper we focus on the case of affine cost functions $j(u,x)=-g(x)u+Q(x)$, where $u$ is the solution of the PDE $-\Delta u=f$ with Dirichlet boundary conditions. We obtain the Lipschitz continuity and the non-degeneracy of the optimal $u$ from the inwards/outwards optimality of $\Omega$ and then we use the stability of $\Omega$ with respect to variations with smooth vector fields in order to study the blow-up limits of the state function $u$.
{By performing a }triple consecutive blow-up, we prove the existence of blow-up sequences converging to homogeneous stable solution of the one-phase Bernoulli problem and according to the blow-up limits, we decompose $\partial\Omega$ into a singular and a regular part. In order to estimate the Hausdorff dimension of the singular set of $\partial\Omega$ we give a new formulation of the notion of stability for the one-phase problem, which is preserved under blow-up limits and allows to  develop a dimension reduction principle. Finally, by combining a higher order Boundary Harnack principle and a
viscosity approach, we prove $C^\infty$ regularity of the regular part of the free boundary when the data are smooth.
\end{abstract}

\maketitle
\setcounter{tocdepth}{1}
\tableofcontents

\section{Introduction}\label{sintro}

This paper is dedicated to the regularity of the optimal shapes, solutions to shape optimization problems of the form
$$\min\big\{J(A)\ :\ A\in\A\big\},$$
where $\A$ is an admissible class of open, Lebesgue measurable or quasi-open subsets of $\R^d$, and where $ J:\A\to\R$ is a given shape functional, with $J(A)$ usually depending on the solution of a PDE on the domain $A$. This kind of minimization problems arise in different models in Biology, Engineering and Physics (see for example~\cite{HP18,bubu} for an overview) and have been extensively studied from both numerical and theoretical points of view. In particular, there are two classes of shape optimization problems of the form above with long history, both leading to overdetermined elliptic PDE problems with Dirichlet boundary conditions.\medskip

The first class involves the so-called {\it spectral functionals}, that is, functionals depending on the eigenvalues of the Dirichlet Laplacian as
$$J(A)=\varphi\big(\lambda_1(A),\dots,\lambda_k(A)\big)+|A|,$$
where $\varphi:\R^k\to\R$ is a real-valued function and $|A|$ denotes the Lebesgue measure of $A$. The associated shape optimization problems
$$\min\Big\{\varphi\big(\lambda_1(A),\dots,\lambda_k(A)\big)+|A|\ :\ A\subseteq\R^d\Big\},$$
have a long history and are related to the classical question ``Can one hear the shape of the drum?'' and, more generally, to the interplay between the geometry of the domains and the spectrum of the Dirichlet Laplacian. The first results on the characterization of the optimal shapes, for the first and the second Dirichlet eigenvalues, go back to the works of Faber-Krahn (1922) and Krahn-Szeg\"o (1923), and consist in finding explicit minimizers (balls and unions of disjoint balls), which is only possible in some special cases as $J=\lambda_1$ and $J=\lambda_2$. Today, thanks to theory developed by Buttazzo and Dal Maso~\cite{BDM93}, and to the more recent results \cite{bulbk,mp}, it is well-known that, for monotone functionals $\varphi$, minimizers exist in a class of measurable (quasi-open) sets. The regularity of the optimal shapes has also been extensively studied; we refer to \cite{brianconlamboley} and \cite{rtv} for the case of optimal sets of $\lambda_1$ in a box, to \cite{maztv2} for the optimal sets of $\lambda_2$ in a box, and to \cite{bmpv,maztv1,kl1,kl2} (see also \cite{CSY,maztv2}) for functionals involving higher eigenvalues $\lambda_k$. \medskip

The second class of functionals involves {\it integral shape functionals}, namely, for every bounded open set $A\subset\R^d$ we define
$$J(A):=\int_A j(u_A,x)\,dx\,,$$
where the cost function
$$j:\R\times\R^d\to\R\,$$
is fixed and the state function $u_A$ is the (weak) solution of the PDE
\begin{equation}\label{e:state-equation-open}
-\Delta u=f\quad \text{in}\quad A\,,\qquad u\in H^1_0(A)\,,
\end{equation}
 the force term $f:\R^d\to\R$ being also a prescribed function. \medskip

 Optimization problems for integral shape functionals arise in Optimal Control and in models from Mechanics, in which the optimization criteria $j(u_A,x)$ takes into account external factors and forces that might appear not immediately, but only after the design is complete and the state function $u_A$ is already fixed. This type of problems pops up also in population dynamics, when one aims to optimize the population size.
As in the case of spectral functionals, also for integral functionals with monotone cost, the existence of optimal shapes in bounded domains follows from the general theory of Buttazzo and Dal Maso, and again the solutions belong to the large class of measurable (or quasi-open) sets. A general existence result in the class of open sets, was proved recently in \cite{BMV1} and in~\cite{BuSh20}. Precisely, it was shown that if $D$ is a bounded open set in $\R^d$ and if the function $j$ satisfies some suitable growth assumptions, then the shape optimization problem 
$$\min\Big\{\int_A j(u_A,x)\,dx\ :\ A-open,\ A\subseteq D\Big\},$$
has a solution $\Omega\subset D$, $\Omega$ - open.\medskip

\noindent On the other hand, even if the existence theory is quite well understood, there is no regularity theory for the minimizers of integral shape functionals, even in the simplest  case
\begin{equation}\label{e:intro-def-j}
j(u,x)=-g(x)u+Q(x)\,,\quad\text{with}\quad g\neq f\,,
\end{equation}
the regularity of the optimal sets was out of reach. 

\medskip In this paper we prove the first general regularity result for the optimal shapes of integral functionals. In order to make our main result (\cref{thm.main}) easier to read, we introduce the following definition.
\begin{definition}\label{def:d*}
Let $D$ be an open set in $\R^d$.
For $k\in \N\setminus\{0\}$, $\alpha\in [0,1]$ and $N\in\N$,
we call a set $\Omega\subset D$ \emph{{$(k,\alpha, N)$-regular in $D$}} if the free boundary $\partial \Omega\cap D$ is the disjoint union of a regular part $Reg(\partial \Omega)$ and a (possibly empty) singular part $\text{\rm Sing}(\partial\Omega)$ such that:
\begin{itemize}
\item $\text{\rm Reg}(\partial\Omega)$ is a relatively open subset of $\partial \Omega\cap D$ and locally a graph of a $C^{k,\alpha}$-regular function;
\item $\text{\rm Sing}(\partial \Omega)$ is a closed subset of $\partial\Omega\cap D$ and has the following properties:
\begin{itemize}
\item If $d<N$, then $\text{\rm Sing}(\partial \Omega)$ is empty.
\item If $d\ge N$, then the Hausdorff dimension of $\text{\rm Sing}(\partial \Omega)$ is at most $d-N$, namely
\[
\HH^{d-N+\eps}\big(\text{\rm Sing}(\partial \Omega)\big)=0\quad\text{for every}\quad \eps>0.
\]
\end{itemize}
\end{itemize}
Moreover, if the regular part $\text{\rm Reg}(\partial \Omega)$ is $C^\infty$, then we say that $\Omega$ is \emph{$(\infty, N)$-regular in $D$}.
\end{definition}
%

The main result of the present paper is the following.

\begin{theo}\label{thm.main}
Let $D$ be a bounded open set in $\R^d$, where $d\ge 2$. Let
$$f:D\to\R\,,\quad g:D\to\R\,,\quad Q:D\to\R\,,$$
be given non-negative functions.
Suppose that the following conditions hold:
\begin{enumerate}[\rm(a)]
	\item $f,g\in L^\infty(D)\cap C^2(D)$;
	\item there are constants $C_1,C_2 > 0$ such that
	\be\label{e:feg}
	0\le C_1 g\le f\le C_2 g\quad\mbox{in}\quad D.
	\ee
	\item $Q\in C^{2}(D)$ and there are a positive constants $c_Q,C_Q$ such that
	$$0<c_Q\le Q(x)\leq C_Q\quad\text{for every}\quad x\in D\,.$$
\end{enumerate}
Then, there is $\alpha\in(0,1)$ such that every solution $\Omega\subseteq D$ to the shape optimization problem
\begin{equation}\label{e:intro-shape-opt-pb-in-main-teo}
\min \bigg\{\int_A \Big(-g(x)u_A+Q(x)\Big)\,dx\ :\ A\subset D,\ A\ \text{open}\bigg\}\,,
\end{equation}
is $(1,\alpha,5)$-regular in $D$. Moreover, 
if $f,g,Q\in C^\infty(D)$, then $\O$ is $(\infty,5)$-regular in $D$.
\end{theo}
\begin{proof}
The definitions of the regular and the singular parts, $\text{\rm Reg}(\partial \Omega)$ and $\text{\rm Sing}(\partial \Omega)$, of the free boundary $\partial\Omega\cap D$ are given in \cref{s:decomposition}. The $C^{1,\alpha}$ regularity of $\text{\rm Reg}(\partial \Omega)$ is proved in \cref{t:regularity-of-Reg}, while the $C^\infty$ regularity follows from \cref{prop:smoothnessbdry}. The bounds on the dimension of $\text{\rm Sing}(\partial \Omega)$ are given in \cref{t:dimension-of-sing}.
\end{proof}

\begin{oss}[On the bound on the dimension of the singular set $\text{\rm Sing}(\partial \Omega)$]
In \cref{section.stable} we develop a theory about the regularity of the stable global solutions of the one-phase Bernoulli problem (the definition of global stable solution is given in \cref{def:global-stable-solutions}); we show that there is a critical dimension $d^\ast$ (see \cref{def:global-stable-solution-critical-dimension}), in which a singular global solution appears for the first time (see \cref{t:stable-dimension-reduction}), and we prove that the $d^\ast$ can be only $5$, $6$, or $7$ (see \cref{t:stable-critical-dimension}). In \cref{section.sing} (\cref{t:dimension-of-sing}) we use this results to show the following bounds on the singular part $\text{\rm Sing}(\partial \Omega)$ of an optimal set $\Omega$, solution to \eqref{e:intro-shape-opt-pb-in-main-teo}: 
\begin{itemize}
\item if $d<d^\ast$, then $\text{\rm Sing}(\partial \Omega)$ is empty;
\item if $d\ge d^\ast$, then the Hausdorff dimension of $\text{\rm Sing}(\partial \Omega)$ is at most $d-d^\ast$.
\end{itemize}
 In particular, since (by \cref{t:stable-critical-dimension}) $d^\ast\ge 5$, we get that: 
 \begin{itemize}
\item if $d<5$, then $\text{\rm Sing}(\partial \Omega)$ is empty;
\item if $d\ge 5$, then the Hausdorff dimension of $\text{\rm Sing}(\partial \Omega)$ is at most $d-5$.
\end{itemize}
In other words, the precise statement of \cref{thm.main} is that under the conditions (a)-(b)-(c), the optimal sets are $(1,\alpha,d^\ast)$-regular, where $d^\ast$ is the critical dimension from \cref{def:global-stable-solution-critical-dimension}.
\end{oss}

\begin{oss}[On the assumptions (a)-(b)-(c) in \cref{thm.main}]
The $C^2$ regularity assumption in (a) is technical and is related to the use we make of the second order variations of the functional $J$ along vector fields (see \cref{section.variations}). The assumption (b) is used in the proofs of the Lipschitz continuity and the non-degeneracy of the state function $u_\Omega$; we notice that (b) is automatically satisfied when $f$ and $g$ are both bounded from above and from below by positive constants. The bounds from above and below on the weight $Q$ in (c) are usual in the class of Bernoulli-type free boundary problems; these bounds are necessary for the Lipschitz continuity and the non-degeneracy of $u_\Omega$ (see \cref{section.lip}), which are essential ingredients for the blow-up analysis in \cref{section.blowup}. The $C^2$ regularity of $Q$, on the other hand, is used again in the computation of the second variation in \cref{section.variations} and the passage to the blow-up limit in the proof of \cref{t:dimension-of-sing}. 
\end{oss}	

\begin{oss}[On the existence of optimal sets]
In \cite[Theorem~1.1]{BMV1} it was proved that if $D$ is a bounded open subset of $\R^d$ and if $f, g, Q$ satisfy the following conditions:
\begin{itemize}
\item $f,g\in L^\infty(D)$;
\item there are positive constants $C_1\le C_2$ such that $0\le C_1g\le f\le C_2g$ in $D$;
\item $Q\in L^\infty(D)$, $Q\ge 0$ in $D$,
\end{itemize}
then, there is an open set $\Omega\subset D$ solution to the shape optimization problem \eqref{e:intro-shape-opt-pb-in-main-teo}.
\end{oss}


The presence of the inclusion constraint $\Omega\subset D$ is essential for the existence theory for general shape optimization problems (see for instance \cite{BDM93} and \cite{BMV1}). In the case of integral functionals with affine cost, as the one in \eqref{e:intro-shape-opt-pb-in-main-teo}, the inclusion constraint can be removed. In the next theorem, which we prove in \cref{s:existence}, we show that optimal sets exist in $\R^d$. Moreover, we prove that the optimal sets in $\R^d$ are bounded, which implies that they are solutions to \eqref{e:intro-shape-opt-pb-in-main-teo} in some sufficiently large ball $D:=B_R$, so the regularity of the free boundary in $\R^d$ is a consequence of \cref{thm.main}.

\begin{theo}\label{cor.main}
In $\R^d$, $d\ge 2$, let $f,g,Q:\R^d\to\R$ be non-negative functions.\\
Suppose that the following conditions hold:
	\begin{enumerate}[\rm(a)]
		\item $f,g\in L^\infty(\R^d)\cap L^1(\R^d)\cap C^2(\R^d)$ and that
	$$f(x)>0\quad\mbox{and}\quad g(x)>0\quad\mbox{for every}\quad x\in\R^d.$$
		\item $Q\in C^{2}(\R^d)$ and there are positive constants $c_Q,C_Q$ such that
		$$0<c_Q\le Q(x)\leq C_Q\quad\text{for every}\quad x\in \R^d\,.$$
	\end{enumerate}
Then, there is an open set $\Omega\subset\R^d$, which is a solution to the shape optimization problem
	\begin{equation}\label{e:intro-shape-opt-pb-in-main-cor}
	\min \bigg\{\int_A \Big(-g(x)u_A+Q(x)\Big)\,dx\ :\ A\subset \R^d,\ A\ \text{open},\ |A|<+\infty\bigg\}\,,
	\end{equation}
and every solution $\Omega$ to~\eqref{e:intro-shape-opt-pb-in-main-cor} is bounded and $(1,\alpha,5)$-regular in $\R^d$. Moreover, if  $f,g,Q\in C^\infty(\R^d)$, then $\O$ is $(\infty,5)$-regular in $\R^d$.

%

%
\end{theo}


\subsection{Integral shape functionals in the case $f=cg$} 
In this section we briefly discuss the case in which $f$ and $g$ are proportional, which is the only instance of integral functional studied in the literature. Precisely, we claim that if $f$ and $g$ are such that
\begin{equation}\label{e:intro-proportional}
g=\frac{1}{2\lambda^2} f\quad\text{for some constant}\quad \lambda>0\,,
\end{equation}
the shape optimization problem \eqref{e:intro-shape-opt-pb-in-main-teo} is equivalent to the Bernoulli free boundary problem
\begin{equation}\label{e:intro-Bernoulli}
\min\bigg\{\int_D\Big(\frac12|\nabla u|^2-f(x)u+\lambda^2Q(x)\ind_{\{u\neq 0\}}\Big)\,dx\ :\ u\in H^1_0(D)\bigg\}.
\end{equation}
Fix a solution $u\in H^1_0(D)$ to \eqref{e:intro-Bernoulli}, for which the set $\{u\neq0\}$ is open, and fix an optimal set $\Omega$ for \eqref{e:intro-shape-opt-pb-in-main-teo} with a state function $u_\Omega$. Since $u$ satisfies
$$-\Delta u=f\quad\text{in}\quad \{u\neq0\}\,\qquad u\in H^1_0(\{u\neq0\})\,,$$
by integrating by parts, we have that
\begin{align*}
\int_D\Big(\frac12|\nabla u|^2-f(x)u+\lambda^2Q(x)\ind_{\{u\neq 0\}}\Big)\,dx&=\int_D\Big(-\frac{1}{2} f(x)u+\lambda^2Q(x)\ind_{\{u\neq 0\}}\Big)\,dx\\
&=\lambda^2\int_{\Omega_u}\big(-g(x)u+Q(x)\big)\,dx=\lambda^2J(\{u\neq0\})\,.
\end{align*}
Analogously, since $\{u_\Omega\neq0\}\subset\Omega$ and $Q$ is positive, we have
\begin{align*}
\int_D\Big(\frac12|\nabla u_\Omega|^2-f(x)u_\Omega+\lambda^2Q(x)\ind_{\{u_\Omega\neq 0\}}\Big)\,dx\le\lambda^2 J(\Omega).
\end{align*}
Thus, if $u$ is a solution to \eqref{e:intro-Bernoulli}, then $J(\{u\neq0\})\le J(\Omega)$ and so, $\{u\neq0\}$ is a solution to \eqref{e:intro-shape-opt-pb-in-main-teo}. Conversely, if $\Omega$ minimizes \eqref{e:intro-shape-opt-pb-in-main-teo}, then $u_\Omega$ is a minimizer of \eqref{e:intro-Bernoulli}. Thus, for proportional $f$ and $g$, the problem \eqref{e:intro-shape-opt-pb-in-main-teo} is equivalent to \eqref{e:intro-Bernoulli}; moreover, we notice that the argument above does not require the positivity of $f$ and $g$, so the equivalence of the problems \eqref{e:intro-Bernoulli} and \eqref{e:intro-shape-opt-pb-in-main-teo} holds also when $f$ and $g$ change sign.

The regularity of the solutions to the free boundary problem  \eqref{e:intro-Bernoulli} is nowadays well-understood (at least when $Q$ satisfies the condition (c) of \cref{thm.main}). When $f\ge0$, the regularity of $\partial\Omega$ follows from the regularity theory for the one-phase Bernoulli problem (see \cite{altcaf}, \cite{w}, \cite{desilva}, \cite{cjk}, \cite{js}, \cite{dj}, \cite{GustShahg}). If the right-hand side $f$ changes sign, then \eqref{e:intro-Bernoulli} becomes a two-phase Bernoulli problem, for which the regularity of the free boundary was obtained recently in \cite{sv} and \cite{DSV}.  \medskip

Finally, we notice that when $f$ and $g$ are not proportional, the state function $u_\Omega$ of an optimal set $\Omega$ (that is, a solution to the shape optimization problem \eqref{e:intro-shape-opt-pb-in-main-teo}) is not a minimizer of a free boundary functional as the one in \eqref{e:intro-Bernoulli}. In particular, this implies that one can test the optimality of $u_\Omega$ only with functions $\widetilde u$ which are themselves state functions of some $\widetilde \Omega$. This means that a function $\widetilde u$, that differs from $u$ only in a small ball $B_r$, cannot be used to test the optimality of $u_\Omega$ (truncations, harmonic replacements and radial extensions in small balls are not admitted), which makes most of the classical free boundary regularity results impossible to apply.


\subsection{Adjoint state and optimality condition on the free boundary} Let us go back to the general case when $f$ and $g$ are not proportional. We will show that the optimality condition on the boundary $\partial\Omega\cap D$ of an optimal open set $\Omega$ for \eqref{e:intro-shape-opt-pb-in-main-teo} leads to a free boundary problem involving the state function $u_\Omega$. In order to see this, we introduce the adjoint state function $v_\Omega$ as follows: for every open set $A\subset D$ we will denote by $v_A$ the weak solution to the problem
\be\label{e:stateqv}
-\Delta v_A=g\quad \text{in}\quad A\,,\qquad v_A\in H^1_0(A)\,.
\ee
By an integration by parts one can see that
$$\int_{D}gu_A\,dx=\int_D\nabla u_A\cdot\nabla v_A\,dx=\int_{D}fv_A\,dx\,,$$
which means that the two state variables $u_A$ and $v_A$ are interchangeable. Precisely, an open set $\Omega\subset D$ is a solution to \eqref{e:intro-shape-opt-pb-in-main-teo} if and only if it minimizes
\begin{equation}\label{e:intro-shape-opt-pb-v-Omega}
\min \bigg\{\int_A \Big(-f(x)v_A+Q(x)\Big)\,dx\ :\ A\subset D,\ A\ \text{open}\bigg\}\,.
\end{equation}
Sometimes, it is more convenient to consider simultaneously the two state functions, by using the following equivalent formulation, which is symmetric in $v_\Omega$ and $u_\Omega$
\begin{equation}\label{e:intro-shape-opt-pb-u-v}
\min \bigg\{\int_A \Big(\nabla u_A\cdot\nabla v_A-f(x)v_A-g(x)u_A+Q(x)\Big)\,dx\ :\ A\subset D,\ A\ \text{open}\bigg\}\,.
\end{equation}
Throughout the paper we will denote the functional from \eqref{e:intro-shape-opt-pb-in-main-teo} by $\mathcal F$. Precisely, we set
\be\label{e:Functional}
\mathcal{F}(\Omega;D):=\int_D \Big(-g(x)u_\Omega+Q(x)\ind_\Omega\Big)\,dx\,,
\ee
which after an integration by parts has the symmetric form
$$\mathcal{F}(\Omega;D)=\int_{D}\Big(\nabla u_\O\cdot\nabla v_\O\, - g(x)u_\O - f(x) v_\O+Q(x)\ind_{\Omega}\Big)\,dx\,,$$
so if $\xi\in C^\infty_c(D;\R^d)$ is a smooth compactly supported vector field in $D$ and $\Omega_t:=(Id+t\xi)(\Omega)$, then the first variation of $\mathcal F$ along $\xi$ is given by (see \cref{l:firstv})
\begin{align*}
\delta \mathcal{F}(\Omega;D)[\xi]:&=\frac{\partial}{\partial t}\bigg|_{t=0}\mathcal{F}({\Omega_t},D)\\
&=\int_\O \Big(\big(\nabla u_\Omega \cdot \nabla v_\Omega + Q(x) \big)\dive\xi - \nabla u_\O \cdot \big((\nabla\xi) + (D\xi)\big)\nabla v_\O \Big)\, dx\\
&\qquad-\int_\Omega \Big(u_\O \dive(g\xi) + v_\O \dive(f\xi)\Big)\,dx\,.
\end{align*}
Moreover, if $\Omega$ is a minimizer and $\partial\Omega$ is smooth, then an integration by parts gives that
$$\delta \mathcal{F}(\Omega;D)[\xi] = \int_{\partial \Omega} (\nu \cdot \xi)\big(Q-|\nabla u_\O ||\nabla v_\O|\big)\,d\mathcal H^{d-1}=0\quad\text{for every}\quad \xi\in C^\infty_c(D;\R^d)\,,$$
so the state functions $u_\Omega$ and $v_\Omega$ are (at least formally) solutions to the free boundary system
\be\label{e:intro-generalsystem}
\begin{cases}
	-\Delta u_\Omega=f\quad\text{in}\quad \Omega\,,\smallskip\\
	-\Delta v_\Omega=g\quad\text{in}\quad \Omega\,,\smallskip\\
	u_\Omega=v_\Omega=0\quad\text{and}\quad|\nabla u_\Omega||\nabla v_\Omega|=Q\quad\text{on}\quad D\cap\partial\Omega\,.
\end{cases}
\ee
%
%
\begin{oss}\label{oss:intro-epsilon-regularity}
An epsilon-regularity theorem for viscosity solutions of the above system was proved recently in \cite{mtv.flat}. Precisely, it was shown that if:
\begin{itemize}
\item $u_\Omega,v_\Omega:D\to\R$ are continuous and non-negative functions such that $$\Omega=\{u_\Omega>0\}=\{v_\Omega>0\}\,;$$
\item $(\Omega,u_\Omega,v_\Omega)$ is a viscosity solutions to \eqref{e:intro-generalsystem}\,;
\item $u_\Omega$ and $v_\Omega$ are  {$\eps$-\it flat} (in a suitable sense) in a ball $B_r(x_0)\subset D$ centered on $\partial\Omega\,$;
\end{itemize}
then the free boundary $\partial\Omega$ is $C^{1,\alpha}$-regular in $B_{\sfrac{r}2}(x_0)$. 

\noindent We notice that an epsilon-regularity theorem for a similar free boundary system, with $f\equiv g\equiv 0$, was studied in \cite{agcasp} in the context of a different shape optimization problem. Precisely, in \cite{agcasp}, by using the minimality of $\Omega$, it was shown that it is an NTA domain, so by the Boundary Harnack Principle for harmonic functions, the system reduces to a one-phase problem for the function $u_\Omega$, for which the epsilon-regularity theorem is known by \cite{altcaf} and \cite{desilva}.
In Section~\ref{s:heat-conduction} we will detail how our results can be applied to the problem of~\cite{agcasp}.
\end{oss}

\subsection{Regularity of the free boundary and dimension of the singular set} 
The key steps in the proof of \cref{thm.main} are the following:
\begin{itemize}
\item Prove that if $\Omega$ is an optimal set for \eqref{e:intro-shape-opt-pb-in-main-teo}, then it is a viscosity solution to \eqref{e:intro-generalsystem}. This, in combination with the epsilon-regularity result cited in \cref{oss:intro-epsilon-regularity}, will imply that the {\it flat} part of the free boundary is smooth. 
\item Show that there exists a critical dimension $d^*\in\{5,6,7\}$ such that in dimension $d< d^*$ all the points on $\partial\Omega\cap D$ are regular.
\item Prove that a Federer-type dimension reduction principle holds for solutions to \eqref{e:intro-shape-opt-pb-in-main-teo}.
\end{itemize}	

There are two main difficulties in following the program outlined in the three points above.\medskip\\ The first difficulty is in the fact that the first order optimality condition
$$\delta \mathcal{F}(\Omega;D)[\xi]=0\quad\text{for every}\quad \xi\in C^\infty_c(D;\R^d)\,,$$
is not leading to a monotonicity formula for $u_\Omega$ and $v_\Omega$; in particular, we do not know if the blow-up limits of $u_\Omega$ and $v_\Omega$ are in general homogeneous.\medskip\\
The second obstruction comes from the impossibility to make external perturbations of $u_\Omega$ and $v_\Omega$ (that is no perturbations of the form $\widetilde u=u_\Omega+\phi$ and $\widetilde v=v_\Omega+\psi$ are allowed), so the only information conserved along blow-up sequences is the one contained in the internal variations of $\Omega$ along smooth vector fields. \medskip

We overcome these difficulties by using the first and the second variation of $\mathcal F$. Indeed, suppose that $\Omega$ is optimal in $D$ and consider the flow $\Phi_t$ associated to a compactly supported vector field $\xi\in C^\infty_c(D;\R^d)$. Then, setting  $\Omega_t:=\Phi_t(\Omega)$, we get that the function 
$$t\mapsto \mathcal{F}(\Omega_t;D),$$
has a minumum in $t=0$, so we have 
\begin{equation}\label{e:intro-stable-critical-point}
\frac{\partial}{\partial t}\Big|_{t=0}\mathcal{F}(\Omega;D)=0\qquad\text{and}\qquad \frac{\partial^2}{\partial t^2}\Big|_{t=0}\mathcal{F}(\Omega;D)\ge0\,,
\end{equation}
that is, $\Omega$ is a stable critical point of the shape functional
$$\Omega\mapsto \mathcal{F}(\Omega;D).$$
We prove that this notion is stable under blow-up limits at free boundary points $x_0\in\partial\Omega$:
$$u_0(x)=\lim_{n\to\infty}\frac{u_{\Omega}(x_0+r_nx)}{r_n}\qquad\text{and}\qquad v_0(x)=\lim_{n\to\infty}\frac{v_{\Omega}(x_0+r_nx)}{r_n}\,.$$
Thus $\Omega_0=\{u_0>0\}=\{v_0>0\}$ is a stable critical point for the same functional, but this time with $f\equiv g\equiv0$ and $Q\equiv Q(x_0)$. Now, since $u_0,v_0:\R^d\to\R$ are harmonic on $\Omega_0$, we can apply the Boundary Harnack Principle from \cite{mtv.BH} to obtain that the ratio
$$\frac{u_0}{v_0}:\Omega_0\to\R$$
is H\"older continuous in $\Omega_0$, up to $\partial\Omega_0$. After a second blow-up (this time in zero), we get that the positivity set $\Omega_{00}$ of the functions
$$u_{00}(x)=\lim_{m\to\infty}\frac{u_{0}(r_mx)}{r_m}\qquad\text{and}\qquad v_{00}(x)=\lim_{m\to\infty}\frac{v_{0}(r_mx)}{r_m}\,,$$
is a stable critical point (in every ball $B_R\subset\R^d$) for the functional $\mathcal F(\cdot,B_R)$, still with $f\equiv g\equiv0$ and $Q\equiv Q(x_0)$. Moreover, by the Boundary Harnack Principle, we also have that $u_{00}$ and $v_{00}$ are proportional. Thus, $u_{00}$ is (up to a constant) a stable critical point (in the sense of \eqref{e:intro-stable-critical-point}) for the Alt-Caffarelli's one-phase functional
$$\mathcal G(u;B_R):=\int_{B_R}|\nabla u|^2\,dx+|\{u>0\}\cap B_R|\quad\text{in every ball}\quad B_R\subset \R^d\,,$$
so the third blow-up in zero
$$u_{000}(x)=\lim_{k\to\infty}\frac{u_{00}(r_kx)}{r_k}\,$$
is a 1-homogeneous stable critical point for the Alt-Caffarelli's functional $\mathcal G$. \medskip

The existence of a homogeneous blow-up $u_{000}$ is a key element in the proof of \cref{thm.main}. Indeed, it allows to prove (see \cref{p:viscosity-solutions}) that the state functions $u_\Omega$ and $v_\Omega$ are viscosity solutions to the system \eqref{e:intro-generalsystem} by showing that if we take a free boundary point admitting a one-sided tangent ball, then the homogeneous blow-up limits $u_{000}$ and $v_{000}$ constructed above are half-plane solutions. This, in combination with the epsilon-regularity theorem cited in \cref{oss:intro-epsilon-regularity}, implies that, in a neighborhood of any boundary point admitting a half-plane solution as blow-up limit, the free boundary is $C^{1,\alpha}$-regular; we will call these points {\it regular points} (see \cref{s:decomposition}), while the remaining part of the free boundary (if any) will be called {\it singular}. Finally, we notice that the smoothness of the regular part requires only the criticality of $\Omega$ (the first part of \eqref{e:intro-stable-critical-point})\medskip  

It is natural to expect that the stability of $\Omega$ (the second part of \eqref{e:intro-stable-critical-point}) leads to an estimate on the dimension of the singular set; in fact, the bounds on the critical dimension (the dimension in which a singularity appears for the first time) for minimizers of the one-phase Alt-Caffarelli functional rely (see \cite{cjk} and \cite{js}) on the well-known {\it stability inequality} of Caffarelli-Jerison-Kenig, which was originally obtained in \cite{cjk} through a particular second order variation of the functional $\mathcal G$. On the other hand, this stability inequality is not easy to handle when it comes to passing to blow-up limits and developing a dimension reduction principle. In \cref{section.stable}, we give a different formulation of the stability, which uses the second variation along vector fields, as defined in \eqref{e:intro-stable-critical-point}. Again in \cref{section.stable}, we show that our notion of stability allows to develop a dimension reduction principle and that there is a critical dimension $d^\ast$, which is the smallest dimension admitting one-homogeneous stable solutions with singularity (see \cref{t:stable-dimension-reduction}). Then, in \cref{p:stable-cones-satisfy-stability-inequality}, we prove that on smooth cones (that is, cones with isolated singularity) our notion of stability is equivalent to the stability inequality of Caffarelli-Jerison-Kenig. Thus, we obtain the bound $5\le d^\ast\le 7$ on the critical dimension $d^\ast$ as a consequence of the results of Jerison-Savin \cite{js} and De Silva-Jerison \cite{dj}. Finally, in \cref{t:dimension-of-sing}, we prove the bounds on the singular set from \cref{thm.main} by (again) a triple blow-up argument, which allows to transfer the dimension bounds from \cref{t:stable-dimension-reduction} to the singular set of the solutions to \eqref{e:intro-shape-opt-pb-in-main-teo}.

\subsection{Plan of the paper}

In \cref{section.variations} we compute the first and second variations of the functional $\mathcal F$. In \cref{section.lip} we prove the non-degeneracy and Lipschitz continuity of the state variables. In \cref{section.blowup} we perform a blow-up analysis by proving the existence of {triple} blow-up sequences converging to a homogeneous limit.\medskip

In \cref{s:decomposition}, we use the blow-up limits from  \cref{section.blowup} to define the decomposition of the free boundary into a regular part and a singular part.\medskip

In \cref{section.reg} we prove that the state functions $u_\Omega, v_\Omega$ of an optimal set $\Omega$ for \eqref{e:intro-shape-opt-pb-in-main-teo} are viscosity solutions of the free boundary system \eqref{e:intro-generalsystem}. Using this information, in \cref{t:regularity-of-Reg}, we prove that the regular part of the free boundary is $C^{1,\alpha}$-smooth. \medskip

In \cref{section.stable} we define the notion of a global stable solution to the one-phase Bernoulli (Alt-Caffarelli) problem and we study the dimension of the singular sets for these global stable solutions. We notice that this section can also be read indipendently and that, together with \cref{section.variations}, it contains the key results for the analysis of the singular set.\medskip

In \cref{section.sing} we use the first and the second variations from \cref{section.variations}, the triple blow-up procedure from \cref{section.blowup} and the theory from \cref{section.stable} in order to prove the dimension bounds on the singular set. This concludes the proof of \cref{thm.main}, which follows from \cref{t:dimension-of-sing}, \cref{t:regularity-of-Reg} and \cref{prop:smoothnessbdry}.\medskip

In \cref{s:existence} we address the existence of optimal sets in $\R^d$ and we prove Theorem~\ref{cor.main}. Ultimately, in \cref{s:heat-conduction} we apply the analysis of \cref{section.stable} and \cref{section.sing}  to optimal sets arising  sets arising from the heat conduction problem studied in \cite{agcasp}. 

\subsection{Notations}\label{s:notations}
In the whole paper we use the notation
$$B_r(x_0)=\{x\in \R^d : |x-x_0|<r\},$$
for the ball of radius $r$ centered at point $x_0$ and, when $x_0=0$, we write $B_r=B_r(0)$; we denote by $\omega_d$ the Lebesgue measure of a ball of radius one in $\R^d$.
For any set $A\subset \R^d$, we set
\[
\mathds{1}_{A}(x):=\begin{cases}
1, &\text{if }\ x\in A,\\
0, &\text{if }\ x\in \R^d\setminus A.
\end{cases}
\]
Given a non-negative function $u$ we often denote its positivity set as $\Omega_u:=\{u>0\}$.

We denote by $H^1(\R^d)$ the set of Sobolev functions in $\R^d$, that is, the closure of the smooth functions with bounded support $C^\infty_c(\R^d)$ with respect to the usual Sobolev norm
$$\|\varphi\|_{H^1}^2=\int_{\R^d}\big(|\nabla\varphi|^2+\varphi^2\big)\,dx\,.$$
Given an open set $\Omega\subset\R^d$, we define the Sobolev space $H^1_0(\Omega)$ as the closure, with respect to $\|\cdot\|_{H^1}$, of the space $C^\infty_c(\Omega)$ of smooth functions compactly supported in $\Omega$. Thus, every $u\in H^1_0(\Omega)$ is identically zero outside $\Omega$ and we have the inclusion $H^1_0(\Omega)\subset H^1(\R^d)$.

We will sometimes use the following notation for minimum and maximum of two real numbers:\[
\min\{a,b\}=a\wedge b,\qquad \max\{a,b\}=a\vee b,\qquad a,b\in\R.
\]

Given a function $u:\R^d\to\R$ we will denote by $\nabla u$ and $Du$ the vectors column and row with components the partial derivatives of $u$, while $D^2u$ will be the Hessian matrix of $u$. Given a vector field $F:\R^d\to\R^d$ with components $F_k$, $k=1,\dots,d$ we will denote by $\nabla F$ the $d\times d$ matrix with columns $\nabla F_k$, $k=1,\dots, d$ and rows $(\partial_jF_1,\partial_jF_2,\dots,\partial_jF_d)$. By convention $DF:=(\nabla F)^T$, where for any matrix $M\in \R^{d\times d}$, we will denote by $M^T$ its transpose. Given a vector field $V:\R^d\to\R^d$ and a matrix with variable coefficients $M=(m_{ij})_{ij}:\R^d\to\R^{d\times d}$ we will denote by $(V\cdot \nabla)(M)$ the $d\times d$ matrix with variable coefficients $V\cdot\nabla m_{ij}$.

\section{First and second variations under inner perturbations}\label{section.variations}
In this Section we compute the first and the second variations (with respect to perturbations with compactly supported vector fields) of the functional $\mathcal F$ from \eqref{e:Functional}. Both variations will be fundamental tools in the study of the blow-up limits of the state variables $u_\Omega$ and $v_\Omega$ on a domain $\Omega$, which is optimal for \eqref{e:intro-shape-opt-pb-in-main-teo}.

\subsection{First and second variation of the state function}\label{sub:first-and-second-variation-of-u}
In \cref{l:first-and-second-variation-along-a-flow} we compute the expansion of the state variable $u_\O$ (solution to \eqref{e:state-equation-open}) with respect to smooth perturbations of a set $\O$. We first prove \cref{l:abstact-first-order-expansion} and \cref{l:abstact-second-order-expansion}, where we compute the expansion of a one-parameter family of solutions to PDEs on the same domain $\Omega$.

\begin{oss}\label{rem:matrix-norms}
In what follows we will denote by $\R^{d\times d}$ the space of $d\times d$ square matrices with real coefficients. Given a real matrix $A=(a_{ij})_{ij}\in \R^{d\times d}$, we define its norm in the space $\R^{d\times d}$ as
$$\|A\|_{\R^{d\times d}}:=\bigg(\sum_{i=1}^d\sum_{j=1}^da_{ij}^2\bigg)^{\sfrac12},$$
and we notice that for every vector $V\in\R^d$, we have $|AV|\le \|A\|_{\R^{d\times d}}|V|$, where $|V|$ is the usual Euclidean norm of $V$.
Next, let $\Omega$ be a measurable set in $\R^d$. Given a matrix $A:\Omega\to \R^{d\times d}$ with variable coefficients $a_{ij}:\Omega\to\R$, we say that
$$A\in L^\infty(\Omega;\R^{d\times d}),$$
if $a_{ij}\in L^\infty(\Omega)$ for every $1\le i,j\le d$. We define the norm $\|\cdot\|_{L^\infty(\Omega;\R^{d\times d})}$ as
$$\|A\|_{L^\infty(\Omega;\R^{d\times d})}:=\big\|\|A\|_{\R^{d\times d}}\big\|_{L^\infty(\Omega)}=\bigg\|\sum_{i=1}^d\sum_{j=1}^da_{ij}^2\bigg\|_{L^\infty(\Omega)}^{\sfrac12}.$$
\end{oss}

\begin{lemma}[First order expansion of solutions to PDEs]\label{l:abstact-first-order-expansion}
Let $\Omega$ be a bounded open set in $\R^d$. Let the functions
\begin{align*}
f:&\R\to L^2(\Omega),&t\mapsto f_t,\\
A\colon &\R\to L^\infty(\Omega;\R^{d\times d}),&t\mapsto A_t,
\end{align*}
be such that:
\begin{enumerate}[\quad\rm(a)]
\item $A_t(x)$ is a symmetric matrix for every $(t,x)\in\R\times\Omega$ and there is a symmetric matrix\\ $\delta A\in L^\infty(\Omega;\R^{d\times d})$ such that
$$A_t=\text{\rm Id}+t(\delta A)+o(t)\quad\text{in }L^\infty(\Omega;\R^{d\times d}).$$
\item there is $\delta f\in L^2(\Omega)$ such that
$$f_t=f_0+t(\delta f)+o(t)\quad\text{in }L^2(\Omega).$$
\end{enumerate}
Then, for every $t$ small enough there is a unique solution $u_t$ to the problem
\begin{equation}\label{e:s21-abstract-equation-u-t}
-\text{\rm div}(A_t\nabla u_t)=f_t\quad\text{in }\Omega,\qquad u_t\in H^1_0(\Omega),
\end{equation}
and
$$u_t=u_0+t(\delta u)+o(t)\quad\text{in }H^1_0(\Omega),$$
where $\delta u$ is the unique weak solution in $H^1_0(\Omega)$ to the PDE
\begin{equation}\label{e:s3-abstract-equation-delta-u}
-\Delta (\delta u)-\text{\rm div}((\delta A)\nabla u_0)=\delta f\quad\text{in }\Omega,\qquad \delta u\in H^1_0(\Omega)\,.
\end{equation}
\end{lemma}
\begin{proof}
Clearly $u_0\in H^1_0(\Omega)$ is the solution to $-\Delta u_0=f_0$ in $\Omega$.
We set $w_t:=\frac1t(u_t-u_0)$. We will prove that $w_t$ converges to $\delta u$ strongly in $H^1_0(\Omega)$.\\ We notice that \eqref{e:s21-abstract-equation-u-t} can be written as
$$-\text{\rm div}((\text{\rm Id}+(A_t-\text{\rm Id}))\nabla (u_0+tw_t))=f_0+(f_t-f_0)\quad\text{in}\quad \Omega.$$
So, using the equation for $u_0$ and dividing by $t$, we get
\begin{equation}\label{e:s3-abstract-equation-w-t}
-\Delta w_t-\text{\rm div}\Big(\frac1t(A_t-\text{\rm Id})\nabla u_0\Big)-\text{\rm div}\Big((A_t-\text{\rm Id})\nabla w_t\Big)=\frac1t(f_t-f_0)\quad\text{in}\quad \Omega.
\end{equation}
If we fix $\eps>0$, we can choose $t$ small enough such that
$$\|A_t-\text{\rm Id}\|_{L^\infty(\Omega;\R^{d\times d})}\le \eps\ ,\quad \left\|\frac1{t}(A_t-\text{\rm Id})-\delta A\right\|_{L^\infty(\Omega;\R^{d\times d})}\le \eps\ ,\quad \left\|\frac1t(f_t-f_0)-\delta f\right\|_{L^2(\Omega)}\le \eps.$$
Thus, by testing \eqref{e:s3-abstract-equation-w-t} with $w_t$, we obtain
\begin{align*}
\int_{\Omega}|\nabla w_t|^2\,dx&=-\int_{\Omega}\nabla w_t\cdot \frac1t(A_t-\text{\rm Id})\nabla u_0\,dx-\int_\Omega\nabla w_t\cdot (A_t-\text{\rm Id})\nabla w_t\,dx+\int_{\Omega}\frac1t(f_t-f_0)w_t\,dx\\
&\le\Big(\eps+\|\delta A\|_{L^\infty(\Omega;\R^{d\times d})}\Big)\|\nabla w_t\|_{L^2}\|\nabla u_0\|_{L^2}+\eps\|\nabla w_t\|_{L^2}^2+\Big(\eps+\|\delta f\|_{L^2(\Omega)}\Big)\|w_t\|_{L^2}.
\end{align*}
Now, by the Poincar\'e inequality,
$$\|\varphi\|_{L^2(\Omega)}^2\le C_d|\Omega|^{\sfrac{2}{d}}\|\nabla \varphi\|_{L^2(\Omega)}^2\quad\text{for every}\quad \varphi\in H^1_0(\Omega)\,,$$
and the equation for $u_0$, we have that
$$\|\nabla u_0\|_{L^2(\Omega)}^2=\int_{\Omega}f_0u_0\,dx\le \|f_0\|_{L^2(\Omega)}\|u_0\|_{L^2(\Omega)}\le  \|f_0\|_{L^2(\Omega)}C_d|\Omega|^{\sfrac{1}{d}}\|\nabla u_0\|_{L^2(\Omega)},$$
which gives the bound
$$\|\nabla u_0\|_{L^2(\Omega)}\le C_d|\Omega|^{\sfrac{1}{d}}\|f_0\|_{L^2(\Omega)}\,.$$
Thus, we deduce that
\begin{align*}
\int_{\Omega}|\nabla w_t|^2\,dx
&\le C_d|\Omega|^{\sfrac{1}{d}}\|f_0\|_{L^2(\Omega)}\Big(\eps+\|\delta A\|_{L^\infty(\Omega;\R^{d\times d})}\Big)\|\nabla w_t\|_{L^2}\\
&\qquad+\eps\|\nabla w_t\|_{L^2}^2+\Big(\eps+\|\delta f\|_{L^2(\Omega)}\Big)C_d|\Omega|^{\sfrac{1}{d}}\|\nabla w_t\|_{L^2},
\end{align*}
and so, for $t$ (and $\eps<1$) small enough
$$\left(\int_{\Omega}|\nabla w_t|^2\,dx\right)^{1/2}\le \frac{C_d|\Omega|^{\sfrac{1}{d}}}{(1-\eps)}\Big(1+\|f_0\|_{L^2(\Omega)}+\|f_0\|_{L^2(\Omega)}\|\delta A\|_{L^\infty(\Omega;\R^{d\times d})}+\|\delta f\|_{L^2(\Omega)}\Big).$$
Thus, for every sequence $t_n\to0$, there is a subsequence for which $w_{t_n}$ converges as $n\to \infty$, strongly in $L^2(\Omega)$ and weakly in $H^1_0(\Omega)$, to some function $w_\infty$. Passing to the limit the equation \eqref{e:s3-abstract-equation-w-t} we get that $w_\infty$ is also a solution to \eqref{e:s3-abstract-equation-delta-u}. Thus $w_\infty=\delta u$. In particular, this implies that $w_t$ converges as $t\to0$, strongly in $L^2(\Omega)$ and weakly in $H^1_0(\Omega)$, to $\delta u$. Finally, in order to prove that the convergence is strong, we test again \eqref{e:s3-abstract-equation-w-t} with $w_t$:
\begin{align*}
\limsup_{t\to\infty}\int_{\Omega}|\nabla w_t|^2\,dx
&=\lim_{t\to0}\int_{\Omega}\nabla w_t\cdot \frac1t(A_t-Id)\nabla u_0\,dx+\lim_{t\to0}\int_{\Omega}\frac1t(f_t-f_0)w_t\,dx\\
&=\int_{\Omega}\nabla (\delta u)\cdot \delta A\nabla u_0\,dx+\int_{\Omega}\delta f\delta u\,dx=\int_{\Omega}|\nabla (\delta u)|^2\,dx\,.
\end{align*}
Combining this estimate with the lower semi-continuity of the $H^1$ norm, we get
$$\lim_{t\to\infty}\int_{\Omega}|\nabla w_t|^2\,dx=\int_{\Omega}|\nabla (\delta u)|^2\,dx\,$$
which implies that $w_t$ converges to $\delta u$ strongly in $H^1_0(\Omega)$.
\end{proof}	

\begin{lemma}[Second order expansion of solutions to PDEs]\label{l:abstact-second-order-expansion}
	Let $\Omega$ be a bounded open set in $\R^d$. Let the functions
	$$f:\R\to L^2(\Omega)\quad\text{and}\quad A:\R\to L^\infty(\Omega;\R^{d\times d})$$
	be such that:
	\begin{enumerate}[\quad\rm(a)]
		\item $A_t(x)$ is a symmetric matrix for every $(t,x)\in\R\times\Omega$ and there are symmetric matrices\\ $\delta A\in L^\infty(\Omega;\R^{d\times d})$ and $\delta^2A\in L^\infty(\Omega;\R^{d\times d})$ such that
		$$A_t=\text{\rm Id}+t\,\delta A+t^2\delta^2A+o(t^2)\quad\text{in }L^\infty(\Omega;\R^{d\times d})\,;$$
		\item there are $\delta f\in L^2(\Omega)$ and $\delta^2f\in L^2(\Omega)$ such that
		$$f_t=f_0+t\,\delta f+t^2\delta^2f+o(t^2)\quad\text{in }L^2(\Omega).$$
	\end{enumerate}
	Then, for every $t$ small enough there is a unique solution $u_t$ to the problem
	\begin{equation}\label{e:s3-abstract-equation-u-t}
	-\text{\rm div}(A_t\nabla u_t)=f_t\quad\text{in }\Omega\,,\qquad u_t\in H^1_0(\Omega),
	\end{equation}
	and
	$$u_t=u_0+t\,\delta u+t^2\delta^2u+o(t^2)\quad\text{in }H^1_0(\Omega),$$
	where $\delta u\in H^1_0(\Omega)$ is the solution to \eqref{e:s3-abstract-equation-delta-u} and where $\delta^2u\in H^1_0(\Omega)$ solves the PDE
	\begin{equation}\label{e:s3-abstract-equation-delta-2-u}
-\Delta (\delta^2 u)=\dive((\delta A)\nabla (\delta u))+\dive((\delta^2 A)\nabla u_0)+\delta^2 f\quad\text{in }\Omega\,,
\qquad\delta^2 u\in H^1_0(\O).
	\end{equation}
\end{lemma}
\begin{proof}
Let $w_t:=\frac1t(u_t-u_0)$ be as in the proof of \cref{l:abstact-first-order-expansion}. We set
$$v_t:=\frac1t(w_t-\delta u)\in H^1_0(\Omega).$$
We will prove that $v_t$ converges strongly in $H^1_0(\Omega)$ to $\delta^2u$. From the equation for $w_t$, we have
\begin{align*}
-\Delta (\delta u+tv_t)-\text{\rm div}\Big(\frac1t(A_t-\text{\rm Id})\nabla u_0\Big)-\text{\rm div}\Big((A_t-\text{\rm Id})\nabla(\delta u+tv_t)\Big)=\frac1t(f_t-f_0).
\end{align*}	
Thus, using the equation \eqref{e:s3-abstract-equation-delta-u} for $\delta u$ ($-\Delta (\delta u)-\text{\rm div}(\delta A\nabla u_0)=\delta f$), we get
\begin{align*}
-\Delta v_t-\text{\rm div}\Big(\frac1{t^2}(A_t-\text{\rm Id}-t\delta A)\nabla u_0\Big)-\text{\rm div}\Big(\frac1t(A_t-\text{\rm Id})\nabla(\delta u+tv_t)\Big)=\frac1{t^2}(f_t-f_0-t\delta f).
\end{align*}	
Now, reasoning as in \cref{l:abstact-first-order-expansion}, we get that the family $v_t$ is uniformly bounded in $H^1_0(\Omega)$ and converges as $t\to0$ strongly in $L^2(\Omega)$ and weakly in $H^1_0(\Omega)$ to the solution $\delta^2u$ of \eqref{e:s3-abstract-equation-delta-2-u}. In order to obtain the strong $H^1$ convergence, we compute
\begin{align*}
\limsup_{t\to\infty}\int_{\Omega}|\nabla v_t|^2\,dx
&=-\lim_{t\to0}\int_{\Omega}\nabla v_t\cdot \frac1{t^2}(A_t-\text{\rm Id}-t\delta A)\nabla u_0\,dx\\
&\qquad-\lim_{t\to0}\int_{\Omega}\nabla v_t\cdot \frac1{t}(A_t-\text{\rm Id})\nabla (\delta u)\,dx +\lim_{t\to0}\int_{\Omega}\frac1{t^2}(f_t-f_0-t\delta f)v_t\,dx\\
&=-\int_{\Omega}\nabla (\delta^2 u)\cdot (\delta^2 A)\nabla u_0\,dx-\int_{\Omega}\nabla (\delta^2u)\cdot (\delta A)\nabla (\delta u)\,dx+\int_{\Omega}(\delta^2 f)(\delta^2 u)\,dx\\
&=\int_{\Omega}|\nabla (\delta^2 u)|^2\,dx\,,
\end{align*}
which concludes the proof.
\end{proof}

\begin{oss}\label{r:expansion-det}	
	We recall that if $M=(m_{ij})_{1\le i,j\le d}\in \R^{d\times d}$ is a matrix with constant coefficients, then we have the following Taylor expansions in $\R^{d\times d}$ (with respect to the norm $\|\cdot\|_{\R^{d\times d}}$)
	$$(\text{\rm Id}+tM)^{-1}=\text{\rm Id}-tM+t^2M^2+o(t^2),$$
	$$\det(\text{\rm Id}+tM)=1+t\,\text{\rm tr}(M)+\frac{t^2}{2}\Big(\big(\text{\rm tr}(M)\big)^2-\text{\rm tr}(M^2)\Big)+o(t^2),$$
	where $\text{\rm Id}$ is the identity matrix in $\R^{d\times d}$ and $\text{\rm tr}(M)$ is the trace $\text{\rm tr}(M):=\sum_{i=1}^dm_{ii}$. As a consequence, if $M\in L^\infty(\Omega;\R^{d\times d})$ is a matrix with variable coefficients, then we have the expansions
	$$(\text{\rm Id}+tM)^{-1}=\text{\rm Id}-tM+t^2M^2+o(t^2)\quad\text{in}\quad L^\infty(\Omega;\R^{d\times d})\,,$$
	$$\det(\text{\rm Id}+tM)=1+t\,\text{\rm tr}(M)+\frac{t^2}{2}\Big(\big(\text{\rm tr}(M)\big)^2-\text{\rm tr}(M^2)\Big)+o(t^2)\quad\text{in}\quad L^\infty(\Omega;\R^{d\times d})\,.$$
	In particular, this implies that given two matrices $M,N\in L^\infty(\Omega;\R^{d\times d})$, we have:
	\begin{align*}
  \big(\text{\rm Id}+tM+t^2N+o(t^2)\big)^{-1}&=\text{\rm Id}-tM+t^2(M^2-N)+o(t^2)\,;\smallskip\\
	\det\big(\text{\rm Id}+tM+t^2N+o(t^2)\big)&=\det(\text{\rm Id}+tM)\det(\text{\rm Id}+t^2N)+o(t^2)\\
	&=1+t\,\text{\rm tr}(M)+\frac{t^2}{2}\Big(\big(\text{\rm tr}(M)\big)^2-\text{\rm tr}(M^2)+2\text{\rm tr}(N)\Big)+o(t^2)\,,
	\end{align*}
and so,
	\begin{align*}
	&\big(\text{\rm Id}+tM+t^2N+o(t^2)\big)^{-1}\big(\text{\rm Id}+tM+t^2N+o(t^2)\big)^{-T}\det\big(\text{\rm Id}+tM+t^2N+o(t^2)\big)\\
	&\quad=\text{\rm Id}+t\Big(-M-M^T+\text{\rm tr}(M)\text{\rm Id}\Big)\\
	&\quad\quad+t^2\Big(MM^T+M^2+(M^{T})^2-(N+N^T)-(M+M^T)\,\text{\rm tr}(M)\Big)\\
	&\quad\qquad+\frac{t^2}{2}\Big(\big(\text{\rm tr}(M)\big)^2-\text{\rm tr}(M^2)+2\text{\rm tr}(N)\Big)\text{\rm Id}+o(t^2).
		\end{align*}
\end{oss}	
	
\begin{prop}\label{l:first-and-second-variation-along-a-flow}
	Let $D$ be a bounded open set in $\R^d$ and let $f \in C^2(D)$ with $f\ge0$ in ${D}$. Given an open set $\Omega\subset D$ and a compactly supported vector field $\xi \in C^\infty_c(D;\R^d)$, we consider the associated flow $\Phi:\R\times\R^d\to\R^d$, determined by the family of ODEs (for every $x\in D$)
	\begin{equation}\label{e:flow-of-xi}
	\begin{cases}
	\partial_t\Phi_t(x)=\xi\big(\Phi_t(x)\big)\quad\text{for every}\quad t\in\R\\
	\Phi_0(x)=x.
	\end{cases}
	\end{equation}
	We consider	the family of open sets $\Omega_t := \Phi_t(\Omega)$ and the corresponding state variables $u_{\O_t}$ given by \eqref{e:state-equation-open}.
	We define $\delta u_\Omega$ and $\delta^2 u_\Omega$ to be the weak solutions in $H^1_0(\Omega)$ to the PDEs
	\be\label{e:deltau}
	-\Delta (\delta u_\O)=\dive((\delta A)\nabla u_\O)+\delta f\quad\text{in}\quad\Omega\,,
	\qquad\delta u_\O\in H^1_0(\O),
	\ee
	and
	\be \label{e:delta2u}
	-\Delta (\delta^2 u_\O)=\dive((\delta A)\nabla (\delta u_\O))+\dive((\delta^2 A)\nabla u_\O)+\delta^2 f\quad\text{in}\quad\Omega\,,
	\qquad\delta^2 u_\O\in H^1_0(\O),
	\ee
	where the matrices $\delta A\in L^\infty(D;\R^{d\times d})$ and $\delta^2A\in L^\infty(D;\R^{d\times d})$ are given by
	\begin{align}\label{e:first-and-second-variation-A}
	\begin{aligned}
	\delta A &:=-D\xi-\nabla\xi+(\dive\xi)\text{\rm Id}\,,\\
	\delta^2 A &:= (D\xi)\,(\nabla\xi)+\frac12\big(\nabla\xi\big)^2+\frac12(D\xi)^2-\frac12(\xi\cdot\nabla)\big[\nabla\xi+D\xi\big]\\
	&\qquad-\big(\nabla\xi+D\xi\big)\dive\xi+\text{\rm Id}\frac{(\dive\xi)^2+\xi\cdot\nabla(\dive\xi)}{2}\,,
	\end{aligned}
	\end{align}
	while the variations $\delta f\in L^2(D)$ and $\delta^2f\in L^2(D)$ of the right-hand side $f$ are:
	\begin{align}\label{e:first-and-second-variation-f}
	\begin{aligned}
	\delta f &:=\dive(f\xi)\,,\\
	\delta^2 f &:= \frac12 \xi \cdot (D^2f)\xi+\frac12\nabla f\cdot D\xi[\xi]+f\frac{(\dive\xi)^2+\xi\cdot\nabla[\text{\rm div}\,\xi]}{2}+(\nabla f \cdot \xi)\dive\xi \,.
	\end{aligned}
	\end{align}
	Then,
	\be\label{expansion.u}
	u_{\O_t}\circ \Phi_t = u_\O + t (\delta u_\O) + t^2 (\delta^2 u_\O) + o(t^2)\quad \mbox{in}\quad H^1_0(D)\,.
	\ee
\end{prop}
\begin{proof}
	We set $u_t:=u_{\O_t}\circ \Phi_t$. Then, $u_t\in H^1_0(\Omega)$ and $u_{\Omega_t}=u_t\circ \Phi_t^{-1}$.\\ Moreover, by a change of variables, we have that $u_t$ is satisfies the PDE
	\be\label{e:uniqu}
	-\dive\big(A_t\nabla u_t\big)=f_t\quad\text{in}\quad\Omega\,,\qquad u_t\in H^1_0(\Omega)\,,
	\ee
	where the matrix $A_t$ and the function $f_t$ are defined as
	\be\label{e:f-t-A-t}
	f_t:=f(\Phi_t)|\text{\rm det}(D\Phi_t)|\qquad\text{and}\qquad
	A_t:=(D\Phi_t)^{-1}(D\Phi_t)^{-T}|\text{\rm det}(D\Phi_t)|,
	\ee
	where for a vector field $F:\R^d\to\R^d$ with components $F_k$, $k=1,\dots,d$, we denote by
	$DF$ the matrix with rows $DF_k=(\nabla F_k)^T$. We next compute the second order Taylor expansion of $D\Phi_t$ in $t=0$. By differentiating the equation for the flow $\Phi_t$, we get
$$\begin{cases}
\partial_t(D\Phi_t)=D\xi(\Phi_t)D\Phi_t\quad\text{for every}\quad t\in\R\,,\\
D\Phi_0=\text{\rm Id}.
\end{cases}$$
Then, taking another derivative in $t$, we get
\begin{align*}
\partial_{tt}(D\Phi_t)=\partial_t\big[D\xi(\Phi_t)D\Phi_t\big]
&=\partial_t\big[D\xi(\Phi_t)\big]D\Phi_t+D\xi(\Phi_t)\partial_t\big[D\Phi_t\big]\\
&=\big(\partial_t\Phi_t\cdot\nabla\big)\big[D\xi\big](\Phi_t)D\Phi_t+D\xi(\Phi_t)D\xi(\Phi_t)D\Phi_t\\
&=\big(\xi(\Phi_t)\cdot\nabla\big)\big[D\xi\big](\Phi_t)D\Phi_t+D\xi(\Phi_t)D\xi(\Phi_t)D\Phi_t\,,
\end{align*}
where for a vector field $F$ and a matrix $M=(m_{ij})_{ij}$, we use the notation $(F\cdot \nabla)[M]$ for the matrix with coefficients $F\cdot \nabla m_{ij}$. Finally, taking $t=0$, we get
$$\frac{\partial^2}{\partial t^2}\Big|_{t=0}(D\Phi_t)=(\xi\cdot\nabla)[D\xi]+(D\xi)^2\,,$$
and the Taylor expansion
$$D\Phi_t=\text{\rm Id}+tD\xi+\frac{t^2}{2}\Big((\xi\cdot\nabla)[D\xi]+(D\xi)^2\Big)+o(t^2).$$
By the expansions from \cref{r:expansion-det}, we get
	\be\label{e:expansion}
	\begin{aligned}
		A_t&=\text{\rm Id}+ t (\delta A) + t^2 (\delta^2 A) + o(t^2)\quad\text{in }L^\infty(D;\R^{d\times d})\,,\\
		f_t&= f + t (\delta f) + t^2 (\delta^2 f) + o(t^2)\quad\text{in }L^2(D)\,,
	\end{aligned}
	\ee
	where $\delta A$, $\delta^2A$, $\delta f$, $\delta^2f$ are given by \eqref{e:first-and-second-variation-A} and \eqref{e:first-and-second-variation-f}. Thus, the claim follows from \cref{l:abstact-first-order-expansion} and \cref{l:abstact-second-order-expansion}.
\end{proof}

\subsection{First and second variation of $\mathcal F$}\label{sub:first-and-second-variation-of-F}
In the next lemma, we compute the first derivative of the functional $\mathcal{F}$ along inner variations with compact support in $D$.

\begin{lemma}[First variation of $\mathcal{F}$ along inner perturbations]\label{l:firstv}
Let $D$ be a bounded open set in $\R^d$ and let $f,g,Q\in C^1(D)$. 	
Let $\O\subseteq D$ be open and $\xi\in C^\infty_c(D;\R^d)$ be a vector field with compact support in $D$.
Let $\Phi_t$ be the flow of the vector field $\xi$ defined by \eqref{e:flow-of-xi} and set $\Omega_t := \Phi_t(\Omega)$. Then
\begin{align}\label{e:first-derivative-F-along-vector-field}
\begin{aligned}
\frac{\partial}{\partial t}\bigg|_{t=0}\mathcal{F}({\Omega_t},D)=&\, \int_\O \left(\nabla u_\Omega \cdot \nabla v_\Omega + Q \right)\dive\xi  + \xi\cdot\nabla Q- \nabla u_\O \cdot ((\nabla\xi) + (D\xi))\nabla v_\O \, dx\\
&\,-\int_\Omega u_\O \dive(g\xi) + v_\O \dive(f\xi)\, dx.
\end{aligned}
\end{align}
Moreover, if $\partial\Omega$ is $C^2$-regular in a neighborhood of the support of $\xi$, then
\be\label{e:statiosmooth}
\frac{\partial}{\partial t}\bigg|_{t=0}\mathcal{F}({\Omega_t},D) = \int_{\partial \Omega} (\nu \cdot \xi)\big(Q-|\nabla u_\O ||\nabla v_\O|\big)\, d\HH^{d-1},
\ee
where $\nu$ is the outer unit normal to $\partial \Omega$.
\end{lemma}

\begin{proof}
By applying \cref{l:first-and-second-variation-along-a-flow} to $u_t=u_{\O_t}\circ\Phi_t$ and $v_t=v_{\O_t}\circ\Phi_t,$
we get that
$$u_t=u_\O+t(\delta u_\O)+o(t)\quad\text{and}\quad v_t=v_\O+t(\delta v_\O)+o(t)\quad\mbox{in}\quad H^1_0(D)\,,$$
where $\delta u_\Omega$ and $\delta v_\Omega$ are the solutions to
\begin{align}\label{equationdeltauv}
\begin{aligned}
-\Delta (\delta u_\O)&=\dive((\delta A)\nabla u_{\O})+\delta f\quad\text{in}\quad\Omega\ ,\qquad
\delta u_\O\in H^1_0(\O)\,,\\
-\Delta (\delta v_\O)&=\dive((\delta A)\nabla v_{\O})+\delta g
\quad\text{in}\quad\Omega\ ,\qquad \delta v_\O\in H^1_0(\O)\,,
\end{aligned}
\end{align}
with $\delta f:=\text{\rm div}(f\xi)$, $\delta g:=\text{\rm div}(g\xi)$, and $\delta A$ as in \cref{l:first-and-second-variation-along-a-flow}.
Therefore, setting
$$f_t:=f(\Phi_t)|\text{\rm det}(D\Phi_t)|\ ,\quad g_t:=g(\Phi_t)|\text{\rm det}(D\Phi_t)|\ ,$$
$$Q_t:=Q(\Phi_t)|\text{\rm det}(D\Phi_t)| \quad\text{and}\quad
A_t:=(D\Phi_t)^{-1}(D\Phi_t)^{-T}|\text{\rm det}(D\Phi_t)|\,,$$
we get
\begin{align*}
\mathcal{F}({\Omega_t},D) =&\,\int_{\Omega_t} \Big(\nabla u_{\O_t} \cdot \nabla v_{\O_t} \, - g u_{\O_t} - f v_{\O_t} +Q\Big)\,dy \\
=&\, \int_{\Omega}\Big(\nabla u_{t}\cdot A_t \nabla v_t -  g_t u_{t} -f_t v_t +Q_t\Big)\, dx \\
=&\, \mathcal{F}(\O,D) + t \int_{D}\Big(\nabla (\delta u_\O) \cdot \nabla v_{\O} + \nabla u_{\O} \cdot \nabla (\delta v_\O)  - g (\delta u_\O) - f(\delta v_\O)\Big)\, dx\\
&\quad+t\int_{\Omega}\Big(\nabla u_\O\cdot(\delta A)\nabla v_\O-u_\O(\delta g)-v_\O(\delta f)+(Q\dive\xi + \nabla Q \cdot \xi)\Big)\,dx+o(t)\\
=&\, \mathcal{F}({\O},D) + t \int_{\O} \Big(\nabla u_{\O}\cdot (\delta A) \nabla v_{\O}  -u_{\O} (\delta g) -v_{\O} (\delta f) +\dive(Q\xi)\Big)\,dx + o(t)
\end{align*}
where in the first equality we applied the change of variables $y=\Phi_t(x)$ and in the last one we use the equations $-\Delta u_\O = f$ and $-\Delta v_\O =g$ in $\O$. Substituting with the expression for $\delta A$ from \eqref{e:first-and-second-variation-A}, we obtain \eqref{e:first-derivative-F-along-vector-field}.

Suppose now that $\partial \Omega$ is $C^2$-smooth. Since in $\Omega$ we have the identity
\begin{align*}
&\left(\nabla u_\O \cdot \nabla v_\O\right)\dive\xi  - \nabla u_\O \cdot ((\nabla\xi) + (D\xi))\nabla v_\O\\
 &\qquad = \dive\Big(\xi(\nabla u_\O \cdot\nabla v_\O)-(\nabla u_\O \cdot \xi)\nabla v_\O - (\nabla v_\O \cdot \xi)\nabla u_\O\Big) +(\nabla v_\O \cdot \xi )\Delta u_\O + (\nabla u \cdot \xi )\Delta v_\O,
\end{align*}
by integrating by parts we get
\begin{align}\label{e:first-variation-F-along-vector-field}
\begin{aligned}
\delta \mathcal{F}(\O, D)[\xi] =&\int_\O \dive\Big(\xi\big((\nabla u_\O \cdot\nabla v_\O)+Q\big)-(\nabla u_\O \cdot \xi)\nabla v_\O - (\nabla v_\O \cdot \xi)\nabla u_\O\Big)\,dx\\
&\,-\int_\O \Big((\nabla u_\O\cdot \xi)g + u_\O \dive(g\xi) + (\nabla v_\O\cdot \xi)f +v_\O \dive(f\xi)\Big)\, dx\\
=&\int_{\partial\O} \Big((\nu \cdot \xi)((\nabla u \cdot\nabla v)+Q)-(\nabla u \cdot \xi)(\nu\cdot\nabla v) - (\nabla v \cdot \xi)(\nu\cdot\nabla u)\Big)\,d\HH^{d-1}.
\end{aligned}
\end{align}
Since $u_\O$ and $v_\O$ are positive in $\O$ and vanish on $\partial\Omega$, we have that
$$\nabla u_\O = -\nu |\nabla u_\O|\quad\text{and}\quad\nabla v_\O = -\nu |\nabla v_\O|\quad\text{on}\quad \partial\Omega\,,$$
and so
\begin{align*}
\delta \mathcal{F}(\O, D)[\xi] =&\int_{\partial \Omega} (\nu \cdot \xi)(|\nabla u_\O ||\nabla v_\O|+Q)-|\nabla u_\O| (\nu\cdot \xi)|\nabla v_\O| - |\nabla v_\O|(\nu\cdot \xi)|\nabla u_\O|  \, d\HH^{d-1}\\
=&\int_{\partial \Omega} (\nu \cdot \xi)(Q-|\nabla u_\O ||\nabla v_\O|)\, d\HH^{d-1},
\end{align*}
which concludes the proof.
\end{proof}

\begin{oss}[First variation and stationary domains]\label{oss:optimal-domains-are-stationary}
Given a bounded open set $D$, functions $f$, $g$ and $Q$ on $D$, and the functional $\mathcal F$ defined in \eqref{e:Functional}, we will use the notation $\delta \mathcal{F}(\Omega,D)[\xi]$ for the first variation of $\mathcal{F}$ at $\Omega$ along a smooth compactly suppported vector field $\xi \in C^\infty_c(D;\R^d)$. Precisely, we set
\begin{align}\label{e:first-variation-of-F}
\begin{aligned}
\delta \mathcal{F}(\Omega,D)[\xi]:=&\, \int_\O \left(\nabla u_\Omega \cdot \nabla v_\Omega + Q \right)\dive\xi  + \xi\cdot\nabla Q- \nabla u_\O \cdot ((\nabla\xi) + (D\xi))\nabla v_\O \, dx\\
&\,-\int_\Omega u_\O \dive(g\xi) + v_\O \dive(f\xi)\, dx.
\end{aligned}
\end{align}
We will say that an open set $\Omega\subset D$ is stationary (or a critical point) for $\mathcal F$ in $D$ if
$$\delta \mathcal{F}(\O,D)[\xi]=0\quad\text{for every}\quad \xi \in C^\infty_c(D;\R^d).$$
By \eqref{e:statiosmooth}, if $\Omega$ is stationary in $D$ and the boundary $\partial \O\cap D$ is smooth, then
\begin{equation}\label{e:stationarity-classical}
|\nabla u_\O ||\nabla v_\O|=Q\quad\text{on}\quad \partial \O\cap D\,.
\end{equation}
Finally, we notice that, by \cref{l:first-and-second-variation-along-a-flow}, any minimizer of \eqref{e:intro-shape-opt-pb-in-main-teo} is stationary for $\mathcal F$.
\end{oss}

\begin{prop}[Second variation of $\mathcal{F}$ along inner perturbations]\label{l:secondstv}
Let $D\subset\R^d$ be a bounded open set in $\R^d$  and let $f,g,Q\in C^2(D)$. 	Let $\Omega \subseteq D$ be an open set and $\xi \in C^\infty_c(D;\R^d)$ be a smooth vector field with compact support. Let $\Phi_t$ be the flow of the vector field $\xi$ defined by \eqref{e:flow-of-xi} and set $\Omega_t := \Phi_t(\Omega)$. Then
\begin{align}\label{e:second-derivative-F-along-vector-field}
\begin{aligned}
\frac12\frac{\partial^2}{\partial t^2}\bigg|_{t=0}\mathcal{F}({\Omega_t},D)
&=\,
\int_\O\Big(\nabla u_\Omega\cdot(\delta^2 A)\nabla v_\Omega-\nabla (\delta u_\O)\cdot \nabla (\delta v_\O)\\
&\qquad\qquad\qquad\qquad\qquad - (\delta^2 f) v_\Omega - (\delta^2 g) u_\Omega+\delta^2Q\Big) dx\,,
\end{aligned}
\end{align}
where $\delta^2 A,\delta^2 f$, $\delta^2 g$, $\delta u_\O$ and $\delta v_\O$ are the ones defined in  \cref{l:first-and-second-variation-along-a-flow} and where
$$\delta^2Q:=(\xi\cdot\nabla Q)\text{\rm div}\xi+\frac12\xi \cdot D^2 Q \xi
+\frac12Q(\text{\rm div}\xi)^2+\frac12Q\xi\cdot\nabla(\text{\rm div}\xi).$$
\end{prop}
\begin{proof}
By applying \cref{l:first-and-second-variation-along-a-flow} to $u_t:=u_{\O_t}\circ \Phi_t$ and $v_t:=v_{\O_t}\circ \Phi_t$, we get that
$$
u_t = u_{\O} + t (\delta u_\O) + t^2 (\delta^2 u_\O)+o(t^2), \quad v_t = v_{\O} + t (\delta v_\O) + t^2 (\delta^2 v_\O)+o(t^2) \quad\mbox{in}\quad H^1_0(\Omega).
$$
with $\delta u_\Omega ,\delta v_\Omega $ satisfying \eqref{equationdeltauv} and $\delta^2 u_\O ,\delta^2 v_\O \in H^1_0(\Omega)$ such that
\begin{align*}
-\Delta (\delta^2 u_\O) &=\dive((\delta A)\nabla (\delta u_\O))+\dive((\delta^2 A)\nabla u_\O)+\delta^2 f\quad\text{in}\quad\Omega\ ,\qquad \delta^2u_\Omega\in H^1_0(\Omega)\,,\\
-\Delta (\delta^2 v_\O) &=\dive((\delta A)\nabla (\delta v_\O))+\dive((\delta^2 A)\nabla v_\O)+\delta^2 g\quad\text{in}\quad\Omega\ ,\qquad \delta^2v_\Omega\in H^1_0(\Omega)\,.
\end{align*}
Thus, by computing the second order (in $t$) Taylor expansion of
\begin{align*}
\mathcal{F}({\Omega_t},D) =\, \int_{\Omega_t}\Big(\nabla u_{t}\cdot A_t \nabla v_t-g_t u_{t} -f_t v_t +Q_t\Big)\, dx\,,
\end{align*}
we get that
\begin{align*}
\frac12	\frac{\partial^2}{\partial t^2}\bigg|_{t=0}\mathcal{F}({\Omega_t},D) =&\, \int_{\O} \nabla (\delta^2 u_\O )\cdot \nabla v_\O+\nabla u_\O\cdot(\delta^2 A)\nabla v_\O+\nabla u_\O\cdot \nabla (\delta^2 v_\O)\, dx\\
		&+\int_{\O}\nabla (\delta u_\O)\cdot \nabla (\delta v_\O)+\nabla (\delta u_\O)\cdot (\delta A)\nabla v_\O+\nabla u_\O\cdot (\delta A)\nabla (\delta v_\O)\,dx\\
		&-\int_{\O}f(\delta^2 v)+(\delta f)(\delta v_\O)+(\delta^2 f)v_\O+g(\delta^2 u_\O)+(\delta g)(\delta u_\O)+(\delta^2 g)u_\O\,dx\\
		&+\int_\Omega\delta^2Q\,dx.
	\end{align*}
	Thus, we only have to show that we can write the above expression as in \eqref{e:second-derivative-F-along-vector-field}.
By using $\delta v_\Omega$ as a test function in the equation for $\delta u_\O$ and vice versa, we obtain
$$\int_{\O}\nabla (\delta v_\O)\cdot (\delta A)\nabla u_\O\,dx-\int_\Omega(\delta v_\O) (\delta f)\,dx=-\int_{\O} \nabla (\delta u_\O)\cdot \nabla (\delta v_\O)\,dx$$
$$\int_{\O}\nabla (\delta u_\O)\cdot (\delta A)\nabla v_\O\,dx-\int_\Omega(\delta u_\O )(\delta g)\,dx=-\int_{\O} \nabla (\delta u_\O)\cdot \nabla (\delta v_\O)\,dx.$$
Then, by testing the equations for $u_\Omega$ and $v_\Omega$ respectively with $\delta^2v_\Omega$ and $\delta^2u_\Omega$, we get
$$\int_{\O} \nabla (\delta^2 u_\O )\cdot \nabla v_\O\,dx=\int_{\O}f\,\delta^2 v\,dx\qquad\text{and}\qquad \int_{\O} \nabla (\delta^2 v_\O) \cdot \nabla u_\O\,dx=\int_{\O}g\,\delta^2 u_\Omega\,dx.$$
Using this identities in the expression of the second derivative, we get precisely \eqref{e:second-derivative-F-along-vector-field}.
\end{proof}

\begin{oss}[Second variation and stable critical domains]
	Given a bounded open set $D$, functions $f$, $g$ and $Q$ on $D$, and the functional $\mathcal F$ from \eqref{e:Functional}, we will indicate by $\delta^2 \mathcal{F}(\Omega,D)[\xi]$ the second variation of $\mathcal{F}$ at $\Omega$ along a smooth compactly supported vector field $\xi \in C^\infty_c(D;\R^d)$. Precisely, we set
	\begin{align}\label{e:second-variation-F-along-vector-field}
	\begin{aligned}
	\delta^2 \mathcal{F}(\Omega,D)[\xi]
	&:=
	\int_\O\nabla u_\Omega\cdot(\delta^2 A)\nabla v_\Omega-\nabla (\delta u_\O)\cdot \nabla (\delta v_\O) - (\delta^2 f) v_\Omega - (\delta^2 g) u_\Omega+\delta^2Q\,dx\,.
	\end{aligned}
	\end{align}
	In particular, we notice that for every open set $\Omega$, we have
	$$\mathcal{F}({\Omega_t},D) =\mathcal{F}(\O,D)+t\delta\mathcal{F}(\O,D)[\xi]+t^2 \delta^2\mathcal{F}(\Omega,D)[\xi]+o(t^2)\,,$$
	where $\Omega_t=\Phi_t(\Omega)$ (with $\Phi_t$ the flow associated to the vector field $\xi$) and $\delta\mathcal{F}(\O,D)[\xi]$ is the first variation \eqref{e:first-variation-F-along-vector-field}.\smallskip
	
	We will say that an open set $\Omega\subset D$ is a\emph{ stable critical point} for $\mathcal F$ in $D$ if
	$$\delta \mathcal{F}(\O,D)[\xi]=0\quad\text{and}\quad \delta^2 \mathcal{F}(\O,D)[\xi]\ge 0\quad\text{for every}\quad \xi \in C^\infty_c(D;\R^d).$$
	By \cref{l:first-and-second-variation-along-a-flow} and \cref{l:secondstv}, any minimizer of \eqref{e:intro-shape-opt-pb-in-main-teo} is a stable critical point for $\mathcal F$.
\end{oss}

\section{Lipschitz regularity and non-degeneracy of the state functions}\label{section.lip}
In this Section we study the regularity of the state functions $u_\Omega$ and $v_\Omega$ on an optimal domain $\Omega$, as well as their behavior close to the free boundary $\partial\Omega$. Consequently, we prove that the set $\Omega$ satisfies some density estimates.
\subsection{Assumptions on $D$, $f$, $g$, and $Q$}\label{sub:assumptions-D-f-g-Q}
Throughout this Subsection we will assume that:
\begin{itemize}
\item $D$ is a open subset of $\R^d$, with $d\ge 2$;
\item $f,g\in L^\infty(D)$ are two functions such that
$$\|f\|_{L^\infty}+\|g\|_{L^\infty}\le M\quad\text{and}\quad 0\le C_1g\le f\le C_2g\quad\text{on}\quad D\,,$$
for some positive constants $M,C_1,C_2>0$.
\item $Q\in L^\infty(D)$ is such that
$$0<c_Q\le Q\le C_Q\quad\text{on}\quad D\,.$$
\end{itemize}

\subsection{Inwards and outwards minimality conditions}\label{subsec:almost}

We will use the following notation. Given a set $A\subset\R^d$, and functions $f\in L^2(A)$ and $\varphi\in H^1(A)$, we set
\begin{equation}\label{e:definition-of-E-f}
E_f(\varphi,A):=\frac12\int_A|\nabla \varphi|^2\,dx-\int_Af(x)\varphi\,dx\,.
\end{equation}
\begin{prop}\label{prop.almost}
Let $D\subset \R^d$ and $f,g,Q\in L^\infty(D)$ be as in \cref{sub:assumptions-D-f-g-Q} and let $\Omega\subset D$ be an open set that minimizes \eqref{e:intro-shape-opt-pb-in-main-teo} in $D$. Then, the solution $u_\Omega$ to \eqref{e:state-equation-open} has  the following properties.
\begin{enumerate}[\rm (i)]
\item {\rm Outwards minimality.} For every open set $\widetilde \Omega\subset D$ such that $\O\subset\widetilde\Omega$ we have
$$E_f(u,D) + \frac{C_2C_Q}{2}|\O|\le
E_f(\phi,D) + \frac{C_2C_Q}{2}|\widetilde\Omega|\qquad\mbox{for every}\qquad\phi \in H^1_0(\widetilde \Omega).$$
In particular, for every $B_r(x_0)\subseteq D$, we have
\be\label{e:local.outw}
E_f(u,B_r(x_0))\le
E_f(\phi,B_r(x_0)) + \frac{C_2C_Q}{2}\omega_d r^d,
\ee
for every $\phi \in H^1(B_r(x_0))$ such that $\phi-u_\O\in H^1_0(B_r(x_0))$.\smallskip
\item  {\rm Inwards minimality.} For every open set $\omega\subset \Omega$ we have
\begin{equation}\label{e:subsolution}
E_f(u,D) + \frac{C_1c_Q}{2}|\Omega|\le
E_f(\phi,D) + \frac{C_1c_Q}{2}|\omega|\qquad\mbox{for every}\qquad\phi \in H^1_0(\omega)\,.
\end{equation}
\end{enumerate}
\end{prop}
\begin{proof}
Suppose that the open set $\widetilde \Omega\subset D$ contains $\Omega$. Then, the optimality of $\Omega$ implies that
$$\int_{\Omega}\Big(-g(x)u_{\Omega}+Q(x)\Big)\,dx\le \int_{\widetilde\Omega}\Big(-g(x)u_{\widetilde\Omega}+Q(x)\Big)\,dx\,,$$
which can be written as
$$\int_{D}g(x)\big(u_{\widetilde \Omega}-u_\Omega\big)\,dx\le \int_{\widetilde \Omega\setminus\Omega}Q(x)\,dx \,.$$
Now, the positivity of $f$ implies that $u_{\widetilde \Omega}\ge u_\Omega$ on $D$, so we get
$$\frac1{C_2}\int_{D}f(x)\big(u_{\widetilde \Omega}-u_\Omega\big)\,dx\le \int_{D}g(x)\big(u_{\widetilde \Omega}-u_\Omega\big)\,dx\le \int_{\widetilde \Omega\setminus\Omega}Q(x)\,dx\le C_Q|\widetilde\Omega\setminus\Omega|\,,$$
which after rearranging the terms gives
$$-\frac12\int_{\Omega}f(x)u_{\Omega}+\frac{C_2C_Q}{2}|\Omega|\le -\frac12\int_{\widetilde\Omega}f(x)u_{\widetilde\Omega}+\frac{C_2C_Q}{2}|\widetilde\Omega|\,,$$
which after an integration by parts on $\widetilde \Omega$ and $\Omega$, reads as
$$E_f(u_\Omega,D)+\frac{C_2C_Q}{2}|\Omega|\le E_f(u_{\widetilde\Omega},D)+\frac{C_2C_Q}{2}|\widetilde\Omega|\,.$$
Finally, since $u_{\widetilde\Omega}$ minimizes the energy $E_f(\cdot,D )$ among all functions in $H^1_0(\widetilde \Omega)$, we obtain (i).

The proof of (ii) is similar. Let $\omega\subset \Omega$ and let $u_\omega$ be the associated state function. Then, using the optimality of $\Omega$ and the bounds on $f,g$ and $Q$, we get
\begin{align*}
c_Q\big(|\Omega|-|\omega|\big)&\le \int_{\Omega\setminus\omega}Q(x)\,dx\le \int_D g\big(u_\Omega-u_\omega\big)\,dx\\
&\le \frac1{C_1}\int_D f\big(u_\Omega-u_\omega\big)\,dx=\frac{2}{C_1}\Big(E_f(u_\omega,D)-E_f(u_\Omega,D)\Big),
\end{align*}
which implies (ii) since $u_\omega$ minimizes $E_f(\cdot,D )$ in $H^1_0(\omega)$.
\end{proof}

\begin{oss}\label{rem:almostv}
The state variable $v_\O$ satisfies analogous inwards/outwards minimality conditions for the functional $E_g(\cdot,D)$, where the constants $C_1$ and $C_2$ are replaced by $1/C_2$ and $ 1/C_1$.
\end{oss}


\subsection{Lipschitz continuity and non-degeneracy}

As a consequence of Proposition \ref{prop.almost} we obtain the Lipschitz continuity and the non-degeneracy of $u_\Omega$ (and of $v_\Omega$).


\begin{coro}\label{lemm.lip}
Let $D\subset \R^d$ and $f,g,Q\in L^\infty(D)$ be as in \cref{sub:assumptions-D-f-g-Q}. If $\O\subset D$ is optimal for~\eqref{e:intro-shape-opt-pb-in-main-teo}, then the state functions $u_\Omega$ and $v_\Omega$ are locally Lipschitz in $D$ with Lipschitz constants depending on $d$, $C_1$, $C_2$, $M$ and $C_Q$.
\end{coro}
\begin{proof}
By Proposition \ref{prop.almost} (i), $u_\Omega$ satisfies \eqref{e:local.outw} for any $B_r(x_0)\subset D$, so $u_\Omega$ is an almost-minimizer in the sense of \cite[Definition 3.1]{bmpv}. Thus, by \cite[Theorem 3.3]{bmpv}, $u_\Omega$ is locally Lipschitz continuous in $D$ with Lipschitz constant depending on $d$ and the bounds from above on $\norm{f}{L^\infty}$ and $C_2C_Q$.
\end{proof}

\begin{lemma}\label{lemm.nondeg}
Let $D\subset \R^d$ and $f,g,Q\in L^\infty(D)$ be as in \cref{sub:assumptions-D-f-g-Q} and let $\O\subset D$ be an optimal set for~\eqref{e:intro-shape-opt-pb-in-main-teo} and $u_\Omega$ be the associated state function. Then, there are constants $C_0, r_0>0$, depending on $d$, $C_1$, $c_Q$ and $M$, such that the following implication holds
$$
\Big(\,\norm{u_\Omega}{L^\infty(B_{r}(x_0))}\leq C_0 r\,\Big) \Rightarrow \Big(\, u_\Omega \equiv 0 \quad\mbox{in }B_{\sfrac{r}2}(x_0)\,\Big),
$$
 for every $x_0 \in \overline{\O}\cap D$ and every $r \in (0,r_0]$.
In other words, if $x_0\in \overline \Omega$ and $B_r(x_0)\subset D$, then
\be\label{e:nondeg}
\sup_{B_r(x_0)} u_\Omega \geq C_0r.
\ee
\end{lemma}
\begin{proof}
By Proposition \ref{prop.almost} (ii), we have that for every $\omega\subset\Omega$
$$E_f(u_\Omega,D)+\frac{c_QC_1}{2}|\Omega|\le E_f(u_\omega,D)+\frac{c_QC_1}{2}|\omega|.$$
Thus, the claim follows by \cite[Lemma 4.4]{altcaf}, \cite[Lemma 3.3]{bucve}, or \cite[Lemma 2.8]{GustShahg}.
\end{proof}

\subsection{Density estimates on the boundary of $\Omega$}
An a consequence of the Lipschitz continuity and the non-degeneracy of $u_\Omega$ and $v_\Omega$, we obtain density estimates for the optimal set $\O$.

\begin{prop}\label{p:density}
Let $D\subset \R^d$ and $f,g,Q\in L^\infty(D)$ be as in \cref{sub:assumptions-D-f-g-Q}. Then, there are $\eps_0, r_0>0$ (depending on $C_1$, $C_2$, $M$, $d$, $c_Q$, $C_Q$) such that for every set $\O\subset D$ optimal for~\eqref{e:intro-shape-opt-pb-in-main-teo}
\be\label{e:density}
\eps_0|B_r|\leq |B_r(x_0)\cap \O| \leq (1-\eps_0)|B_r|.
\ee
for every ball $B_r(x_0)\subset D$ of radius $r\le r_0$ centered on $\partial\Omega$.
\end{prop}

\begin{proof}
Assume that $x_0=0\in \partial \Omega $. The lower estimate is an immediate consequence of the Lipschitz continuity (Lemma~\ref{lemm.lip}) and the non-degeneracy (Lemma~\ref{lemm.lip}) of $u_\Omega$.
The upper bound can be obtained as in \cite{altcaf}. Precisely, consider the solution $h$ to
$$-\Delta h =M\quad\mbox{in}\quad B_{r}\ ,\qquad h = u_\Omega\quad \mbox{on}\quad  D \setminus B_r\,.$$
Since $\Delta u_\Omega+f\ge 0$ in $\R^d$, we get that $-\Delta (h-u_\Omega)\ge M-f\geq 0$
in $B_r$. In particular, we have that $u_\Omega \leq h$ and $\{u_\Omega>0\}\subseteq \{h>0\}$ in $B_r$. Thus, testing the optimality \eqref{e:local.outw} of $u_\Omega$ with $h$,
\begin{align*}
\frac{C_2C_Q}{2}|B_r \cap \{u_\Omega=0\}| &\geq
E_f(u_\Omega,B_r) - E_f(h,B_r) \\
&=
\frac12 \int_{B_r}|\nabla (u_\Omega-h)|^2\,dx + \int_{B_r} \Big(\nabla h\cdot \nabla (u_\Omega-h) - f(u_\Omega-h)\Big)\, dx\\
&\geq
\frac12 \int_{B_r}|\nabla (u_\Omega-h)|^2\, dx.
\end{align*}
By the Poincar\'e and Cauchy–Schwarz inequalities, we have
$$
\int_{B_r}|\nabla (u_\Omega-h)|^2\, dx \geq \frac{C_d}{|B_r|}\left(\frac1r \int_{B_r}(h-u_\Omega)\,dx\right)^2,
$$
so in order to prove the upper bound in \eqref{e:density}, we only need to show that
$\ds\frac1{r^{d+1}} \int_{B_{r}}(h-u_\Omega)\,dx$
is bounded from below by a positive constant. Notice that, by the non-degeneracy of $u_\Omega$, we have
$$\widetilde{C}r \leq \sup_{B_{r/2}} u_\Omega \leq \sup_{B_{r/2}}h\,.$$
On the other hand, since $h(x)+\frac{M}{2d}|x|^2$ is harmonic in $B_R$, the Harnack inequality in $B_r$ implies
$$\widetilde{C}r \le \sup_{B_{r/2}}h\le C_d\big(h(x)+Mr^2\big)\quad\text{for every}\quad x\in B_{\sfrac{r}2}\,.$$
Thus, by taking $r_0$ such that $2C_dr_0M\le \widetilde C$, we get that $h \geq C_d\widetilde{C} r = \overline{C}r$  in $B_{\sfrac{r}2}.$
On the other hand, if $L$ is the Lipschitz constant of $u_\Omega$, then for any $\eps\in(0,1)$, $u_\Omega\leq L \eps r$ in $B_{\eps r}$. Then
$$
\int_{B_r}(h-u_\Omega)\, dx \geq \int_{B_{\eps r}} (h-u_\Omega)\, dx \geq (\bar C r-L\eps r)|B_{\eps r}|,
$$
which concludes the proof after choosing $\eps\le \sfrac12$ small enough.
\end{proof}

\subsection{An estimate on the level sets of $u_\Omega$}
We conclude the Section with this auxiliary result that will play a crucial role in Section \ref{sub.homo} in the proof of the existence of homogeneous blow-up limits.

\begin{lemma}\label{l:conditionf}
Let $D\subset \R^d$ and $f,g,Q\in L^\infty(D)$ be as in \cref{sub:assumptions-D-f-g-Q}; let $\O$ be a solution to \eqref{e:intro-shape-opt-pb-in-main-teo} and let $u_\Omega$ be the associated state function. Then, there are constants $C>0$ and $r_0>0$, depending only on $d,M,C_1,C_2,c_Q,C_Q$, such that
\be\label{e:conditionf}
|\{0<u<rt\}\cap B_r(x_0)|\leq C t |B_r|,
\ee
for every $B_r(x_0)\subset D$ centered on $\partial\Omega$ and every $t\in (0,1)$.
\end{lemma}

\begin{proof}
The estimate is contained in the proof of \cite[Theorem 1.10]{BMV1}; we sketch the idea for the sake of completeness. Let $x_0 = 0 \in \partial \O$ and $t>0$. We fix a function $\eta \in C^\infty_c(B_{2r})$ such that $0\le \eta\le 1$ in $B_{2r}$ and $\eta \equiv 1$ in $B_r$, and we use the competitor
$$\phi = \eta (u_\Omega - rt)^+ + (1-\eta) u_\Omega.$$
to test the optimality condition \eqref{e:subsolution} in $B_{2r}$. Since
$$
\phi = u_\Omega- rt \eta \quad\mbox{in}\quad\{u_\Omega>rt \}\qquad\text{and}\qquad
\phi = (1-\eta) u_\Omega \quad\mbox{in}\quad\{0\leq u_\Omega\leq rt\}\,,
$$
we get
\begin{align*}
E_f(\phi,B_{2r}) =& \int_{\{u_\Omega>rt\}\cap B_{2r}}\bigg(\frac{1}{2}|\nabla u_\Omega|^2+ \frac{(rt)^2}{2}|\nabla \eta|^2-rt\nabla u_\Omega \cdot \nabla \eta - fu_\Omega+ frt\eta\bigg)\,dx\\
&+ \int_{\{0\leq u_\Omega\leq rt\}\cap B_{2r}}\bigg(\frac{(1-\eta)^2}{2}|\nabla u_\Omega|^2+ \frac{u_\Omega^2}{2}|\nabla \eta|^2-(1-\eta)u_\Omega \nabla u_\Omega\cdot \nabla \eta - f u_\Omega + f u_\Omega\eta\bigg)\,dx\\
\leq &\, E_f(u_\Omega,B_{2r}) - rt \int_{\{u_\Omega>rt\}\cap B_{2r}} \nabla u_\Omega \cdot \nabla \eta \,dx + \int_{B_{2r}}\bigg(\frac{(rt)^2}{2}|\nabla \eta|^2+ frt\eta\bigg)\,dx\\
&+ \int_{\{0\leq u_\Omega\leq rt\}\cap B_{2r}}\bigg(\frac{(1-\eta)^2-1}{2}|\nabla u_\Omega|^2-(1-\eta)u_\Omega \nabla u_\Omega \cdot \nabla \eta\bigg)\,dx\\
\leq &\, E_f(u_\Omega,B_{2r}) + C_d\Big(t\|\nabla u_\Omega\|_{L^{\infty}} +rt\norm{f}{L^\infty } +t^2\Big)|B_r|.
\end{align*}
Thus \eqref{e:subsolution} implies
\begin{align*}
\frac{C_1c_Q}{2}|\{0<u_\Omega\le rt\}\cap B_r|&\le\frac{C_1c_Q}{2}\Big(|\{u_\Omega>0\}\cap B_{2r}|-|\{\varphi>0\}\cap B_{2r}|\Big)\\
&\le
E_f(\phi,D)-E_f(u_\Omega,D)\le C_d\Big(t\|\nabla u\|_{L^{\infty}} +rt\norm{f}{L^\infty } +t^2\Big)|B_r|.
\end{align*}
Therefore, we get
$$
\frac{C_1c_Q}{2}|B_{r} \cap \{0\leq u\leq rt\}| \leq Ct|B_r|
$$
with $C>0$ depending on $d$, $\|u_\O\|_{L^\infty(B_{2r})}$, $M$, $C_1$ and $c_Q$.
\end{proof}

\section{Compactness and convergence of blow-up sequences} \label{section.blowup}

Take an optimal set $\Omega$  for~\eqref{e:intro-shape-opt-pb-in-main-teo} in some $D\subset\R^d$, and consider the corresponding state functions $u=u_\Omega$ and $v=v_\Omega$. For any $x_0 \in \partial \O\cap D$ and any sequence $r_k \to 0^+$, we set
\be\label{blow.upseq}
u_{x_0,r_k}(x) := \frac{1}{r_k}u(x_0 +r_k x)\,,\quad v_{x_0 ,r_k}(x) := \frac{1}{r_k}v(x_0 +r_k x)\,,\quad\Omega_{x_0,r_k}:= \frac{\O - x_0 }{r_k}\,.
\ee
Since $u$ and $v$ vanish in $x_0$ and are Lipschitz (in a neighborhood of $x_0$), we have that $u_{x_0,r_k}$ and $v_{x_0,r_k}$ vanish in $0$ and (for large $k$) are uniformly Lipschitz in any ball $B_R$. Thus, there are functions $u_0,v_0:\R^d\to\R$ and subsequences of $u_{x_0,r_{k}}$ and $v_{x_0,r_{k}}$ that converge locally uniformly in $\R^d$ respectively to $u_0$ and $v_0$. As usual we say that $u_0$ and $v_0$ are blow-up limits of $u$ and $v$ in $x_0$; and we recall that they might depend on the sequence $r_{k}$. We notice that the blow-up limits of $u$ and $v$ will always be taken along the same sequence $r_k\to0$. \smallskip

The main results are \cref{p:first-blow-up} and \cref{p:final-blow-up}. In \cref{p:first-blow-up} we list the properties of any couple of functions $u_0,v_0$ obtained as blow-up limits of $u,v$, while in \cref{p:final-blow-up} we show that there is at least one sequence $r_k\to0$ that provides blow-up limits which are $1$-homogeneous and stationary for the one-phase Alt-Caffarelli functional.
%

\subsection{A general lemma about the convergence of blow-up sequences}

The construction of the blow-up limit from \cref{p:final-blow-up} will require taking three consecutive blow-ups. We give here a general lemma, which we will use several times in this Section.

\begin{lemma}\label{l:lemma-generale-blow-up-1}
Let $B_{2R}$ be a ball in $\R^d$. Let $u_n:\overline B_{2R}\to\R$ be a sequence of non-negative Lipschitz functions converging uniformly to a Lipschitz function $u_\infty:\overline B_{2R}\to\R$, and suppose that there is a constant $L>0$ such that
$$\|\nabla u_n\|_{L^\infty(B_{2R})}\le L\quad\text{for every}\quad n\ge 1.$$
Then, the following holds.
\begin{enumerate}[\quad\rm(i)]
\item Suppose that there is a constant $\tilde C>0$ such that, for every $n\ge 1$,
\begin{equation}\label{e:lemma-generale-blow-up-1-ipo-nondegeneracy}
\sup_{B_r(x_0)} u_n \ge \tilde{C} r\quad\text{for every}\quad x_0\in B_R\cap\overline{\{u_n>0\}}\quad\text{and every}\quad r\in(0,R)\,.
\end{equation}
Then
\begin{equation}\label{e:lemma-generale-blow-up-1-tesi-nondegeneracy}
\sup_{B_r(x_0)} u_\infty \ge \tilde{C} r\quad\text{for every}\quad x_0\in B_R\cap\overline{\{u_\infty>0\}}\quad\text{and every}\quad r\in(0,R),
\end{equation}
and
\begin{equation}\label{e:lemma-generale-blow-up-1-tesi-level-sets}
\ind_{\{u_n>0\}}\to\ind_{\{u_\infty>0\}}\quad\text{pointwise a.e. in}\quad B_R\,.
\end{equation}
\item Suppose that there is a constant $\tilde M>0$ such that, for every $n\ge 1$, we have the bound
\begin{equation}\label{e:lemma-generale-blow-up-1-ipo-laplacian}
\left|\int_{B_{2R}}\nabla u_n\cdot\nabla \varphi\,dx\right| \le \tilde{M} \|\varphi\|_{L^\infty(B_{2R})}\quad\text{for every}\quad \varphi\in C^{0,1}_c(B_{2R}),
\end{equation}
where $C^{0,1}_c(B_{2R})$ is the space of Lipschitz functions with compact support in $B_{2R}$.\\
Then $u_n$ converges to $u_\infty$ strongly in $H^1(B_R)$.
\end{enumerate}	
\end{lemma}
\begin{proof}
We first prove (i). Suppose that $x_0\in B_R\cap\overline{\{u_\infty>0\}}$. Then, there is a sequence $x_n\to x_0$ of points $x_n\in B_R\cap\overline{\{u_\infty>0\}}$. Let $r_n:=r-|x_n-x_0|$. By \eqref{e:lemma-generale-blow-up-1-ipo-nondegeneracy}, there is a point $y_n\in B_{r_n}(x_n)\subset B_r(x_0)$ such that $u_n(y_n)\ge \widetilde Cr_n$. Thus, by the uniform convergence of $u_n$, we get \eqref{e:lemma-generale-blow-up-1-tesi-nondegeneracy}. In order to prove \eqref{e:lemma-generale-blow-up-1-tesi-level-sets}, we first notice that by the pointwise convergence of $u_n$ to $u_\infty$, we have that
$$\ind_{\{u_\infty>0\}}(x_0)=1\quad\Rightarrow\quad \ind_{\{u_n>0\}}(x_0)=1\ \text{ for large $n$}\,,$$
so it is sufficient to prove that the set
$$S:=\Big\{x_0\in B_{R}\ :\  \ind_{\{u_\infty>0\}}(x_0)=0\ \text{ and }\ \limsup_{n\to\infty}\ind_{\{u_{n}>0\}}(x_0)=1\Big\}\,,$$
is of measure zero. Now, by \eqref{e:lemma-generale-blow-up-1-ipo-nondegeneracy}, we get that for every $r\in(0,R)$ there is a sequence $y_n\to x_0$ such that $u_n(y_n)\ge \tilde Cr$. Then, by the uniform convergence of $u_n$, we have that there is $y_\infty\in \overline B_r(x_0)$ such that $u_\infty(x_\infty)\ge \tilde Cr$. Then, by the Lipschitz continuity of $u_\infty$, the ball $B_{\tilde C r/L}(y_n)$ is contained in $\{u_\infty>0\}$. Since $r$ is arbitrary, we get that the Lebesgue density of $\{u_\infty>0\}$ in $x_0$ cannot be zero, so the set $S$ has zero measure. \smallskip

\noindent We next prove (ii). Since $u_n$ is uniformly bounded in $H^1(B_{2R})$, we have that $u_n$ converges to $u_\infty$ weakly in $H^1(B_{2R})$. Thus, the estimate \eqref{e:lemma-generale-blow-up-1-ipo-laplacian} holds also for $u_\infty$, that is
$$\left|\int_{B_{2R}}\nabla u_\infty\cdot\nabla \varphi\,dx\right| \le \tilde{M} \|\varphi\|_{L^\infty(B_{2R})}\quad\text{for every}\quad \varphi\in C^{0,1}_c(B_{2R}).$$
 Choose a function $\varphi\in C^\infty_c(B_{2R})$ such that $\varphi\equiv 1$ in $B_R$. Then,
\begin{align*}
\int_{B_R}|\nabla (u_n-u_\infty)|^2\,dx&\le \int_{B_{2R}}|\nabla (\varphi(u_n-u_\infty))|^2\,dx\\
&= \int_{B_{2R}}|\nabla \varphi|^2(u_n-u_\infty)^2\,dx+\int_{B_{2R}}\nabla (u_n-u_\infty)\cdot\nabla\big(\varphi^2(u_n-u_\infty)\big)\,dx\\
&\le \int_{B_{2R}}|\nabla \varphi|^2(u_n-u_\infty)^2\,dx+2\tilde M\|\varphi^2(u_n-u_\infty)\|_{L^\infty(B_{2R})}.
\end{align*}
Since the right-hand side converges to zero, we get the claim.
\end{proof}	

Finally, we notice that the above lemma can be applied to the state functions $u_\Omega$ and $v_\Omega$ of an optimal domain. This is a consequence of the following lemma.

\begin{lemma}\label{l:lemma-generale-blow-up-2}
	Let $B_{2R}\subset\R^d$ and $u\in H^1(B_{2R})$ be a non-negative  function with the following properties.
	\begin{enumerate}[\quad\rm(a)]
		\item There is a function $f\in L^\infty(B_{2R})$ such that
		\begin{equation*}
        -\Delta u=f\quad\text{in}\quad \Omega_u:=\{u>0\},\qquad u=0\quad\text{su}\quad \partial\Omega_u\cap B_{2R}\,,
		\end{equation*}
in the sense that
$$\int_{B_{2R}}\nabla u\cdot\nabla\varphi\,dx=\int_{B_{2R}}uf\,dx\quad\text{for every}\quad \varphi\in H^1_0(B_{2R})\quad\text{with}\quad \varphi=0\quad\text{in}\quad B_{2R}\setminus\Omega_u\,.$$
\item There is $\Lambda>0$ such that, for every non-negative $\varphi \in H^1_0(B_{2R})$,
\begin{equation*}
E_f(u,B_{2R}) + \frac{\Lambda}{2}|B_{2R}\cap\{u>0\}| \le E_f(u+\varphi,B_{2R}) + \frac{\Lambda}{2}|B_{2R}\cap\{u+\varphi>0\}|.
\end{equation*}
	\end{enumerate}	
Then, for every $\varphi\in C^{0,1}_c(B_{R})$, we have
$$\left|\int_{B_{R}}\nabla u\cdot\nabla\varphi\,dx\right|\le C_d\Big(1+\Lambda+R\|f\|_{L^\infty(B_{2R})}\Big)R^{d-1}\|\varphi\|_{L^\infty(B_{R})}\,.$$
\end{lemma}
\begin{proof}
We only sketch the proof and we refer to \cite[Chapter~3]{V} for the details. By (a) we have that
$$\mu:=\Delta u+|f|$$
is a positive Radon measure on $B_{2R}$, where
$$\int_{B_{2R}}\varphi\,d\mu:=\int_{B_{2R}}\Big(-\nabla u\cdot\nabla\varphi+\varphi |f|\Big)\,dx\quad\text{for every}\quad\varphi\in C^{0,1}_c(B_{2R}).$$
By testing the optimality condition in $(b)$ with a function $\varphi\in C^{0,1}_c(B_{2R})$, we get
$$\int_{B_{2R}}\varphi\,d\mu =\int_{B_{2R}}\Big(-\nabla u\cdot\nabla\varphi+\varphi f\Big)\,dx\le \int_{B_{2R}}|\nabla \varphi|^2\,dx+\Lambda|B_{2R}|,$$
and choosing $\varphi=R\phi$ with $\phi\equiv 1$ in $B_{R}$, we obtain
$$\mu(B_{R})\le \int_{B_{2R}}\phi\,d\mu\le C_d(1+\Lambda)R^{d-1}.$$
As a consequence, for every $\varphi\in C^{0,1}_c(B_R)$, we have
$$\left|\int_{B_{R}}\nabla u\cdot\nabla\varphi\,dx\right|\le \int_{B_R}|\varphi|\,d\mu+\int_{B_{R}}|\varphi| |f|\,dx\le C_d\Big(1+\Lambda+R\|f\|_{L^\infty}\Big)R^{d-1}\|\varphi\|_{L^\infty(B_R)}.$$
which concludes the proof.
\end{proof}	

\subsection{First blow-up}\label{s:1}
In this subsection we list the properties of any couple of blow-ups
$$u_0,v_0:\R^d\to\R$$
of the state functions $u_\Omega$ and $v_\Omega$ at a boundary point $x_0\in\partial\Omega\cap D$. The qualitative properties (Lipschitz continuity, non-degeneracy, density estimates) of $u_\Omega$ and $v_\Omega$ are conserved under blow-up limits; the stationarity condition also passes to the limit; the main difference is that $u_0$ and $v_0$ are harmonic where they are positive.

\begin{prop}\label{p:first-blow-up}
Let $D\subset\R^d$ be a bounded open set, let $\O$ be a solution to \eqref{e:intro-shape-opt-pb-in-main-teo} with $f,g,Q$ as in \cref{thm.main}, and let $u:=u_\Omega$ and $v:=v_\Omega$ be the state functions on $\Omega$ defined in \eqref{e:state-equation-open} and \eqref{e:stateqv}. We consider a point $x_0 \in \partial \Omega\cap D$ and blow-up sequences $u_{x_0,r_k},v_{x_0,r_k}$ of $u,v$ converging locally uniformly in $\R^d$ to blow-up limits $u_0,v_0\in C^{0,1}(\R^d)$.
Then
\begin{enumerate}
\item taking $C_1$ and $C_2$ to be the constants from \eqref{e:feg}, we have
$$C_1v_0\le u_0\le C_2v_0\quad\text{on}\quad\R^d;$$
\item the functions $u_0$ are $v_0$ are harmonic in the open set $\Omega_0:=\{u_0>0\}=\{v_0>0\}.$
\item \label{item:measure} there are constants $\eps_0>0$ and $C>0$ such that
\be\label{eq:densestbu}
\eps_0 |B_r|\leq |B_r(x_0)\cap \Omega_0| \leq (1-\eps_0)|B_r|,
\ee
for every $x_0\in \partial \Omega_0$, $r>0$, and
\be\label{conditionf.rescal}
|\{0<u_0<rt\}\cap B_r(x_0)|\leq Ct|B_r|,
\ee
for every $x_0\in \partial \O_0, r>0,t>0$;
\item \label{item:nondeg} $0\in\partial\Omega_0$ and there is a constant $C_0 > 0$ such that
$$\sup_{B_r(x_0)} u_0 \ge C_0 r\quad\text{and}\quad\sup_{B_r(x_0)} v \ge C_0 r\quad\text{for every}\quad x_0\in \overline\Omega_0\quad\text{and every}\quad r>0\,;$$
\item \label{item:ass} there is a constant $\Lambda>0$ depending on $C_1,C_2,C_Q$ and $d$ such that, for every $R>0$,
$$\left|\int_{B_{R}}\nabla u_0\cdot\nabla\varphi\,dx\right|+\left|\int_{B_{R}}\nabla v_0\cdot\nabla\varphi\,dx\right|\le \Lambda R^{d-1}\|\varphi\|_{L^\infty(B_{R})}\,;$$
\item\label{item:limit-first-variation} for every compactly supported smooth vector field $\xi \in C^\infty_c(\R^d;\R^d)$, we have
\be\label{stationary.Onep}
\int_{\R^d} \Big(- \nabla u_{0} \cdot \big((\nabla\xi) + (D\xi)\big)\nabla v_{0}+\big(\nabla u_{0}\cdot \nabla v_{0}+ Q(x_0)\ind_{\Omega_0} \big)\dive\xi\Big)\, dx=0.
\ee
\end{enumerate}
\end{prop}
\begin{proof}
For simplicity, we set
$$u_k:=u_{x_0,r_k},\ v_k:=v_{x_0,r_k},\ f_k:=f_{x_0,r_k},\ g_k:=g_{x_0,r_k}\ \text{ and }\ \O_k:= (\O-x_0)/r_k,$$
and we notice that
\begin{equation}\label{e:rescaled}
-\Delta u_k=r^2_kf_k\quad\text{and}\quad -\Delta v_k=r^2_kg_k\quad\text{in}\quad \Omega_k\,.
\end{equation}
The first two claims are just a consequence of the locally uniform convergence of $u_k$ and $v_k$ to $u_0$ and $v_0$.
By \cref{prop.almost}, \cref{lemm.lip} and \cref{lemm.nondeg}, we already know that $u$ and $v$ fulfill the assumptions of \cref{l:lemma-generale-blow-up-2}. Therefore, \eqref{item:nondeg} and \eqref{item:ass} follow from \cref{l:lemma-generale-blow-up-1}, while \eqref{item:limit-first-variation} follows from \cref{l:firstv} and \cref{oss:optimal-domains-are-stationary}, and the strong $H^1$ convergence of $u_k$ and $v_k$.\smallskip

We next prove \eqref{item:measure}. By  \cref{p:density} and rescaling, we know that, for every $k>0$,
\be\label{togh}
\eps_0 |B_r| \leq |B_r(x_0)\cap \{u_k>0\}|\leq (1-\eps_0)|B_r|, \quad \text{for $r<r_0/r_k, x_0 \in \partial \O_k$},
\ee
so, by the strong convergence of $\ind_{\Omega_k}$ to $\ind_{\Omega_0}$ in $L^1_{loc}(\R^d)$ (see \cref{l:lemma-generale-blow-up-1}), we get the density estimate \eqref{eq:densestbu} for $\Omega_0$. Similarly, by rescaling \eqref{e:conditionf} we obtain
$$
|\{0<u_k<rt\}\cap B_r(x_0)|\leq C t |B_r| \quad \mbox{for}\quad  r<\frac{r_0}{r_k}\,,\ x_0\in\partial \O_k\,,\ t>0\,,
$$
which, passing to the limit as $k\to\infty$, gives \eqref{conditionf.rescal}.
\end{proof}

\subsection{Second blow-up}\label{sub.homo}
Consider blow-up limits $u_{0},v_{0}$ as in the previous subsection and let
$$u_{00},v_{00}:\R^d\to\R\,.$$
be blow-up limits of $u_0$ and $v_0$ in zero. Then $u_{00}$ and $v_{00}$ are still blow-up limits of the state functions $u_\Omega$ and $v_\Omega$ at $x_0\in\partial\Omega\cap D$ and \cref{p:first-blow-up} still applies. On the other hand, \cite[Theorem 1.2]{mtv.BH} applies to the domain $\Omega_0$ and the function $u_0$, so the Boundary Harnack Principle (see~\cite[Definition~1.1]{mtv.BH}) holds on $\Omega_0$; since $u_0$ and $v_0$ are harmonic in $\Omega_0$, we get that the ratio $u_0/v_0$ is H\"older continuous up to the boundary $\partial\Omega_0$. This, in particular means that the second blow-ups $u_{00}$ and $v_{00}$ at any boundary point (and thus in zero) are proportional.

\begin{lemma}\label{p:second-blow-up}
	Let $u_0,v_0:\R^d\to\R$ be non-negative Lipschitz functions on $\R^d$ with the same positivity set and let $\Omega_0:=\{u_0>0\}=\{v_0>0\}$. Suppose that $u_0$ and $v_0$ satisfy the conditions (1)-(6) from \cref{p:first-blow-up} and let $u_{00},v_{00}:\R^d\to\R$ be  blow-ups of $u_0,v_0$ at zero.
	Then
	\begin{enumerate}
		\item there is a constant $\lambda\in(C_1,C_2)$ such that $u_{00}=\lambda v_{00}$ {on} $\R^d$;
		\item the function $u_{00}$ is harmonic in the open set $\Omega_{00}:=\{u_{00}>0\}$;
		\item there are constants $\eps_0>0$ and $C>0$ such that
		\be\label{eq:densestbu00}
		\eps_0 |B_r|\leq |B_r(x_0)\cap \Omega_{00}| \leq (1-\eps_0)|B_r|,
		\ee
		for every $x_0\in \partial \Omega_{00}$, $r>0$, and
		\begin{equation*}
		|\{0<u_{00}<rt\}\cap B_r(x_0)|\leq Ct|B_r|,
		\end{equation*}
		for every $x_0\in \partial \O_{00}, r>0,t>0$;
		\item \label{item:nondeg-second} $0\in\partial\Omega_{00}$ and there is a constant $C_0 > 0$ such that
		$$\sup_{B_r(x_0)} u_{00} \ge C_0 r\quad\text{for every}\quad x_0\in \overline\Omega_{00}\quad\text{and every}\quad r>0\,;$$
		\item \label{item:ass-second} there is a constant $\Lambda>0$ such that, for every $R>0$,
		$$\left|\int_{B_{R}}\nabla u_{00}\cdot\nabla\varphi\,dx\right|\le \Lambda R^{d-1}\|\varphi\|_{L^\infty(B_{R})}\,;$$
		\item\label{item:limit-first-variation-second} for every compactly supported smooth vector field $\xi \in C^\infty_c(\R^d;\R^d)$, we have
		\be\label{stationary.Onep-00}
		\int_{\R^d} \bigg(- \nabla u_{00} \cdot \big((\nabla\xi) + (D\xi)\big)\nabla u_{00}+\Big(|\nabla u_{00}|^2+ {\lambda}Q(x_0)\ind_{\Omega_{00}} \Big)\dive\xi\bigg)\, dx=0.
		\ee
	\end{enumerate}
\end{lemma}
\begin{proof}
Let $R>0$ be fixed and let $\phi:=u_0$. In order to show that the Boundary Harnack Principle holds on $\Omega_0$, we check that $\Omega_0$ and $\phi$ satisfy the list of assumptions (a)-(g) from \cite[Theorem 1.2]{mtv.BH} in the ball $B_R$ :
\begin{enumerate}[\rm(a)]
\item by definition of $\Omega_0$, we have $\phi>0$ in $\Omega_0$ and $\phi\equiv 0$ on $B_R\setminus \Omega_0$;
\item by hypothesis $\phi$ is Lipschitz continuous in $\R^d$ and so, in $B_R$;
\item since $\phi$ is harmonic in $\Omega_0$ and satisfies the condition \eqref{item:nondeg} from \cref{p:first-blow-up}, we can apply \cite[Lemma 6.8]{V}; thus, there is a constant $\kappa>0$ such that
$$\phi\ge \kappa \,\text{\rm dist}_{B_R\setminus\Omega_0}\quad\text{in}\quad B_{\sfrac{R}{2}}\,;$$
\item since $\phi\geq 0$ and $\Delta \phi =0$ in $\O_0$, we have that $\Delta \phi\ge 0$ {in} $\R^d$;
\item for every $x_0\in\partial\Omega\cap B_{R}$, we have
$$|B_r(x_0)\setminus \Omega_0|\ge \eps_0 |B_r(x_0)|\quad\text{for every}\quad r\in(0,R-|x_0|)\,;$$
\item\label{item:6-levels} for every $x_0\in\partial\Omega_0\cap B_{R}$ and every $r\in(0,R-|x_0|)$, we have
$$\big|\{0<\phi<rt\}\cap B_{r}(x_0)\big|\le C t|B_r|\quad\text{for every}\quad t>0\,;.$$
\item by \eqref{item:nondeg} of \cref{p:first-blow-up}, for every $x_0\in\partial\Omega_0\cap B_{R}$ and every $r\in(0,R-|x_0|)$, we have
$$\sup_{B_r(x_0)}\phi\ge C_0 r.$$
\end{enumerate}	
Therefore, all the assumptions of \cite[Theorem 1.2]{mtv.BH} are fulfilled and so, the ratio $$\frac{u_0}{v_0}:\Omega_0\to\R\,,$$
can be extended to a H\"older continuous function on $\overline\Omega_0\cap B_{\sfrac{R}{2}}$. Thus, the blow-ups $u_{00}$ and $v_{00}$ are proportional and all the other claims follow as in \cref{p:first-blow-up}.
\end{proof}

\subsection{Third blow-up}\label{s:3}
Take the state functions $u_\Omega$ and $v_\Omega$ on an optimal domain $\Omega$. Let
$$u_{00},v_{00}:\R^d\to\R\,,$$
be the second blow-up limits of $u_\Omega$ and $v_\Omega$ at a free boundary point $x_0\in\partial\Omega\cap D$. Then, $u_{00}$ and $v_{00}$ are proportional and satisfy the conditions listed in \cref{p:second-blow-up}. We will show that if we perform a further blow-up in zero, then we obtain functions $u_{000},v_{000}:\R^d\to\R$ that still satisfy the conditions from \cref{p:second-blow-up} but are also $1$-homogeneous. Before we give the precise statement, we notice that the stationarity condition \eqref{stationary.Onep-00} implies the monotonicity of the associated Weiss' boundary adjusted energy from \cite{w}.
\begin{lemma}[Monotonicity formula]\label{l:weiss}
Let $B_R\subset\R^d$, $u \in H^1(B_R)$ be a continuous non-negative function, and let $\Omega:=\{u>0\}$. Suppose that for every smooth compactly supported vector field $\xi \in C^\infty_c(\R^d;\R^d)$, we have:
\be\label{statio}
\int_\O \Big(\dive\xi \left(|\nabla u|^2 + \Lambda \right)- \nabla u \cdot ((\nabla\xi) + (D\xi))\nabla u\Big) \, dx =0.
\ee
Then, for every $x_0 \in \partial \O, r>0$ the map
$$
r \mapsto W_\Lambda(u_{x_0,r}):= \int_{B_1}|\nabla u_{x_0,r}|^2\,dx+\Lambda|\{u_{x_0,r}>0\}\cap B_1|- \int_{\partial B_1} u_{x_0,r}^2 \, d\HH^{d-1},
$$
is non-decreasing in $(0,+\infty)$ and
\be\label{weiss.deriv}
\frac{\partial}{\partial r} W_\Lambda(u_{x_0,r}) \geq \frac{2}{r}\int_{\partial B_1}|x\cdot \nabla u_{x_0,r}-u_{x_0,r}|^2\, d\HH^{d-1},
\ee
In particular, if $r\mapsto W_\Lambda(u_{x_0,r})$ is constant, then $u$ is $1$-homogeneous.
\end{lemma}
\begin{proof}
See for instance \cite[Proposition 9.9]{V}.	
\end{proof}	

\begin{lemma}\label{l:third-blow-up}
	Let $u_{00},v_{00}:\R^d\to\R$ be non-negative Lipschitz functions on $\R^d$ with the same positivity set and let $\Omega_{00}:=\{u_{00}>0\}=\{v_{00}>0\}$. Suppose that $u_{00}$ and $v_{00}$ satisfy the conditions (1)-(6) from \cref{p:second-blow-up} and let $u_{000},v_{000}:\R^d\to\R$ be  blow-ups of $u_{00},v_{00}$ at zero.
	Then:
	\begin{enumerate}
		\item \label{item:000-proportionality} $u_{000}=\lambda v_{000}$ {on} $\R^d$;
		\item\label{item:000-first-variation} for every compactly supported smooth vector field $\xi \in C^\infty_c(\R^d;\R^d)$, we have
		\be\label{e:third-blow-up-stationary}
		\int_{\R^d} \bigg(- \nabla u_{000} \cdot \big((\nabla\xi) + (D\xi)\big)\nabla u_{000}+\Big(|\nabla u_{000}|^2+ {\lambda}Q(x_0)\ind_{\Omega_{000}} \Big)\dive\xi\bigg)\, dx=0;
		\ee
		\item\label{item:000-homogeneity} $u_{000}$ is $1$-homogeneous in $\R^d$.
	\end{enumerate}
\end{lemma}
\begin{proof}
Set $\Lambda:={\lambda}Q(x_0)$ and for simplicity, let $u:=u_{00}$. By \cref{l:weiss}, the function
$$r\mapsto W_\Lambda\big(u_r\big),$$
is non-decreasing in $r$ and so, it admits a limit $\Theta$ as $r\to0$; moreover, by the Lipschitz continuity of $u$, $\Theta$ is finite. Let $r_k\to0$ be such that $u_{r_k}\to u_{000}$. Then, by Lemma \ref{l:lemma-generale-blow-up-1}, $u_{r_k}$ converges to $u_{000}$ strongly in $H^1_{loc}$ and the level sets $\{u_{r_k}>0\}$ converge in $L^1_{loc}$ to $\Omega_{000}$. Thus, for any $s>0$
$$\Theta:=\lim_{k\to+\infty}W_\Lambda\big(u_{sr_k}\big)=W_\Lambda\big((u_{000})_s\big).$$
Moreover, using again the strong convergence of $u_{r_k}$ and their level sets, we get that $u_{000}$ satisfies \eqref{e:third-blow-up-stationary}. Thus, using again \cref{l:weiss} and the fact that $s\mapsto W_\Lambda\big((u_{000})_s\big)$ is constantly equal to $\Theta$, we get that $u_{000}$ is homogeneous.
\end{proof}	

As an immediate consequence, we obtain the following proposition.

\begin{prop}[Existence of stationary $1$-homogeneous blow-ups]\label{p:final-blow-up}
	Let $D\subset\R^d$ be a bounded open set and $\O$ be a solution to \eqref{e:intro-shape-opt-pb-in-main-teo} with $f,g,Q$ as in \cref{thm.main}. Let $u:=u_\Omega$ and $v:=v_\Omega$ be the state functions on $\Omega$ defined in \eqref{e:state-equation-open} and \eqref{e:stateqv} and let $x_0 \in \partial \Omega\cap D$. Then, there is a sequence
	$r_k\to0$ such that the corresponding blow-up limits
	$$u_0:=\lim_{k\to\infty}u_{x_0,r_k}\qquad\text{and}\qquad v_0:=\lim_{k\to\infty}v_{x_0,r_k}\,$$
	satisfy the following conditions:
	\begin{enumerate}[\rm(i)]
		\item\label{item:i} $u_0$ and $v_0$ are $1$-homogeneous in $\R^d$ and $u_{0}=\lambda v_{0}$ for some constant $\lambda\in(C_1,C_2)$;
     	\item\label{item:ii} for every compactly supported smooth vector field $\xi \in C^\infty_c(\R^d;\R^d)$, we have
		\be\label{e:final-blow-up-stationary}
		\int_{\R^d} \bigg(- \nabla u_{0} \cdot \big((\nabla\xi) + (D\xi)\big)\nabla u_{0}+\Big(|\nabla u_{0}|^2+ {\lambda}Q(x_0)\ind_{\Omega_{0}} \Big)\dive\xi\bigg)\, dx=0.
		\ee
	\end{enumerate}
\end{prop}
\begin{remark}\label{rem:blowup}
  The blow-up sequence arising from \cref{p:final-blow-up} is obtained by a diagonal argument between the three different blow-ups defined respectively in Subsection \ref{s:1}, Subsection \ref{sub.homo} and Subsection \ref{s:3}. Therefore, by combining all the previous result, we can see that $u_0$ and $v_0$ fulfill the conditions \ref{item:def-stable-1}, \ref{item:def-stable-2}, \ref{item:def-stable-3}, \ref{item:def-stable-4} and \ref{item:def-stable-5} in \cref{def:global-stable-solutions}.
\end{remark}

\section{Regular and singular parts of the free boundary}\label{s:decomposition}

	Let $D\subset\R^d$ be a bounded open set and let $\O$ be a solution to the problem \eqref{e:intro-shape-opt-pb-in-main-teo}. As in \cref{section.blowup}, we denote by $u, v$ the associated state variables.
	\begin{definition}[Regular and singular points]\label{d:decomposition-free-boundary}
\rm 		We will say that a boundary point $x_0\in\partial\Omega\cap D$ is {\it regular} if
		there is a blow-up limit $(u_{0},v_0)$ of $(u,v)$ at $x_0$ such that
		$$ u_{0}(x) = \alpha (x \cdot \nu)_+\qquad \mbox{and}\qquad v_{0}(x)=\beta (x \cdot \nu)_+\ ,
		$$
		for some unit vector $\nu \in \R^d$ and some $\alpha>0$ and $\beta>0$ such that $\alpha \beta = Q(x_0)$. If such a blow-up limit does not exist, then we will say that $x_0$ is {\it singular}.
		\end{definition}
	We will denote by $\text{\rm Reg}(\partial \O)$ the set of all regular points on $\partial\Omega\cap D$ and by $\text{\rm Sing}(\partial \O)$ the set of all singular points on $\partial\Omega\cap D$. Clearly, we have that
	$$\text{\rm Reg}(\partial \O)\cap \text{\rm Sing}(\partial \O)=\emptyset\qquad\text{and}\qquad \text{\rm Reg}(\partial \O)\cup \text{\rm Sing}(\partial \O)=\partial\Omega\cup D\,.$$
	In \cref{section.reg} we will show that the regular part $\text{\rm Reg}(\partial \O)$ is in fact, locally, a smooth manifold (and in particular, a relatively open subset of $\partial\Omega\cap D$), while in \cref{section.sing} we will give an estimate on the dimension of the singular set $\text{\rm Sing}(\partial \O)$.

\subsection{Regular and singular parts in dimension two}
One can easily show that in dimension $d=2$ the free boundary $\partial\Omega\cap D$ is composed only of regular points; in particular, in dimension two the proof of \cref{thm.main} is concluded already in \cref{section.reg}, while the results from \cref{section.sing} are needed only when $d\ge 3$.

\begin{lemma}\label{l:regular-set-d=2}
Let $D$ be a bounded open set in $\R^2$ and let $\O$ be a solution to \eqref{e:intro-shape-opt-pb-in-main-teo}. Then, every point $x_0 \in \partial \O\cap D$ is a regular point in the sense of \cref{d:decomposition-free-boundary}.
\end{lemma}
\begin{proof}
Let $r_k\to0$ be a sequence such that the blow-up limits
$$u_0:=\lim_{k\to\infty}u_{r_k,x_0}\qquad\text{and}\qquad v_0:=\lim_{k\to\infty}v_{r_k,x_0}\,,$$
are, as in \cref{p:final-blow-up}, proportional ($u_0=\lambda v_0$), non-negative and $1$-homogeneous functions on $\R^d$, which are harmonic on the positivity set $\Omega_0:=\{u_0>0\}=\{v_0>0\}$. Now, reasoning as in \cite[Proposition 9.13]{V}, we write $u_0$ and $v_0$ in polar coordinates as
$$u_0(r,\theta)=r\phi(\theta)\quad\text{and}\quad v_0(r,\theta)=r\psi(\theta)\,,$$
where (since $u_0$ and $v_0$ are harmonic) $\phi$ and $\psi$ are solutions to
$$-\phi''(\theta)=\phi(\theta)\quad\text{and}\quad-\psi''(\theta)=\psi(\theta)\quad\text{in}\quad\partial\Omega_0\cap\partial B_1.$$
Since the only solutions of this equations (up to a rotation) are multiples of $\sin\theta$, and since $\HH^{1}(\Omega_0\cap\partial B_1)<2\pi$ (this follows from the density estimate in \cref{p:first-blow-up}), we get that
$$ u_{0}(x) = \alpha (x \cdot \nu)_+\qquad \mbox{and}\qquad v_{0}(x)=\beta (x \cdot \nu)_+\ ,
$$
for some unit vector $\nu \in \R^d$ and some $\alpha>0$ and $\beta>0$. Moreover, by \cref{p:final-blow-up}, the function $u_0$ is critical point for the one-phase problem (that is, \eqref{e:third-blow-up-stationary} holds). Thus,
$$|\nabla u_0|^2={\lambda}Q(x_0)\quad\text{on}\quad\partial\Omega_0=\{x\in\R^2\,:\,x\cdot\nu=0\},$$
 where $\lambda=\alpha/\beta$. Since $|\nabla u_0|^2=\alpha^2$, we get that $\alpha\beta=Q(x_0)$, which concludes the proof.
\end{proof}

\subsection{A geometric condition for the regularity in every dimension}

By an argument similar to the one in \cref{l:regular-set-d=2}, we have that the free boundary points admitting one-sided tangent ball are regular points. This result holds in every dimension and will be useful in \cref{section.reg}.

\begin{lemma}\label{l:tangent-ball-condition}
	Let $D$ be a bounded open set in $\R^d$, $d\ge 2$, and let $\O$ be a solution to \eqref{e:intro-shape-opt-pb-in-main-teo}. Suppose that there is a one-sided tangent ball at the boundary point $x_0 \in \partial \O\cap D$ in the sense that:
	\begin{equation}\label{e:one-sided-ball}
	\text{there is $B_r(y_0)\subset\Omega$ with $x_0\in\partial B_r(y_0)$ or there is  $B_r(z_0)\subset\R^d\setminus\overline{\Omega}$ with $x_0\in \partial B_r(z_0)$.}
	\end{equation}
		Then, $x_0$ is a regular point in the sense of \cref{d:decomposition-free-boundary}.
\end{lemma}
\begin{proof}
As in \cref{l:regular-set-d=2}, we consider a sequence $r_k\to0$ for which the blow-up limits
$$u_0:=\lim_{k\to\infty}u_{r_k,x_0}\qquad\text{and}\qquad v_0:=\lim_{k\to\infty}v_{r_k,x_0}\,,$$
are $1$-homogeneous, non-negative, harmonic on their positivity set $\Omega_0:=\{u_0>0\}$ and are such that $u_0=\lambda v_0$. The one-sided ball condition implies that there is a unit vector $\nu\in\R^d$ such that
$$\Omega_0\subseteq\{x\in\R^d\,:\,x\cdot\nu>0\}\quad\text{or}\quad \Omega_0\supseteq\{x\in\R^d\,:\,x\cdot\nu>0\}.$$
Since $\Omega_0$ satisfies an exterior density estimate (by \cref{p:first-blow-up}, claim \cref{item:measure}), the only possibility is that $\Omega_0=\{x\in\R^d\,:\,x\cdot\nu>0\}$ and
$$ u_{0}(x) = \alpha (x \cdot \nu)_+\qquad \mbox{and}\qquad v_{0}(x)=\beta (x \cdot \nu)_+\ ,
$$
for some $\alpha>0$ and $\beta>0$ with $\alpha/\beta=\lambda$. As in \cref{l:regular-set-d=2}, since $u_0$ satisfies \eqref{e:third-blow-up-stationary}, we get that
$$|\nabla u_0|^2=\frac{\alpha}{\beta}Q(x_0)\quad\text{on}\quad\partial\Omega_0=\{x\in\R^2\,:\,x\cdot\nu=0\},$$
which gives that $\alpha\beta=Q(x_0)$.
\end{proof}	

\section{Regularity of $\text{\rm Reg}(\partial \Omega)$}\label{section.reg}

In this Section we prove that the regular part $\text{\rm Reg}(\partial \Omega)$ of the boundary of an optimal set $\Omega$ is locally the graph of a smooth function.

\subsection{Viscosity formulation} In this subsection we prove that on the free boundary $\partial\Omega\cap D$ of an optimal set $\Omega$, solution to \eqref{e:intro-shape-opt-pb-in-main-teo}, we have the following optimality condition
$$|\nabla u_\Omega||\nabla v_\Omega| = Q\quad\mbox{on}\quad\partial \Omega\,,$$
in viscosity sense, as in \cite[Section 2]{mtv.flat}, in terms of the blow-up limit of the state functions $u_\Omega$ and $v_\Omega$ at free boundary points (see Remark \ref{r:equiv}).

\begin{defi}[Viscosity solutions]\label{d:viscosity-solutions}
Let $D$ be a bounded open set in $\R^d$ and let $f,g\in L^\infty(\R^d)$.	
Let $u,v :D\to\R$ be two non-negative continuous functions with the same support $$\Omega:=\{u>0\}=\{v>0\}\,,$$
on which they satisfy the PDE
\begin{equation}\label{e:equation-viscosity-inside}
-\Delta u = f\quad\text{and}\quad -\Delta v = g \quad\mbox{in}\quad \O\cap D\,.
\end{equation}
We say that the boundary condition
\begin{equation}\label{e:equation-viscosity-boiundary}
|\nabla u||\nabla v|=Q \quad\mbox{on}\quad\partial\Omega\cap D\,,
\end{equation}
holds in viscosity sense if at any point $x_0\in \partial\Omega\cap D$, at which $\O$ admits a one-sided tangent ball in the sense of \eqref{e:one-sided-ball}, there exist:
\begin{enumerate}[\quad\rm(a)]
\item a decreasing sequence $r_k\to 0$;
\item two positive constants $\alpha,\beta >0$ such that $\alpha \beta=Q(x_0)$;
\item a unit vector $\nu\in\R^d$;
\end{enumerate}	
such that the rescalings
$$u_{x_0,r_k}(x):=\frac{u(x_0+r_kx)}{r_k}\qquad\text{and}\qquad v_{x_0,r_k}(x):=\frac{v(x_0+r_kx)}{r_k}\,,$$
converge uniformly in every ball $B_R\subset\R^d$ respectively to the blow-up limits
\be
u_0(x):=\alpha \left(x\cdot\nu\right)_+\qquad\text{and}\qquad v_0(x):=\beta \left(x\cdot\nu\right)_+\ .
\ee
\end{defi}

\begin{remark}\label{r:equiv}
In \cite{mtv.flat} the authors addressed the $\eps$-regularity theory for viscosity solutions of \eqref{e:intro-generalsystem}. In particular, in \cite[Lemma 2.9]{mtv.flat} they proved that if the free boundary condition is satisfied in the sense of Definition \ref{d:viscosity-solutions}, then for any smooth function $\varphi\in C^\infty(D)$ the following holds:
\begin{enumerate}[\quad\rm(i)]
\item If $\varphi_+$ touches $\sqrt{uv}$ from below at a point $x_0\in D\cap\partial \Omega$, then $|\nabla \varphi(x_0)|\le \sqrt{Q(x_0)}$.
\item If $\varphi_+$ touches $\sqrt{uv}$ from above at a point $x_0\in D\cap\partial \Omega$, then $|\nabla \varphi(x_0)|\ge \sqrt{Q(x_0)}$.
\item If $a$ and $b$ are constants such that
$$a>0\,,\quad b>0\quad\text{and}\quad a b=Q(x_0),$$
and if $\varphi_+$ touches $w_{ab}:=\frac12(au+bv)$ from above at $x_0\in D\cap\partial \Omega$, then $|\nabla \varphi(x_0)|\ge \sqrt{Q(x_0)}$.
\end{enumerate}	
\end{remark}

\begin{prop}\label{p:viscosity-solutions}
Let $D\subset \R^d$ be a bounded open set and let $f,g,Q:D\to\R$ be as in \cref{thm.main}.	
Let $\O$ be a solution to \eqref{e:intro-shape-opt-pb-in-main-teo}. Then the state variables $u:=u_\Omega$ and $v:=v_\Omega$ satisfy
\be\label{e:FBcond}
|\nabla u||\nabla v|=Q\quad\mbox{on}\quad \partial \O\cap D\,,
\ee
in the sense of \cref{d:viscosity-solutions}.
\end{prop}
\begin{proof}
It follows as in the proof of \cref{l:tangent-ball-condition}.	
\end{proof}
We can now prove that $\text{\rm Reg}(\partial \O)$ is $C^{1,\alpha}$-regular for some $\alpha \in (0,1)$ by exploiting the $\eps$-regularity theory developed in \cite{mtv.flat}.

\begin{theo}[Regularity of $\text{\rm Reg}(\partial\Omega)$]\label{t:regularity-of-Reg}
Let $D$ be a bounded open set in $\R^d$, where $d\ge 2$. Let
$$f:D\to\R\,,\quad g:D\to\R\,,\quad Q:D\to\R\,,$$
be given non-negative functions.
Suppose that the following conditions hold:
\begin{enumerate}[\rm(a)]
	\item $f,g\in L^\infty(D)$;
	\item there are constants $C_1,C_2 > 0$ such that
	\begin{equation*}
	0\le C_1 g\le f\le C_2 g\quad\mbox{in}\quad D.
	\end{equation*}
	\item $Q\in C^{0,\alpha_Q}(D)$, for some $\alpha_Q>0$, and there are a positive constants $c_Q,C_Q$ such that
	$$0<c_Q\le Q(x)\leq C_Q\quad\text{for every}\quad x\in D\,.$$
\end{enumerate}	
Let $\O$ be a solution to \eqref{e:intro-shape-opt-pb-in-main-teo}. Then, the regular part $\text{\rm Reg}(\partial\Omega)$, defined in \cref{s:decomposition}, is locally the graph of a $C^{1,\alpha}$ function, for some $\alpha>0$.
\end{theo}
\begin{proof}
The proof follows by the epsilon-regularity theory developed in \cite{mtv.flat}. Indeed, if $x_0 \in \text{\rm Reg}(\partial\O)$ and $u_{x_0,\rho_k}, v_{x_0,\rho_k}$ are the blow-up sequences from the definition of $\text{\rm Reg}(\partial\Omega)$. Then, there are $\alpha,\beta>0 , \nu \in \R^d$ such that $\alpha \beta=Q(x_0), |\nu|=1$ and
$$\lim_{k\to\infty}\norm{u_{x_0,\rho_k}-\alpha (x\cdot \nu)_+}{L^\infty(B_1)}=0\quad\text{and}\quad
\lim_{k\to\infty}\norm{v_{x_0,\rho_k}-\beta (x\cdot \nu)_+}{L^\infty(B_1)}=0.$$
Moreover, the Lipschitz continuity and the non-degeneracy imply that for every $\eps>0$ there exists $k_0>0$ such that for every $k\geq k_0$ we have
\begin{align*}
\alpha(x\cdot \nu -\eps)_+ \leq &u_{x_0,\rho_k} \leq \alpha(x\cdot \nu + \eps)_+\quad\text{for every}\quad x\in B_1,\\
\beta(x\cdot \nu -\eps)_+ \leq &v_{x_0,\rho_k} \leq \beta(x\cdot \nu + \eps)_+\quad\text{for every}\quad x\in B_1,
\end{align*}
that is, $u_{x_0,\rho_k}, v_{x_0,\rho_k}$ are $\eps$-flat in the direction $\nu$ (see \cite[Definition 1.2]{mtv.flat}). By rescaling the associated state equations, we have
$$
-\Delta u_{x_0,\rho_k} = \rho_k^2 f_{x_0,\rho_k} \quad\text{and}\quad -\Delta v_{x_0,\rho_k} = \rho_k^2 g_{x_0,\rho_k}\quad\mbox{in}\quad B_1\cap \{u_{x_0,\rho_k}>0\},
$$
where
$$
\norm{\Delta u_{x_0,\rho_k}}{L^\infty(B_1)}+ \norm{\Delta v_{x_0,\rho_k}}{L^\infty(B_1)} \leq \rho_k\left(\norm{f}{L^\infty(B_1)}+\norm{g}{L^\infty(B_1)}\right).
$$
On the other hand, since both $u_{x_0,\rho_k}$ and $v_{x_0,\rho_k}$ still satisfy \eqref{e:FBcond} in viscosity sense, by applying \cite[Theorem 3.1]{mtv.flat}, $\partial\{u_{x_0,\rho_k}>0\}$ is $C^{1,\alpha}$ in $B_{\sfrac12}$. Finally, the result follows by rescaling back to the original problem.
\end{proof}

\subsection{Higher regularity}
We can pass from $C^{1,\alpha}$ to $C^\infty$-regularity of $\text{\rm Reg}(\partial\Omega)$ by exploiting the higher order Boundary Harnack Principle for solutions to \eqref{e:intro-generalsystem}.
\begin{prop}[Higher regularity of $\text{\rm Reg}(\partial\Omega)$]\label{prop:smoothnessbdry}
Let $D$ be a bounded open set in $\R^d$ and let $f,g,Q$ be as in \cref{t:regularity-of-Reg}. If $f,g,Q\in C^{k,\alpha}(D)$ for some $k\ge 1$ and $\alpha>0$, then the regular part $\text{\rm Reg}(\partial\Omega)$ of the free boundary $\partial\Omega\cap D$ of any solution $\O$ to \eqref{e:intro-shape-opt-pb-in-main-teo} is locally the graph of a $C^{1+k,\alpha}$ function for some $\alpha>0$.
In particular, if $f,g,Q$ are $C^\infty$, then $\text{\rm Reg}(\partial\Omega)$ is locally the graph of a $C^\infty$ function.
\end{prop}
\begin{proof}
We use a bootstrap argument as in \cite[Section~5.4]{maztv1}. Suppose that $\text{\rm Reg}(\partial \O)$ is locally the graph of a $C^{k,\alpha}$-regular function, for some $k\geq 1$ (when $k=1$ the claim follows  from \cref{t:regularity-of-Reg}).  Let $x_0 \in \partial \O$ and $r>0$ be such that $\partial\O\cap B_r(x_0)= \text{\rm Reg}(\partial \O)$.
Since $u$ and $v$ satisfy \eqref{e:equation-viscosity-inside} in $\Omega$ and since $v$ is non-degenerate, we can apply the higher order Boundary Harnack Principle (for PDEs with right-hand side) from \cite[Theorem 1.3]{Kuk} and \cite[Theorem 1.3]{TVT}, obtaining a non-negative $C^{k,\alpha}$-regular  function $w\colon B_r(x_0)\cap \overline{\O}\to \R$ satisfying $u =w v$ in $B_r(x_0)\cap \overline{\O}$. Then, we have:
$$|\nabla u||\nabla v|=Q\quad\text{and}\quad |\nabla u|=w|\nabla v|\quad\text{on}\quad \partial \O\cap B_r(x_0)\,.$$
Thus, $u$ is a solution to the problem
\begin{equation}\label{eq:onephasefb}
-\Delta u=f\quad \text{in}\quad\Omega\cap B_r(x_0)\ ,\qquad |\nabla u||\nabla u|=\sqrt{wQ}\quad\text{on}\quad\partial \Omega\cap B_r(x_0),
\end{equation}
which by \cite[Theorem 2]{KN} implies that $\text{\rm Reg}(\partial \O)$ is locally the graph of a $C^{k+1,\alpha}$ function.
\end{proof}

\section{Stable homogeneous solutions of the one-phase Bernoulli problem}\label{section.stable}

In this Section we study the singular set of the stable global solutions of the one-phase Bernoulli problem (see \cref{def:global-stable-solutions} for the definition of global stable solution). The main results are \cref{t:stable-dimension-reduction} and \cref{t:stable-critical-dimension}, which we will use in \cref{t:dimension-of-sing} in order to estimate the dimension of the singular set of the free boundary of the optimal sets for \eqref{e:intro-shape-opt-pb-in-main-teo}. The present Section can be read separately from the rest of the paper; we will use only the Taylor expansions from \cref{section.variations} and the general results from \cref{section.blowup}.

\subsection{Solutions of PDEs in unbounded domains}\label{sub:section-stable-general-pde-theory}
In $\R^d$, $d\ge 3$, we define
$$\dot H^1(\R^d):=\Big\{u\in L^{2^\ast}(\R^d)\ :\ \nabla u\in L^2(\R^d;\R^d)\Big\}\quad\text{where}\quad 2^\ast:=\frac{2d}{d-2}\,.$$
It is well known that $\dot H^1(\R^d)$ is a Hilbert space equipped with the norm $\|u\|_{\dot H^1(\R^d)}:=\|\nabla u\|_{L^2(\R^d;\R^d)}$ and that $C^\infty_c(\R^d)$ is dense in $\dot H^1(\R^d)$. Moreover, given an open (bounded or unbounded) set $\Omega\subset\R^d$, we define the space $\dot H^1_0(\Omega)$ as the closure of $C^\infty_c(\Omega)$ with respect to $\|\cdot \|_{\dot H^1(\R^d)}$. \medskip

Let $\Omega$ be an open (bounded or unbounded) subset of $\R^d$ and let $F\in L^2(\R^d;\R^d)$ be a given vector field. We say that $w$ is a weak solution of the PDE
\begin{equation}\label{e:equation-with-div-F}
-\Delta w=\text{\rm div}\,F\quad\text{in}\quad \Omega\ ,\qquad w\in \dot H^1_0(\Omega)\,,
\end{equation}
if $w\in\dot H^1_0(\Omega)$ and
$$\int_{\R^d}\nabla w\cdot\nabla \varphi\,dx=-\int_{\R^d}\nabla\varphi\cdot F\,dx\quad\text{for every}\quad \varphi\in \dot H^1_0(\Omega).$$
It is standard to check that $w$ is a solution to \eqref{e:equation-with-div-F} if and only if $w$ minimizes the functional
\begin{equation}\label{e:equation-with-div-F-functional}
J(\varphi):=\frac12\int_{\R^d}|\nabla \varphi|^2\,dx+\int_{\R^d}\nabla\varphi\cdot F\,dx\,,
\end{equation}
among all functions $\varphi\in \dot H^1_0(\Omega)$. Since it is immediate to check that a minimizer of \eqref{e:equation-with-div-F-functional}  in $\dot H^1_0(\Omega)$ exists and is unique, we get that also the solution to \eqref{e:equation-with-div-F} exists and is unique. Finally, we notice that if $w\in \dot H^1_0(\Omega)$ is the solution to \eqref{e:equation-with-div-F}, then
$$\int_{\R^d}|\nabla w|^2\,dx=-\int_{\R^d}\nabla w\cdot F\le \|F\|_{L^2}\|\nabla w\|_{L^2}\,,$$
which gives
\begin{equation}\label{e:equation-with-div-F-L-2-norm}
\int_{\R^d}|\nabla w|^2\,dx\le \int_{\R^d}|F|^2\,dx\,.
\end{equation}

\begin{oss}[Exterior density estimate and the space $\dot H^1_0$]
We say that an open set $\Omega\subset\R^d$ satisfies a uniform exterior density estimate with a constant $c>0$ if
\begin{equation}\label{e:exterior-density-estimate-F}
|B_r(x)\setminus\Omega|\ge c|B_r|\quad\text{for every}\quad r\in(0,1)\quad\text{and every}\quad x\in \R^d\setminus\Omega\,.
\end{equation}
It is known (see for example~\cite{deve}) that if an open set $\Omega\subset\R^d$ satisfies \eqref{e:exterior-density-estimate-F} then the space $\dot H^1_0(\Omega)$ can be characterized as:
\begin{equation}\label{e:characterization-H-1-0}
\dot H^1_0(\Omega)=\Big\{u\in \dot H^1(\R^d)\ :\ u=0\ \text{ a.e. on }\ \R^d\setminus \Omega\Big\}.
\end{equation}
\end{oss}	
\begin{lemma}[Convergence of solutions]\label{l:equation-with-div-F-convergence}
Let $\Omega_n$ be a sequence of open sets in $\R^d$, $d\ge 3$ such that:
\begin{enumerate}[\ \rm(a)]
\item there is a constant $c>0$ such that, for every $n\ge 1$, $\Omega_n$ satisfies the  exterior density estimate \eqref{e:exterior-density-estimate-F};
\item there is an open set $\Omega_\infty\subset\R^d$ satisfying the exterior density estimate \eqref{e:exterior-density-estimate-F} with the constant $c>0$ and such that the sequence of characteristic functions $\ind_{\Omega_n}$ converges pointwise almost-everywhere in $\R^d$ to $\ind_{\Omega_\infty}$.
\end{enumerate}
Let $F_n\in L^2(\R^d;\R^d)$ be a sequence of vector fields converging strongly in $L^2(\R^d;\R^d)$ to the vector field $F_\infty\in L^2(\R^d;\R^d)$. For every $n\ge 1$, let $w_n$ be the solution to the PDE
\begin{equation}\label{e:equation-with-div-F-w-n}
-\Delta w_n=\text{\rm div}\,F_n\quad\text{in}\quad \Omega_n\ ,\qquad w_n\in \dot H^1_0(\Omega_n)\,.
\end{equation}
Then, $w_n$ converges strongly in $\dot H^1(\R^d)$ to the solution $w_\infty$ of
\begin{equation}\label{e:equation-with-div-F-w-infty}
-\Delta w_\infty=\text{\rm div}\,F_\infty\quad\text{in}\quad \Omega_\infty\ ,\qquad w_\infty\in \dot H^1_0(\Omega_\infty)\,.
\end{equation}
\end{lemma}
\begin{proof}
First of all, we notice that the sequence $w_n$ is bounded in $\dot H^1(\R^d)$ (by \eqref{e:equation-with-div-F-L-2-norm}). Thus, we can extract a subsequence, that we still denote by $w_{n}$, which converges weakly in $\dot H^1(\R^d)$ and pointwise almost-everywhere on $\R^d$ to a function $w\in \dot H^1(\R^d)$. We will show that $w=w_\infty$.\medskip

 First, notice that for almost every $x\in \R^d\setminus\Omega_\infty$ we have that
$$w_n(x)\to w(x)\quad\text{and}\quad \ind_{\Omega_n}(x)\to\ind_{\Omega_\infty}(x)=0.$$
But then $w_n(x)=0$ (by \eqref{e:characterization-H-1-0}) and so $w(x)=0$. Thus, using again \eqref{e:characterization-H-1-0}, we get  $w\in \dot H^1_0(\Omega_\infty)$. \medskip

Now, let $\varphi\in C^\infty_c(\Omega_\infty)$. We will show that for large enough $n$, $\varphi\in C^\infty_c(\Omega_n)$. Indeed, let $\delta>0$ be a constant such that $B_\delta(x)\subset\Omega_\infty$ for every $x$ in the support of $\varphi$. Suppose by contradiction that there is a sequence $x_n\in\overline{\{\varphi\neq 0\}}$ such that $x_n\notin \Omega_n$. Then, by the density estimate for $\Omega_n$, $|B_\delta(x_n)\cap \Omega_n|\le (1-c)|B_{\delta}|$. Now, up to a subsequence, $x_n\to x_\infty\in \overline{\{\varphi\neq 0\}}$. But then,
$$(1-c)|B_{\delta}|\ge \lim_{n\to\infty}|B_\delta(x_n)\cap \Omega_n|=|B_\delta(x_\infty)\cap \Omega_\infty|=|B_\delta|,$$
which is a contradiction. Thus, for large $n$, $\overline{\{\varphi\neq 0\}}\subset\Omega_n$. Now, using the equation for $w_n$,
$$\int_{\R^d}\nabla w_n\cdot\nabla \varphi\,dx=-\int_{\R^d}\nabla\varphi\cdot F_n\,dx\,$$
and passing to the limit, we get that
$$\int_{\R^d}\nabla w_\infty\cdot\nabla \varphi\,dx=-\int_{\R^d}\nabla\varphi\cdot F_\infty\,dx\,,$$
that is $w=w_\infty$.\medskip

 Finally, in order to prove that the convergence is strong, we use the equations for $w_n$ and $w_\infty$ and the strong convergence of $F_n$ to $F_\infty$:
 $$\int_{\R^d}|\nabla w_n|^2\,dx=-\int_{\R^d}\nabla w_n\cdot F_n\,dx\to -\int_{\R^d}\nabla w_\infty\cdot F_\infty\,dx=\int_{\R^d}|\nabla w_\infty|^2\,dx\,,$$
 which concludes the proof.
\end{proof}	

\subsection{Global stable solutions of the one-phase problem}
We define the functionals
$$\delta \mathcal G: C^{0,1}(\R)\times C^\infty_c(\R^d;\R^d)\to\R\qquad\text{and}\qquad \delta^2 \mathcal G: C^{0,1}(\R)\times C^\infty_c(\R^d;\R^d)\to\R\,,$$
as follows. Given:
\begin{itemize}
\item a Lipschitz function $u:\R^d\to\R$,
\item a smooth compactly supported vector field $\xi\in C^\infty_c(\R^d;\R^d)$,
\end{itemize}
we set:
$$\delta \mathcal G(u)[\xi]:=\int_{\R^d} \Big(\nabla u \cdot \delta A\nabla u+ \ind_{\Omega_u} \dive\xi\Big)\, dx\,,$$
\begin{align}\label{e:global-stable-solutions-delta-2-G}
\begin{aligned}
\delta^2 \mathcal G(u)[\xi]:=\,
\int_{\R^d}2\nabla u\cdot(\delta^2 A)\nabla u-2|\nabla (\delta u)|^2 +\ind_{\Omega_u}\Big((\dive\xi)^2+\xi\cdot\nabla(\text{\rm div}\,\xi)\Big)\,dx,
\end{aligned}
\end{align}
where $\Omega_u:=\{u>0\}$, and where $\delta A,\,\delta^2A$ are defined as in \eqref{e:first-and-second-variation-A}
and where $\delta u\in \dot H^1_0(\Omega_u)$ is the weak solution to the PDE
\begin{equation}\label{e:global-stable-solutions-delta-u}
-\Delta(\delta u)=\text{\rm div}\big((\delta A)\nabla u\big)\quad\text{in}\quad\Omega_u\ ,\qquad \delta u\in \dot H^1_0(\Omega_u).
\end{equation}

\begin{definition}[Global stable solutions of the one-phase problem]\label{def:global-stable-solutions}
We say that a function $u:\R^d\to\R$ is a global stable solution of the one-phase problem if, for every compactly supported smooth vector field $\xi \in C^\infty_c(\R^d;\R^d)$, we have:
\begin{equation}\label{e:global-stable-solutions-main}
\delta \mathcal G(u)[\xi]=0\qquad\text{and}\qquad \delta^2 \mathcal G(u)[\xi]\ge0\,,
\end{equation}
and if the following conditions hold:
\begin{enumerate}[\rm(a)]
	\item\label{item:def-stable-1} $u$ is globally Lipschitz continuous and non-negative on $\R^d$\rm ;\it
	\item\label{item:def-stable-2} $u$ is harmonic in the open set $\Omega_{u}:=\{u>0\}$\rm ;\it
	\item\label{item:def-stable-3}  there is a constant $c>0$ such that
	$$|B_r(x_0)\cap \Omega_{u}| \leq (1-c)|B_r|,$$
	for every $x_0\in \R^d\setminus \Omega_{u}$ and every $r>0$\rm ;\it
	\item\label{item:def-stable-4}  $0\in\partial\Omega_{u}$ and there is a constant $\eta > 0$ such that
	$$\sup_{B_r(x_0)} u \ge \eta r\quad\text{for every}\quad x_0\in \overline\Omega_{u}\quad\text{and every}\quad r>0\,;$$
	\item\label{item:def-stable-5}  there is a constant $C>0$ such that, for every $R>0$,
	$$\left|\int_{B_{R}}\nabla u\cdot\nabla\varphi\,dx\right|\le C R^{d-1}\|\varphi\|_{L^\infty(B_{R})}\quad\text{for every}\quad \varphi\in C^\infty_c(B_R)\,.$$
\end{enumerate}
\end{definition}

\begin{oss}
The functionals $\delta \mathcal{G}$ and $\delta^2\mathcal{G}$ correspond to the first and the second variation along vector fields of the one-phase Alt-Caffarelli functional\footnote{Precisely, since $\mathcal G(u)=\infty$ as soon as $u$ has positivity set of infinite measure, we will work with a localized version of $\mathcal G$, in which the integration is over compact sets.}
$$\mathcal{G}(u)=\int_{\R^d}\Big(|\nabla u|^2+\ind_{\{u>0\}}\Big)\,dx\,,$$
so one may expect that the natural definition of a global stable solution is a function $u:\R^d\to\R$ that satisfies \eqref{e:global-stable-solutions-main}. Unfortunately, the condition \eqref{e:global-stable-solutions-main} alone seems to be quite weak. For instance,
$$u(x,y)=1+|x|\quad\text{e}\quad u(x,y)=1+|xy|\,$$
satisfy \eqref{e:global-stable-solutions-main}, but they are not even harmonic in $\{u>0\}$, while the function
$$u(x,y)=|xy|\,,$$
is harmonic in $\{u>0\}$, but the free boundary $\partial\{u>0\}$ has a cross-like singularity in zero; more generally, the solutions of optimal partition problems satisfy \eqref{e:global-stable-solutions-main} and the structure of their nodal sets can be very different from the one of the one-phase free boundaries.

 We add the conditions \ref{item:def-stable-1}, \ref{item:def-stable-2}, \ref{item:def-stable-3}, \ref{item:def-stable-4} and \ref{item:def-stable-5} in order to have a regularity theory for one-phase stable solutions, which is similar to the one available for minimizers of the Alt-Caffarelli functional. The Lipschitz continuity \ref{item:def-stable-1} and the non-degeneracy \ref{item:def-stable-4} guarantee the existence of non-trivial blow-up limits obtained by $1$-homogeneous rescalings of $u$. The condition \ref{item:def-stable-5} is needed for the strong convergence of the blow-up sequences (see \cref{l:lemma-generale-blow-up-1}), which together with the exterior density estimate \ref{item:def-stable-3} allows to transfer the stability condition \eqref{e:global-stable-solutions-main} to the blow-up limits of $u$ (thanks to \cref{l:equation-with-div-F-convergence}). We also notice that the conditions \ref{item:def-stable-5} and \ref{item:def-stable-2} imply the Lipschitz continuity of $u$, so we could actually avoid adding \ref{item:def-stable-1} to the list.

 Finally, we highlight that the conditions \ref{item:def-stable-1}-\ref{item:def-stable-2}-\ref{item:def-stable-3}-\ref{item:def-stable-4}-\ref{item:def-stable-5}	are satisfied by the blow-ups of numerous one-phase problems, for instance by the state functions on the domains minimizing \eqref{e:intro-shape-opt-pb-in-main-teo} (see \cref{section.blowup}), and of course, by the global minimizers of the classical one-phase Bernoulli problem.
\end{oss}

\begin{oss}[Blow-ups of global stable solutions]
We notice that if $u:\R^d\to\R$ satisfies \ref{item:def-stable-1}, \ref{item:def-stable-2}, \ref{item:def-stable-3}, \ref{item:def-stable-4} and \ref{item:def-stable-5}, then any blow-up $u_0:\R^d\to\R$ of $u$ at $x_0\in\partial\Omega_u$,
$$u_0=\lim_{n\to\infty}u_{x_0,r_n}\ ,\quad\text{with}\quad u_{x_0,r_n}(x):=\frac{u(x_0+r_nx)}{r_n}\quad\text{and}\quad \lim_{n\to\infty}r_n=0,$$
still satisfies \ref{item:def-stable-1}, \ref{item:def-stable-2}, \ref{item:def-stable-3}, \ref{item:def-stable-4} and \ref{item:def-stable-5}. In particular, by \cref{l:lemma-generale-blow-up-1}, this means that the convergence $u_{x_0,r_n}\to u_0$ is strong in $H^1_{loc}$ and thus, by \cref{l:equation-with-div-F-convergence} and \cref{l:weiss}, if $u$ is a global stable solution, then $u_0$ is a $1$-homogeneous global stable solution.
\end{oss}
	
\begin{oss}[Decomposition of the free boundary]
Let $u:\R^d\to\R$ be a global stable solution in the sense of \cref{def:global-stable-solutions} and let $\Omega_u:=\{u>0\}$. We decompose the free boundary $\partial\Omega_u$ as
$$\partial\Omega_u=\text{\rm Reg}(\partial\Omega_u)\cup \text{\rm Sing}(\partial\Omega_u),$$
where the regular part $\text{\rm Reg}(\partial\Omega_u)$ consists of all points $x_0\in\partial\Omega_u$ at which there is a blow-up limit $u_0$ which is a half-space solution, that is,
\begin{equation}\label{e:half-space-solution}
u_0(x)=(x\cdot\nu)_+\quad\text{for some unit vector}\quad \nu\in\R^d,
\end{equation}
while the singular part is given by $\text{\rm Sing}(\partial\Omega_u)=\partial\Omega_u\setminus\text{\rm Reg}(\partial\Omega_u)$. As in \cref{section.reg}, it is immediate to check that the global stable solutions satisfy the optimality condition
$$|\nabla u|=1\quad\text{on}\quad\partial\Omega_u$$
in viscosity sense, so the $\eps$-regularity theorem of \cite{desilva} holds and we have that the regular part is a relatively open subset of $\partial\Omega_u$ and a $C^\infty$ manifold. We notice that this decomposition is precisely the one from \cref{s:decomposition} with $\alpha=\beta=Q=1$, $u=v$, and $f=g=0$.
\end{oss}	

\begin{definition}[Critical dimension for stable solutions]\label{def:global-stable-solution-critical-dimension}
We define $d^\ast$ to be the smallest dimension admitting a $1$-homogeneous global stable solution (in the sense of \cref{def:global-stable-solutions}) $u:\R^d\to\R$ with $0\in\text{\rm Sing}(\partial\Omega_u)$.
\end{definition}

\begin{theo}[Dimension reduction for stable solutions]\label{t:stable-dimension-reduction}
	Suppose that $u:\R^d\to\R$ is a $1$-homogeneous global stable solution of the one-phase problem (in the sense of \cref{def:global-stable-solutions}).
	\begin{enumerate}[\rm (i)]
		\item If $d< d^\ast$, then $u$ is a half-plane solution.
		\item If $d=d^\ast$, then $\text{\rm Sing}(\partial\Omega_u)=\{0\}$ or $u$ is a half-plane solution.
		\item If $d>d^\ast$, then the Hausdorff dimension of $\text{\rm Sing}(\partial\Omega_u)$ is at most $d-d^\ast$, that is,
		$$\HH^{d-d^\ast+\eps}\big(\text{\rm Sing}(\partial\Omega_u)\big)=0\quad\text{for every}\quad \eps>0,$$
	\end{enumerate}		
where $d^\ast$ is the critical dimension from \cref{def:global-stable-solution-critical-dimension}.
\end{theo}
\begin{proof}
The proof follows from a classical argument that can be found for instance in \cite{V} (in particular, \cite[Proposition 10.13]{V}).
\end{proof}	

\begin{theo}[Bounds on the critical dimension for stable solutions]\label{t:stable-critical-dimension}
$5\le d^\ast\le 7$, where $d^\ast$ is the critical dimension from \cref{def:global-stable-solution-critical-dimension}.
\end{theo}
\begin{proof}
The claim follows from \cref{p:global-minimizers-are-global-stable-solutions} and \cref{p:stable-cones-satisfy-stability-inequality} below.
\end{proof}	
	
\subsection{Global minimizers and global stable solutions}
Given an open set $D\subset\R^d$ and a function $u\in H^1(D)$, we define:
$$\mathcal G(u,D):=\int_{D}\Big(|\nabla u|^2+\ind_{\{u>0\}}\Big)\,dx\,.$$
\begin{definition}[Global minimizers]\label{def:global-minimizer}
We say that a function $u:\R^d\to\R$ is a global minimizer of the Alt-Caffarelli functional, if:
\begin{itemize}
\item $u$ is non-negative and $u\in H^1_{loc}(\R^d)$;
\item $\mathcal G(u,B_R)\le \mathcal G(v,B_R)$, for every $B_R\subset\R^d$ and every $v:\R^d\to\R$ such that $v-u\in H^1_0(B_R)$.
\end{itemize}	
\end{definition}
	
\begin{prop}\label{p:global-minimizers-are-global-stable-solutions}
Suppose that $u:\R^d\to\R$ is a global minimizer in the sense of \cref{def:global-minimizer}. Then, $u$ is a global stable solution in the sense of \cref{def:global-stable-solutions}. In particular, $d^\ast\le 7$, where $d^\ast$ is the critical dimension from \cref{def:global-stable-solution-critical-dimension}.
\end{prop}	
\begin{proof}
It is well-known that the global minimizers satisfy the conditions \ref{item:def-stable-1}-\ref{item:def-stable-2}-\ref{item:def-stable-3}-\ref{item:def-stable-4}-\ref{item:def-stable-5} of \cref{def:global-stable-solutions} (see for instance \cite{V}). Moreover, by \cite[Lemma 9.5 and Lemma 9.6]{V}, the global minimizers are critical points, that is,
$$\delta \mathcal{G}(u)[\xi]=0\qquad\text{for every}\qquad\xi\in C^\infty_c(\R^d;\R^d).$$
Thus, it only remains to prove the positivity of the second variation:
$$\delta^2 \mathcal G(u)[\xi]\ge 0\qquad\text{for every}\qquad\xi\in C^\infty_c(\R^d;\R^d).$$
Let $\xi\in C^\infty_c(\R^d;\R^d), \Phi_t$ be the associated defined by \eqref{e:flow-of-xi} and let $\Omega_t := \Phi_t(\Omega)$, for $t\in\R$. In $B_R$, we consider the solution $u_t$ to the problem
$$-\Delta u_t=0\quad\text{in}\quad\Omega_t\cap B_R\,,\qquad u_t=0\quad\text{on}\quad \partial \Omega_t\cap B_R\,, \qquad u_t=u\quad\text{on}\quad \partial B_R.$$
By the optimality of $u_t$, we have that
$$\mathcal G(u_t,B_R)\ge \mathcal G(u,B_R)\quad\text{for every}\quad t\in\R\,,$$
so,
$$\frac{d}{dt}\Big|_{t=0}\mathcal G(u_t,B_R)=0\quad\text{and}\quad \frac{d^2}{dt^2}\Big|_{t=0}\mathcal G(u_t,B_R)\ge0.$$
By \cref{l:secondstv}, we have that
\begin{align*}
\begin{aligned}
\frac{d^2}{dt^2}\Big|_{t=0}\mathcal G(u_t,B_R):=\,
\int_{B_R}2\nabla u\cdot(\delta^2 A)\nabla u-2|\nabla w_R|^2 +\ind_{\Omega_u}\Big((\dive\xi)^2+\xi\cdot\nabla(\text{\rm div}\,\xi)\Big)\,dx,
\end{aligned}
\end{align*}
where $\Omega_u=\{u>0\}$ and $w_R$ is the solution to the PDE
\begin{equation*}
-\Delta w_R=\text{\rm div}\big((\delta A)\nabla u\big)\quad\text{in}\quad\Omega_u\cap B_R\ ,\qquad w_R\in H^1_0(\Omega_u\cap B_R).
\end{equation*}
Thus, for every $R>0$,
\begin{align*}
\int_{\R^d}|\nabla w_R|^2\,dx\ge \,
\int_{B_R}\nabla u\cdot(\delta^2 A)\nabla u+\frac12\ind_{\Omega_u}\Big((\dive\xi)^2+\xi\cdot\nabla(\text{\rm div}\,\xi)\Big)\,dx\,.
\end{align*}
Since by \cref{l:equation-with-div-F-convergence} the sequence $w_R\to \delta u$ strongly in $\dot H^1(\R^d)$ as $R\to\infty$, we get that
\begin{align*}
\int_{\R^d}|\nabla (\delta u)|^2\,dx\ge \,
\int_{B_R}\nabla u\cdot(\delta^2 A)\nabla u+\frac12\ind_{\Omega_u}\Big((\dive\xi)^2+\xi\cdot\nabla(\text{\rm div}\,\xi)\big)\Big)\,dx\,,
\end{align*}
which is precisely the inequality $\delta^2\mathcal{G}(u)[\xi]\ge 0$. Finally, the bound $d^\ast\le 7$ follows by the example of a singular $1$-homogeneous global minimizer constructed by De Silva and Jerison in \cite{dj}.
\end{proof}

\subsection{Global stable solutions and the stability inequality of Caffarelli-Jerison-Kenig}$ $

\noindent Let $u:\R^d\to\R$ be a $1$-homogeneous global stable solution of the one-phase problem with an isolated singularity in zero, that is,
$$\text{\rm Sing}(\partial\Omega_u)=\{0\}.$$
In particular, we have that the regular part
$$\text{\rm Reg}(\partial\Omega_u):=\partial\Omega_u\setminus\{0\}\,,$$
is a smooth $C^\infty$ manifold and the function $u$ is $C^\infty$ in $\overline\Omega_u\setminus \{0\}$, up to the boundary $\partial\Omega_u\setminus \{0\}$. Thus, $u$ is a classical solution to the PDE
\begin{equation}\label{e:smooth-one-phase}
\Delta u=0\quad\text{in}\quad\Omega_u\,,\qquad |\nabla u|=1\quad\text{on}\quad\partial\Omega_u\setminus\{0\}\,.
\end{equation}
Together with the homogeneity of $u$ this implies (see for instance \cite{js}) that
$$H>0\quad\text{on}\quad \partial\Omega_u\setminus\{0\}\,,$$
where $H$ is the mean curvature of $\partial \Omega_u$ oriented towards the complement of $\Omega_u$.

We will say that $\Omega_u$ supports the {\it stability inequality of Caffarelli-Jerison-Kenig} if
\begin{equation}\label{e:stability-inequality}
\int_{\Omega_u}|\nabla \varphi|^2\,dx \geq \int_{\partial \Omega_u}H\varphi^2\,d\HH^{d-1}\qquad\text{for every}\qquad \varphi\in C^\infty_c(\R^d\setminus\{0\})\,.
\end{equation}
In \cite{cjk} and \cite{js} it was shown that if:
\begin{itemize}
\item $d=3$ (see \cite{cjk}) or $d=4$ (see \cite{js});	
\item $u:\R^d\to\R$ is a $1$-homogeneous non-negative Lipschitz function;
\item $\partial\Omega_u\setminus\{0\}$ is $C^\infty$ smooth;
\item $u$ is a solution of the one-phase Bernoulli problem \eqref{e:smooth-one-phase};
\item $\Omega_u$ supports the stability inequality \eqref{e:stability-inequality};
\end{itemize}
then $u$ is a half space solution, that is,
$$u(x)=(x\cdot\nu)_+\quad\text{for some unit vector}\quad \nu\in\R^d.$$
Thus, in order to show that in dimension $3$ and $4$ there are no global stable solutions (in the sense of \cref{def:global-stable-solutions}) with singularities, it is sufficient to prove the following proposition.

\begin{prop}[The global stable cones satisfy the stability inequality]\label{p:stable-cones-satisfy-stability-inequality}
Let $u:\R^d\to\R$ be a $1$-homogeneous global stable solution (in the sense of \cref{def:global-stable-solutions}) with $\text{\rm Sing}(\partial\Omega_u)=\{0\}$. Then, $\Omega_u$ supports the stability inequality \eqref{e:stability-inequality}. In particular, $d^\ast\ge 5$, where $d^\ast$ is the critical dimension from \cref{def:global-stable-solution-critical-dimension}.
\end{prop}	
\begin{oss}
Following the proof of \cref{p:stable-cones-satisfy-stability-inequality}, it is immediate to check that also the converse is true. Precisely, if $u:\R^d\to\R$ is a $1$-homogeneous function, if $\partial\Omega_u\setminus\{0\}$ is smooth and if $u$ is a solution to \eqref{e:smooth-one-phase} such that $\Omega_u$ supports the stability inequality \eqref{e:stability-inequality}, then $u$ is a global stable solution in the sense of \cref{def:global-stable-solutions}.
\end{oss}	
\begin{proof}
For any bounded open set $D\subset\R^d$ and any function $u\in H^1(D)$, we define:
$$\mathcal G(u,D):=\int_D|\nabla u|^2\,dx+|D\cap\{u>0\}|.$$	
Moreover, for any $R>1$, we call the annulus
$$A_R:=B_R\setminus \overline B_{{1}/{R}}\,.$$
We fix a smooth vector field $\xi\in C^\infty_c(A_R,\R^d)$ and we define the open set
$$\Omega_t:=\Phi_t(\Omega_u)\quad\text{for every}\quad t\in\R\,,$$
where $\Phi_t$ is the flow of $\xi$ defined by \eqref{e:flow-of-xi}. Let $u_t:A_R\to\R$ be the solution of the PDE
$$\Delta u_t=0\quad\text{in}\quad \Omega_t\cap A_R\,,\qquad u_t=0\quad\text{on}\quad \partial\Omega_t\cap A_R\,,\qquad  u_t=u\quad\text{in}\quad \Omega_t\cap \partial A_R\,.$$
\noindent{\bf Step 1.} We will show that
\begin{equation}\label{e:global-stable-solutions-1st-inequality}
\frac{d^2}{dt^2}\Big|_{t=0}\mathcal G(u_t,A_R)\ge \delta^2\mathcal{G}(u)[\xi],
\end{equation}
where $\delta^2\mathcal{G}$ is defined in \eqref{e:global-stable-solutions-delta-2-G}.
Following \cref{l:secondstv} we have that
\begin{align*}
\frac{d^2}{dt^2}\Big|_{t=0}\mathcal G(u_t,A_R)&=\int_{A_R\cap\Omega_u}2\nabla u\cdot(\delta^2 A)\nabla u +\Big((\dive\xi)^2+\xi\cdot\nabla(\text{\rm div}\,\xi)\Big)-2|\nabla w_R|^2\,dx.
\end{align*}
where $\delta A$ and $\delta^2A$ are defined in \eqref{e:first-and-second-variation-A} and $w_R$ is the solution of the PDE
\begin{equation}\label{e:global-stable-solutions-equation-w-R}
-\Delta w_R=\text{\rm div}\big((\delta A)\nabla u\big)\quad\text{in}\quad\Omega_u\cap A_R\ ,\qquad w_R\in H^1_0(\Omega_u\cap A_R).
\end{equation}
Now, let $\delta u$ be the solution to \eqref{e:global-stable-solutions-delta-u}. By the variational characterization of \eqref{e:global-stable-solutions-delta-u} in $\dot H^1_0(\Omega_u)$, the fact that $ w_R\in \dot H^1_0(\Omega_u)$, and an integration by parts, we have that
\begin{align*}
-\frac12\int_{B_R}|\nabla (\delta u)|^2\,dx&=\frac12\int_{B_R}|\nabla (\delta u)|^2\,dx+\int_{B_R}\nabla (\delta u)\cdot(\delta A)\nabla u\,dx\\
&\le \frac12\int_{B_R}|\nabla w_R|^2\,dx+\int_{B_R}\nabla w_R\cdot(\delta A)\nabla u\,dx=-\frac12\int_{B_R}|\nabla w_R|^2\,dx,
\end{align*}
which gives \eqref{e:global-stable-solutions-1st-inequality}.\medskip

\noindent{\bf Step 2.} We set $u'$ to be the solution of the PDE
\begin{equation}\label{eq:u'}
\begin{cases}
\Delta u' = 0 & \mbox{in}\quad\Omega_u\cap A_{R}\,,\\
u' = \xi\cdot \nu & \mbox{on}\quad\partial\Omega_u\cap A_{R}\,,\\
u' = 0 & \mbox{on}\quad\Omega_u\cap \partial A_{R}\,,
\end{cases}
\end{equation}
where $\nu$ is the outer normal to $\partial \Omega_u$. We will first show that
$$u'=w_R-\xi\cdot\nabla u.$$
Indeed, since $\xi$ is supported in $A_R$ and since $\nabla u=-\nu$ on $\partial\Omega_u\cap A_R$, we have:
$$u'=w_R-\xi\cdot\nabla u\quad\text{on}\quad \partial(\Omega_u\cap A_R).$$
In order to show that $u'$ is harmonic in $\Omega_u$, we compute (using the repeated index summation convention)
\begin{align*}
\text{\rm div}\big((\delta A)\nabla u\big)&=\partial_j\big[-\partial_i\xi_j\partial_iu-\partial_j\xi_i\partial_iu+\partial_i\xi_i\partial_ju\big]\\
&=-\partial_{ij}\xi_j\partial_iu-\partial_i\xi_j\partial_{ij}u-\partial_{jj}\xi_i\partial_iu-\partial_j\xi_i\partial_{ij}u+\partial_{ij}\xi_i\partial_ju\\
&=-2\partial_i\xi_j\partial_{ij}u-\partial_{jj}\xi_i\partial_iu=-\partial_j\big[\partial_j\xi_i\partial_iu+\xi_i\partial_{ij}u\big]=-\Delta(\xi\cdot\nabla u),
\end{align*}
and then we use the equation \eqref{e:global-stable-solutions-equation-w-R} for $w_R$.\medskip

\noindent{\bf Step 3.} We next compute the second derivative  of $\mathcal G(u_t,A_R)$ in terms of $u'$. For the sake of simplicity, through the rest of the proof we use the notations introduced in \cref{s:notations}.
\begin{align*}
-\int_{\Omega_u\cap A_R}|\nabla w_R|^2\,dx&=\int_{A_R}-|\nabla (u'+\xi\cdot\nabla u)|^2\,dx\\
&=\int_{\Omega_u\cap A_R}-|\nabla u'|^2-2\nabla u'\cdot\nabla(\xi\cdot\nabla u)-|\nabla(\xi\cdot\nabla u)|^2\,dx\\
&=-\int_{\Omega_u\cap A_R}\Big(|\nabla u'|^2+|\nabla(\xi\cdot\nabla u)|^2\Big)\,dx-2\int_{\partial(\Omega_u\cap A_R)}\frac{\partial u'}{\partial \nu} (\xi\cdot\nabla u)\,d\HH^{d-1}\\
&=-\int_{\Omega_u\cap A_R}\Big(|\nabla u'|^2+|\nabla(\xi\cdot\nabla u)|^2\Big)\,dx-2\int_{\partial(\Omega_u\cap A_R)}\frac{\partial u'}{\partial \nu} u'\,d\HH^{d-1}\\
&=\int_{\Omega_u\cap A_R}\Big(|\nabla u'|^2-|\nabla(\xi\cdot\nabla u)|^2\Big)\,dx.
\end{align*}
On the other hand
\begin{align*}
\int_{\Omega_u\cap A_R}\nabla u\cdot(\delta^2 A)\nabla u\,dx
&=\int_{\Omega_u\cap A_R}|D\xi(\nabla u)|^2+\nabla u\cdot(D\xi)^2(\nabla u)\,dx\\
&\qquad-\int_{\Omega_u\cap A_R}\nabla u\cdot\Big((\xi\cdot\nabla)[D\xi]\Big)\nabla u\,dx-\int_{\Omega_u\cap A_R}2\,\text{\rm div}\,\xi\,\nabla u\cdot D\xi\nabla u\,dx\\
&\qquad\qquad+\int_{\Omega_u\cap A_R}\frac12|\nabla u|^2\text{\rm div}\Big(\xi\text{\rm div}\,\xi\Big)\,dx\\
&=\int_{\Omega_u\cap A_R}|D\xi(\nabla u)|^2+\nabla u\cdot(D\xi)^2(\nabla u)\,dx\\
&\qquad-\int_{\Omega_u\cap A_R}\nabla u\cdot\Big((\xi\cdot\nabla)[\nabla\xi]\Big)\nabla u\,dx-\int_{\Omega_u\cap A_R}2\,\text{\rm div}\,\xi\,\nabla u\cdot \nabla\xi\nabla u\,dx\\
&\qquad\qquad-\int_{\Omega_u\cap A_R}D^2 u(\nabla u)\cdot\xi (\text{\rm div}\,\xi)\,dx+\int_{\partial\Omega_u\cap A_R}\frac12|\nabla u|^2\text{\rm div}\,\xi (\xi\cdot\nu)\,d\HH^{d-1}.
\end{align*}
Notice that
\begin{align*}
-\nabla u\cdot\Big((\xi\cdot\nabla)[\nabla\xi]\Big)\nabla u&=-\partial_iu\,\xi_k\partial_{ki}\xi_j\,\partial_ju\\
&=-\partial_k\Big[\partial_iu\,\xi_k\partial_{i}\xi_j\,\partial_ju\Big]+\partial_{ki}u\,\xi_k\partial_{i}\xi_j\,\partial_ju+\partial_iu\,\xi_k\partial_{i}\xi_j\,\partial_{kj}u+(\text{\rm div}\,\xi)\partial_iu\,\partial_{i}\xi_j\,\partial_ju
\end{align*}
and
\begin{align*}
-D^2 u(\nabla u)\cdot\xi (\text{\rm div}\,\xi)&=-\partial_{ij}u\,\partial_ju\xi_i(\text{\rm div}\,\xi)\\
&=-\partial_j\Big[\partial_{i}u\,\partial_ju\xi_i(\text{\rm div}\,\xi)\Big]+\partial_{i}u\,\partial_ju\partial_j\xi_i(\text{\rm div}\,\xi)+\partial_{i}u\,\partial_ju\xi_i\partial_j(\text{\rm div}\,\xi),
\end{align*}
which leads to
\begin{align*}
-\nabla u\cdot\Big((\xi\cdot\nabla)[\nabla\xi]\Big)&\nabla u-2\text{\rm div}\,\xi\,\nabla u\cdot \nabla\xi(\nabla u)-D^2 u(\nabla u)\cdot\xi (\text{\rm div}\,\xi)\\
&=-\partial_k\Big[\partial_iu\,\xi_k\partial_{i}\xi_j\,\partial_ju\Big]-\partial_j\Big[\partial_{i}u\,\partial_ju\xi_i(\text{\rm div}\,\xi)\Big]\\
&\qquad+\partial_{ki}u\,\xi_k\partial_{i}\xi_j\,\partial_ju+\partial_iu\,\xi_k\partial_{i}\xi_j\,\partial_{kj}u+\xi_i\,\partial_{i}u\,\partial_ju\,\partial_j(\text{\rm div}\,\xi).
\end{align*}
Similarly, we can compute
\begin{align*}
-|\nabla(\xi\cdot\nabla u)|^2=-\partial_k(\xi_i\partial_iu)\partial_k(\xi_j\partial_ju)
&=-(\partial_k\xi_i\partial_iu+\xi_i\partial_{ki}u)(\partial_k\xi_j\partial_ju+\xi_j\partial_{kj}u)\\
&=-|D\xi(\nabla u)|^2-\xi_i\partial_{ki}u\xi_j\partial_{kj}u-2\partial_k\xi_i\partial_iu\xi_j\partial_{kj}u\,,
\end{align*}
and so
\begin{align*}
|\nabla\xi(\nabla u)|^2&+\nabla u\cdot(\nabla\xi)^2(\nabla u)-|\nabla(\xi\cdot\nabla u)|^2\\
&=\partial_iu\,\partial_i\xi_j\partial_j\xi_k\partial_ku-2\partial_k\xi_i\partial_iu\xi_j\partial_{kj}u-\xi_i\partial_{ki}u\xi_j\partial_{kj}u\\
&=\partial_iu\,\partial_i\xi_j\partial_j\xi_k\partial_ku-2\partial_k\xi_i\partial_iu\xi_j\partial_{kj}u\\
&\qquad-\partial_k\Big[\xi_i\partial_{i}u\xi_j\partial_{kj}u\Big]+\partial_k\xi_i\partial_{i}u\xi_j\partial_{kj}u+\xi_i\partial_{i}u\partial_k\xi_j\partial_{kj}u\\
&=\partial_iu\,\partial_i\xi_j\partial_j\xi_k\partial_ku-2\partial_k\xi_i\partial_iu\xi_j\partial_{kj}u\\
&\qquad-\partial_k\Big[\xi_i\partial_{i}u\xi_j\partial_{kj}u\Big]+\partial_k\xi_i\partial_{i}u\xi_j\partial_{kj}u\\
&\qquad\qquad+\partial_j\Big[\xi_i\partial_{i}u\partial_k\xi_j\partial_{k}u\Big]-\partial_j\xi_i\partial_{i}u\partial_k\xi_j\partial_{k}u-\xi_i\partial_{ij}u\partial_k\xi_j\partial_{k}u-\xi_i\partial_{i}u\partial_k(\text{\rm div}\xi)\partial_{k}u.
\end{align*}
Finally, by collecting the previous computations, we obtain the identity
\begin{align*}
|\nabla\xi&[\nabla u]|^2+\nabla u\cdot(\nabla\xi)^2(\nabla u)-|\nabla(\xi\cdot\nabla u)|^2\\
&-\nabla u\cdot\Big((\xi\cdot\nabla)[\nabla\xi]\Big)\nabla u-2\text{\rm div}\,\xi\,\nabla u\cdot \nabla\xi(\nabla u)-D^2 u(\nabla u)\cdot\xi (\text{\rm div}\,\xi)\\
&=\cancel{\partial_iu\,\partial_i\xi_j\partial_j\xi_k\partial_ku}-2\partial_k\xi_i\partial_iu\xi_j\partial_{kj}u-\partial_k\Big[\xi_i\partial_{i}u\xi_j\partial_{kj}u\Big]+\partial_k\xi_i\partial_{i}u\xi_j\partial_{kj}u\\
&\qquad+\partial_j\Big[\xi_i\partial_{i}u\partial_k\xi_j\partial_{k}u\Big]-\cancel{\partial_j\xi_i\partial_{i}u\partial_k\xi_j\partial_{k}u}-\xi_i\partial_{ij}u\partial_k\xi_j\partial_{k}u-\cancel{\xi_i\partial_{i}u\partial_k(\text{\rm div}\xi)\partial_{k}u}\\
&\qquad\qquad-\partial_k\Big[\partial_iu\,\xi_k\partial_{i}\xi_j\,\partial_ju\Big]-\partial_j\Big[\partial_{i}u\,\partial_ju\xi_i(\text{\rm div}\,\xi)\Big]\\
&\qquad\qquad\qquad+\partial_{ki}u\,\xi_k\partial_{i}\xi_j\,\partial_ju+\partial_iu\,\xi_k\partial_{i}\xi_j\,\partial_{kj}u+\cancel{\xi_i\,\partial_{i}u\,\partial_ju\,\partial_j(\text{\rm div}\,\xi)}\\
&=-\cancel{2\partial_k\xi_i\partial_iu\xi_j\partial_{kj}u}-\partial_k\Big[\xi_i\partial_{i}u\xi_j\partial_{kj}u\Big]+\cancel{\partial_k\xi_i\partial_{i}u\xi_j\partial_{kj}u}\\
&\qquad+\partial_j\Big[\xi_i\partial_{i}u\partial_k\xi_j\partial_{k}u\Big]-\xi_i\partial_{ij}u\partial_k\xi_j\partial_{k}u-\partial_k\Big[\partial_iu\,\xi_k\partial_{i}\xi_j\,\partial_ju\Big]\\
&\qquad\qquad-\partial_j\Big[\partial_{i}u\,\partial_ju\xi_i(\text{\rm div}\,\xi)\Big]+\cancel{\partial_{ki}u\,\xi_k\partial_{i}\xi_j\,\partial_ju}+\partial_iu\,\xi_k\partial_{i}\xi_j\,\partial_{kj}u\\
&=-\partial_k\Big[\xi_i\partial_{i}u\xi_j\partial_{kj}u\Big]+\partial_j\Big[\xi_i\partial_{i}u\partial_k\xi_j\partial_{k}u\Big]-\cancel{\xi_i\partial_{ij}u\partial_k\xi_j\partial_{k}u}-\partial_k\Big[\partial_iu\,\xi_k\partial_{i}\xi_j\,\partial_ju\Big]\\
&\qquad\qquad-\partial_j\Big[\partial_{i}u\,\partial_ju\xi_i(\text{\rm div}\,\xi)\Big]+\cancel{\partial_iu\,\xi_k\partial_{i}\xi_j\,\partial_{kj}u}\\
&=-\partial_k\Big[\xi_i\partial_{i}u\xi_j\partial_{kj}u\Big]+\partial_j\Big[\xi_i\partial_{i}u\partial_k\xi_j\partial_{k}u\Big]-\partial_k\Big[\partial_iu\,\xi_k\partial_{i}\xi_j\,\partial_ju\Big]-\partial_j\Big[\partial_{i}u\,\partial_ju\xi_i(\text{\rm div}\,\xi)\Big]\\
&=\text{\rm div}\Big(-(\xi\cdot\nabla u)D^2u(\xi)+(\xi\cdot\nabla u)D\xi(\nabla u)-(\nabla u\cdot \nabla\xi(\nabla u))\xi-(\xi\cdot\nabla u)(\text{div}\,\xi)\nabla u\Big).
\end{align*}
Thus, integrating by parts and using that $\nabla u= - \nu$ on $\partial \O_u\cap A_R$, we obtain
\begin{align*}
\frac12&\frac{d^2}{dt^2}\Big|_{t=0}\mathcal G(u_t,A_R)=\int_{A_R\cap\Omega_u}\nabla u\cdot(\delta^2 A)\nabla u +\frac12\Big((\dive\xi)^2+\xi\cdot\nabla(\text{\rm div}\,\xi)\Big)-|\nabla w_R|^2\,dx\\
&=\int_{\Omega_u\cap A_R}|\nabla u'|^2\,dx+\cancel{\frac12\int_{\partial\Omega_u}|\nabla u|^2\text{\rm div}\,\xi (\xi\cdot\nu)\,d\HH^{d-1}}+\cancel{\frac12\int_{\partial\Omega_u}\text{\rm div}\,\xi (\xi\cdot\nu)\,d\HH^{d-1}}\\
&\qquad+\int_{\partial\Omega_u\cap A_R}\Big(-(\xi\cdot\nabla u)D^2u(\xi)+(\xi\cdot\nabla u)D\xi(\nabla u)-(\nabla u\cdot \nabla\xi(\nabla u))\xi-\cancel{(\xi\cdot\nabla u)(\text{div}\,\xi)\nabla u}\Big)\cdot\nu\,d\HH^{d-1}\\
&=\int_{\Omega_u\cap A_R}|\nabla u'|^2\,dx+\int_{\partial\Omega_u\cap A_R}(\xi\cdot \nu)\Big(\nu \cdot D^2 u (\xi) -\cancel{\nabla u\cdot D\xi(\nabla u)}+\cancel{\nabla u\cdot \nabla\xi(\nabla u)}\Big)\,d\HH^{d-1}\\
&=\int_{\Omega_u\cap A_R}|\nabla u'|^2\,dx+\int_{\partial\Omega_u\cap A_R}(\xi\cdot\nu)^2 (\nu \cdot D^2u (\nu))\,d\HH^{d-1}\\
&=\int_{\Omega_u\cap A_R}|\nabla u'|^2\,dx-\int_{\partial\Omega_u\cap A_R}(\xi\cdot\nu)^2 H\,d\HH^{d-1}\\
&=\int_{\Omega_u\cap A_R}|\nabla u'|^2\,dx-\int_{\partial\Omega_u\cap A_R}(u')^2H\,d\HH^{d-1},
\end{align*}
where $H$ is the mean curvature of $\partial \O_u\cap A_R$ oriented towards the complement of $\O_u$.\\

\noindent{\bf Step 4. Conclusion.} Given any $\varphi\in C^\infty_c(\R^d\setminus\{0\})$, we consider an annulus $A_R$ containing the support of $\varphi$ and the vector field $\xi:=\varphi\nabla u$. Thus, by the minimality of $u'$ in $\Omega_u\cap A_R$, we get
\begin{align*}
\int_{\Omega_u}|\nabla \varphi|^2\,dx-\int_{\partial\Omega_u}\varphi^2H\,d\HH^{d-1}&\ge \int_{\Omega_u\cap A_R}|\nabla u'|^2\,dx-\int_{\partial\Omega_u\cap A_R}(u')^2H\,d\HH^{d-1}\\
&=\frac12\frac{d^2}{dt^2}\Big|_{t=0}\mathcal G(u_t,A_R)\ge \frac12\delta^2\mathcal{G}(u)[\xi]\ge 0,
\end{align*}
where the last inequality follows from \eqref{e:global-stable-solutions-1st-inequality}.
\end{proof}

\section{Hausdorff dimension of $\text{\rm Sing}(\partial \Omega)$}\label{section.sing}
In this Section we will estimate the dimension of the singular set $\text{\rm Sing}(\partial \Omega)$ of a solution $\Omega$ to the shape optimization problem \eqref{e:intro-shape-opt-pb-in-main-teo}; our main result is the following.

\begin{theo}[Dimension of the singular set]\label{t:dimension-of-sing}
Let $D$ be a bounded open set in $\R^d$ and let $f,g,Q$ be as in \cref{thm.main}, namely:
\begin{enumerate}[\rm(a)]
	\item $f,g\in C^2(D)\cap L^\infty(D)$;
	\item there are constants $C_1,C_2 > 0$ such that $0\le C_1 g\le f\le C_2 g$ in $D$;
	\item $Q\in C^{2}(D)$ and there are constants $c_Q,C_Q$ such that $0<c_Q\le Q\leq C_Q$ on $D$.
\end{enumerate}
Let $\Omega$ be a solution to \eqref{e:intro-shape-opt-pb-in-main-teo} and let the singular part $\text{\rm Sing}(\partial\Omega)$ of the free boundary $\partial\Omega\cap D$ be as in \cref{s:decomposition}. Then, the following holds:	
\begin{enumerate}[\rm (i)]
\item If $d< d^\ast$, then $\text{\rm Sing}(\partial\Omega)=\emptyset$.
\item If $d\ge d^\ast$, then the Hausdorff dimension of $\text{\rm Sing}(\partial\Omega)$ is at most $d-d^\ast$, that is,
$$\HH^{d-d^\ast+\eps}\big(\text{\rm Sing}(\partial\Omega)\big)=0\quad\text{for every}\quad \eps>0.$$
\end{enumerate}		
\end{theo}

In order to prove \cref{t:dimension-of-sing} above, we will first show that the stability of $\Omega$, expressed in terms of the state functions $u_\Omega$ and $v_\Omega$ as in \cref{l:secondstv}, passes to a blow-up limit. 

\begin{lemma}[Stability of the blow-up limits]\label{l:stability-of-blow-up-limits}
Let $D$ be a bounded open set in $\R^d$ and let $f,g,Q$ be as in \cref{thm.main} and \cref{t:dimension-of-sing}. Let $\Omega$ be an optimal set for \eqref{e:intro-shape-opt-pb-in-main-teo} and let $u:=u_\Omega$ and $v:=v_\Omega$ be the state functions defined in \eqref{e:state-equation-open} and \eqref{e:stateqv}. 
Suppose that the couple $u_0,v_0:\R^d\to\R$ is a blow-up limit of $u,v$ at a point $x_0\in\partial\Omega\cap D$. Then, 
\begin{equation}\label{e:stability-of-the-first-blow-up}
\int_{\Omega_0}\Big(\nabla u_0\cdot(\delta^2A)\nabla v_0-\nabla (\delta u_0)\cdot \nabla (\delta v_0) +Q(x_0)\frac{(\text{\rm div}\xi)^2+\xi\cdot\nabla(\text{\rm div}\xi)}2\Big)\,dx\ge 0,
\end{equation}
where $\Omega_0:=\{u_0>0\}=\{v_0>0\}$, and $\delta u_0$ and $\delta v_0$ are the solutions respectively to the PDEs
\begin{align}\label{e:blow-up-limit-delta-u-v}
	\begin{aligned}
	-\Delta (\delta u_{0})=\text{\rm div}\big((\delta A)\nabla u_{0}\big)\quad\text{in}\quad \Omega_{0}\ ,\qquad \delta u_{0}\in \dot H^1_0(\Omega_{0})\,,\\
 -\Delta (\delta v_{0})=\text{\rm div}\big((\delta A)\nabla u_{0}\big)\quad\text{in}\quad \Omega_{0}\ ,\qquad \delta v_{0}\in \dot H^1_0(\Omega_{0})\,,
	\end{aligned}
	\end{align}
 in the sense explained in \cref{sub:section-stable-general-pde-theory}. In particular, if $u_0$ and $v_0$ are proportional, then there is a constant $\lambda>0$ such that $\lambda u_0$ is a global stable solution of the one-phase Bernoulli problem in the sense of \cref{def:global-stable-solutions}.
\end{lemma}
\begin{proof}
Let $r_k\to0$, be a sequence such that 
$$u_{k}(x):=\frac1{r_k}u(x_0+r_kx)\quad\text{and}\quad v_{k}(x):=\frac1{r_k}v(x_0+r_kx)\,,$$
converge respectively to $u_0$ and $v_0$ locally uniformly and (by \cref{p:first-blow-up}) strongly in $H^1_{loc}(\R^d)$. We define:
$$f_{k}(x):={r_k}f(x_0+r_kx)\,,\quad g_{k}(x):={r_k}g(x_0+r_kx)\,,\quad Q_{k}(x):=Q(x_0+r_kx)\,.$$
Then, the functions $u_{k}$ and $v_{k}$ satisfy the PDEs
$$-\Delta u_{k}=f_{k}\quad\text{in}\quad\Omega_{k}\ ,\qquad u_{k}\in H^1_0(\Omega_{k}),$$
$$-\Delta v_{k}=g_{k}\quad\text{in}\quad\Omega_{k}\ ,\qquad v_{k}\in H^1_0(\Omega_{k}),$$
where 
$$\Omega_{k}:=\frac1{r_k}(-x_0+\Omega)\,.$$
Moreover, $\Omega_{k}$ is optimal in the rescaled domain 
$$D_{k}:=\frac1{r_k}(-x_0+D)\,,$$
for the functional 
$$\mathcal F_{k}(A):=\int_A\Big(\nabla u_A\cdot\nabla v_A-u_Ag_{k}-v_A f_{k}+Q_{k}\Big)\,dx\,$$
where, this time, by $u_A$ and $v_A$ we denote the solutions to 
$$-\Delta u_A=f_{k}\quad\text{in}\quad A\ ,\qquad u_A\in H^1_0(A),$$
$$-\Delta v_A=g_{k}\quad\text{in}\quad A\ ,\qquad v_A\in H^1_0(A).$$
We fix a compactly supported smooth vector field $\xi\in C^\infty_c(\R^d;\R^d)$. Since $r_k\to 0$, 
for $k$ large enough the support of $\xi$ is contained in $D_{k}$ and the stability of $\Omega_{k}$ (\cref{l:secondstv}) reads as 
\begin{equation}\label{e:second-variation-along-blow-up-sequence}
\int_{\R^d}\Big(\nabla u_{k}\cdot(\delta^2A)\nabla v_{k}-\nabla (\delta u_{k})\cdot \nabla (\delta v_{k})-(\delta^2 f_{k}) v_{k} - (\delta^2 g_{k}) u_{k}+\delta^2Q_{k}\Big)\,dx\ge 0,
\end{equation}
where $\delta u_{k}$ and $\delta v_{k}$ are the solutions to 
$$-\Delta (\delta u_{k})=\text{\rm div}\big((\delta A)\nabla u_{k}\big)+\delta f_{k}=\text{\rm div}\big((\delta A)\nabla u_{k}+f_k\xi\big)\quad\text{in}\quad \Omega_{k}\ ,\qquad \delta u_{k}\in H^1_0(\Omega_{k}),$$
$$-\Delta (\delta v_{k})=\text{\rm div}\big((\delta A)\nabla v_{k}\big)+\delta g_{k}=\text{\rm div}\big((\delta A)\nabla v_{k}+g_k\xi\big)\quad\text{in}\quad \Omega_{k}\ ,\qquad \delta v_{k}\in H^1_0(\Omega_{k}),$$
and where we used the following notation: 
\begin{itemize}
\item $\delta A$ and $\delta^2A$ are the matrices defined in \eqref{e:first-and-second-variation-A} (we notice that $\delta A$ and $\delta^2A$ are defined in terms of $\xi$ only);
\item the variations $\delta f_{k}$ and $\delta^2 f_{k}$ ($\delta g_{k}$ and $\delta^2 g_{k}$ are defined analogously) are given by;
\begin{align*}
	\begin{aligned}
	\delta f_k(x) &:=\dive(f_k\xi)=r_k^2\nabla f(x_0+r_kx)\cdot\xi+r_kf(x_0+r_kx)\,\text{\rm div}\,\xi\,,\\
	\delta^2 f_k(x) &:= \frac{r_k^3}2 \xi \cdot (D^2f(x_0+r_kx))\xi+\frac{r_k^2}2\nabla f(x_0+r_kx)\cdot D\xi[\xi]\\
 &\quad+r_kf(x_0+r_kx)\frac{(\dive\xi)^2+\xi\cdot\nabla[\text{\rm div}\,\xi]}{2}+r_k^2(\nabla f(x_0+r_kx) \cdot \xi)\dive\xi \,,
	\end{aligned}
	\end{align*}
 \item $\delta^2Q_k$ is given by 
 \begin{align*}
	\begin{aligned}
 \delta^2Q_k(x)&:=(\xi\cdot\nabla Q_k)\text{\rm div}\xi+\frac12\xi \cdot D^2 Q_k \xi
+Q_k\frac12\Big((\text{\rm div}\xi)^2+\xi\cdot\nabla(\text{\rm div}\xi)\Big)\\
&=r_k(\xi\cdot\nabla Q(x+r_kx))\text{\rm div}\xi+\frac{r_k^2}2\xi \cdot D^2 Q(x_0+r_kx) \xi\\
&\qquad
+\frac12Q(x_0+r_k)\Big((\text{\rm div}\xi)^2+\xi\cdot\nabla(\text{\rm div}\xi)\Big),
\end{aligned}
\end{align*}
where in all the formulas above, the field $\xi$ and its derivatives are all computed in $x$.
\end{itemize} 
We notice that $\delta f_k$ and $\delta^2f_k$ vanish outside the support of $\xi$. Moreover, since $f$ and $g$ are $C^2$, we get that $\delta f_k$ and $\delta^2f_k$ converge to zero uniformly. Similarly, since $Q$ is $C^2$, we get that 
$$\delta^2Q_k\to \frac12Q(x_0)\Big((\text{\rm div}\xi)^2+\xi\cdot\nabla(\text{\rm div}\xi)\Big)\quad\text{strongly in}\quad L^1(\R^d).$$
 Next, we notice that by \cref{p:first-blow-up}, $u_k$ and $v_k$ converge strongly in $H^1_{loc}(\R^d)$ to the blow-up limits $u_0$, $v_0$. Thus, since the support of $\delta^2A$ is compact, we get
 $$\lim_{k\to\infty}\int_{\R^d}\nabla u_k\cdot(\delta^2A)\nabla v_k\,dx= \int_{\R^d}\nabla u_0\cdot(\delta^2A)\nabla v_0\,dx.$$
 Similarly, we have that 
 $$(\delta A)\nabla u_{k}+f_k\xi\to (\delta A)\nabla u_0\qquad\text{and}\qquad(\delta A)\nabla v_{k}+g_k\xi\to (\delta A)\nabla v_0\,,$$
 strongly in $L^2(\R^d;\R^d)$. Thus, by \cref{l:equation-with-div-F-convergence}, we get that $\delta u_k$ and $\delta v_k$ converge respectively to the solutions $\delta u_0$ and $\delta v_0$ of \eqref{e:blow-up-limit-delta-u-v}. Now, \eqref{e:stability-of-the-first-blow-up} follows by passing to the limit \eqref{e:second-variation-along-blow-up-sequence}.
 \end{proof}

\begin{proof}[\bf Proof of \cref{t:dimension-of-sing}]
    The strategy is similar to the one for minimizers of the one-phase problem. We split the proof in two different cases.\\

    Case 1: $d<d^*$. Let $x_0 \in \partial \O_u\cap D$ and $r_k \to 0^+$ be the infinitesimal sequence such that the rescalings $u_{x_0,r_k}$ and $v_{x_0,r_k}$ converge to the $1$-homogeneous blow-up limits $u_0$ and $v_0$ given by \cref{p:final-blow-up}. Then, given $\lambda \in (C_1,C_2)$ as in \cref{p:final-blow-up}, the blow-up limits
    \be
    u_0:=\lim_{k\to\infty}\frac{1}{ {\sqrt{\lambda} Q(x_0)}}u_{x_0,r_k}\qquad\text{and}\qquad v_0:=\lim_{k\to\infty}\sqrt{\lambda}v_{x_0,r_k}\,\ee
    coincide up to a multiplicative constant (that is $v_0=Q(x_0)v_0$) and, by \cref{l:stability-of-blow-up-limits} and \cref{p:final-blow-up}, $u_0$ is a $1$-homogeneous stable solutions of the one-phase problem in the sense of \cref{def:global-stable-solutions}. Therefore, by definition of $d^*$, $\text{\rm Sing}(\partial\Omega_{u_0})=\emptyset$ and $u_0$ is a half space solution:
$$u_0(x)=(x\cdot\nu)_+\quad\text{for some unit vector}\quad \nu\in\R^d.
$$ Finally, by rewriting the result in terms of the blow-up sequence of \cref{p:final-blow-up}, we deduce the existence $\alpha>0$ and $\beta>0$ such that $\alpha \beta =Q(x_0)$ such that 
$$
\frac{1}{r_k}u(x_0+r_k x) \to \alpha(x\cdot \nu)_+,\quad\mbox{and}\quad\frac{1}{r_k}v(x_0+r_k x) \to \beta(x\cdot \nu)_+
$$
as $k\to \infty$. Thus, by definition, $x_0 \in \text{\rm Reg}(\partial \O)$. Since this is true at every free boundary point $x_0\in\partial\Omega\cap D$, we get that $\text{\rm Sing}(\partial\Omega)=\emptyset$.\\


Case 2: $d\geq d^*$. Given $\eps>0$, let us prove that 
$$\HH^{d-d^\ast+\eps}\big(\text{\rm Sing}(\partial\Omega)\big)=0.
$$
We will apply consecutively the three blow-ups from \cref{section.blowup}. By contradiction, assume that 
\begin{equation*}
\HH^{d-d^\ast+\eps}\big(\text{\rm Sing}(\partial\Omega)\big)\geq C>0.
\end{equation*}
Therefore, by \cite[Lemma 10.5]{BV15}, there are $x_0 \in \text{\rm Sing}(\partial \O)$ and a sequence $r_k\to 0$ such that 
$$\HH^{d-d^\ast+\eps}\big(\text{\rm Sing}(\partial\Omega)\cap B_{r_k}(x_0)\big)\geq C r_k^{d-d^\ast+\eps}.
$$
By the first blow-up analysis from 
\cref{s:1} (see \cref{p:first-blow-up}), we deduce that the blow-up sequences $u_{x_0,r_k},v_{x_0,r_k}$ and $\O_{x_0,r_k}$ converge to some limits $u_0, v_0$ and $\O_0$ such that 
$$\HH^{d-d^\ast+\eps}\big(\text{\rm Sing}(\partial\Omega_0)\cap B_{1}\big)\geq C,
$$
where $u_0$, $v_0$ and $\Omega_0$ are as in \cref{p:first-blow-up} and where $u_0$ and $v_0$ satisfy the stability condition \eqref{e:stability-of-the-first-blow-up} from \cref{l:stability-of-blow-up-limits}. By applying again \cite[Lemma 10.5]{BV15},  there exists a point $x_{00} \in \text{\rm Sing}(\partial \O_0)$ and another sequence $r_k\to 0^+$ such that 
$$\HH^{d-d^\ast+\eps}\big(\text{\rm Sing}(\partial\Omega_0)\cap B_{r_k}(x_{00})\big)\geq C r_k^{d-d^\ast+\eps}.
$$
Hence, by applying the second blow-up analysis of \cref{sub.homo} to the sequences
\be\label{e:bl}
\frac{1}{r_k\sqrt{\lambda}\,Q(x_0)}u_0(x_{00}+r_kx),\qquad \frac{\sqrt{\lambda}}{r_k}v_0(x_{00}+r_kx)\quad\mbox{and}\quad\frac{1}{r_k}(\O_0-x_{00}),
\ee
we get that 
$$\HH^{d-d^\ast+\eps}\big(\text{\rm Sing}(\partial\Omega_{00})\cap B_{1}\big)\geq C,$$
where $v_{00}=Q(x_0) u_{00}$ and $u_{00}$ satisfies (2)-(3)-(4)-(5)-(6) in \cref{p:second-blow-up}. Moreover, by the strong convergence of the blow-up sequence, $u_{00}$ and $v_{00}$ still satisfy \eqref{e:stability-of-the-first-blow-up}, so $u_{00}$ is a global stable solution of the one-phase problem in the sense of \cref{def:global-stable-solutions}. Finally, by applying for the last time \cite[Lemma 10.5]{BV15}, we deduce that
\begin{equation}\label{e:Hausd}\HH^{d-d^\ast+\eps}\big(\text{\rm Sing}(\partial\Omega_{000})\cap B_{1}\big)\geq C,
\end{equation}
in which $\O_{000}=\{u_{000}>0\}$ and $u_{000}$ is a blow-up limit of $u_{00}$ at some free boundary point $x_{000}\in\partial\Omega_{00}$. But now $u_{000}$ is a $1$-homogeneous stable solution of the one-phase problem in the sense of \cref{def:global-stable-solutions}, so  \eqref{e:Hausd} contradicts \cref{t:stable-dimension-reduction}.
\end{proof}



\section{Existence of optimal sets in $\R^d$ and proof of Theorem~\ref{cor.main}}\label{s:existence}
In this section we prove \cref{cor.main}; we prove the existence of an optimal set for \eqref{e:intro-shape-opt-pb-in-main-cor} and then we show how to obtain the regularity of the optimal sets as a consequence from \cref{thm.main}.
\subsection{Statement of the problem in the class of measurable sets} One can not obtain the existence of optimal sets (for \eqref{e:intro-shape-opt-pb-in-main-cor} or \eqref{e:intro-shape-opt-pb-in-main-teo}) directly in the class of open sets. Thus, we extend the definition of the shape functional to the class of measurable sets. Precisely, if $\Omega$ is a Lebesgue measurable set of finite measure in $\R^d$, we define the space
$\widetilde H^1_0(\Omega)$ of all $H^1(\R^d)$ functions vanishing (Lebesgue-)almost-everywhere outside $\Omega$. We will say that $u$ is a (weak) solution to the problem
\begin{equation}\label{e:state-equation-measure}
-\Delta u=f\quad \text{in}\quad \Omega\,,\qquad u\in \widetilde H^1_0(\Omega)\,,
\end{equation}
if $u\in \widetilde H^1_0(\Omega)$ and
\begin{equation}\label{e:state-equation-measure-weak}
\int_{\Omega}\nabla v\cdot \nabla u\,dx=\int_{\Omega}vf\,dx\quad\text{for every}\quad v\in \widetilde H^1_0(\Omega).
\end{equation}
It is immediate to check that $u$ satisfies \eqref{e:state-equation-measure-weak} if and only if
\begin{equation}\label{e:def-J-f-intro}
\frac12\int_{\R^d}|\nabla u|^2\,dx-\int_{\Omega}uf(x)\,dx\le \frac12\int_{\R^d}|\nabla v|^2\,dx-\int_{\Omega}vf(x)\,dx\quad\text{for every}\quad v\in \widetilde H^1_0(\Omega)\,.
\end{equation}
From now on, we will denote the unique solution to \eqref{e:state-equation-measure} by $\widetilde u_\Omega$.\medskip

We will first prove the following lemma, which in particular implies that for every open set $\Omega\subset\R^d$ of finite measure and non-negative functions $f,g\in L^2(\Omega)$, we have
$$\int_\Omega\Big(-g(x)\widetilde u_\Omega+Q(x)\Big)\,dx\le \int_\Omega\Big(-g(x) u_\Omega+Q(x)\Big)\,dx\,,$$
where $u_\Omega$ is the weak solution to \eqref{e:state-equation-open}.
\begin{lemma}\label{l:comparison-open-measurable}
	Let $\Omega$ be a bounded open set in $\R^d$ and $f\in L^2(\Omega)$ a non-negative function. Then
	\begin{equation}\label{e:existence-inequality}
	0\le u_\Omega\le \widetilde u_\Omega\,,
	\end{equation}
	where $\widetilde u_\Omega\in \widetilde H^1_0(\Omega)$ is the solution to \eqref{e:state-equation-measure} and $u_\Omega\in H^1_0(\Omega)$ is the weak  solution to \eqref{e:state-equation-open}.\\	
	Moreover, if $\Omega$ satisfies the following exterior density estimate:
	\begin{equation}\label{e:exterior-density-intro}
	\begin{array}{ll}
	\text{there are constants $r_0>0$ and $c>0$ such that}\smallskip\\
	\qquad \text{for every $x_0\in\partial\Omega$ and every $r\le r_0$}\quad |B_r(x_0)\setminus \Omega|\ge c|B_r|\,,
	\end{array}
	\end{equation}
	then $H^1_0(\Omega)=\widetilde H^1_0(\Omega)$ and $ u_\Omega=\widetilde u_\Omega$.
\end{lemma}	
\begin{proof}
	First we notice that $u_\Omega$ and $\widetilde u_\Omega$ are the unique minimizers of the functional \eqref{e:def-J-f-intro} in $H^1_0(\Omega)$ and $\widetilde H^1_0(\Omega)$. Since $f\ge 0$, by testing the optimality of $u_\Omega$ with $u_\Omega\vee 0\in H^1_0(\Omega)$ and the optimality of $\widetilde u_\Omega$ with $\widetilde u_\Omega\vee 0\in \widetilde H^1_0(\Omega)$, we get that $u_\Omega\ge 0$ and $\widetilde u_\Omega\ge 0$ in $\Omega$. Now, a standard argument (see for instance \cite[Lemma~2.6]{V}) gives that
	$\Delta u_\Omega+f\ge 0$ on $\R^d$ in the sense of distributions, precisely:
	\begin{equation}\label{e:equation-u-Omega-across-the-boundary}
	-\int_{\R^d}\nabla u_\Omega\cdot\nabla \varphi\,dx+\int_{\R^d}\varphi f\,dx\ge 0\quad\text{for every}\quad \varphi\ge 0,\ \varphi\in H^1(\R^d).
	\end{equation}
	Next, we notice that for every open set $\Omega$, we have the inclusion $H^1_0(\Omega)\subset\widetilde H^1_0(\Omega)$. Thus,
	if we take $\varphi$ to be the negative part of $\widetilde u_\Omega-u_\Omega$,
	$$\varphi:=-\Big((\widetilde u_\Omega-u_\Omega)\wedge 0\Big),$$
	we have that $\varphi\in \widetilde H^1_0(\Omega).$
	So, using the (weak) equation for $\widetilde u_\Omega$, we get
	$$\int_{\R^d}\nabla \varphi\cdot\nabla \widetilde u_\Omega\,dx=\int_{\Omega} \varphi\,f(x)\,dx$$
	On the 	other hand, by the positivity of $\varphi$ and \eqref{e:equation-u-Omega-across-the-boundary},
	$$\int_{\R^d}\nabla \varphi\cdot\nabla  u_\Omega\,dx\le \int_\Omega  \varphi\,f(x)\,dx.$$
	Combining the two, we obtain
	$$\int_{\R^d}|\nabla\varphi|^2\,dx=-\int_{\R^d}\nabla \varphi\cdot\nabla  (\widetilde u_\Omega-u_\Omega)\,dx\le 0,$$
	which gives that $\varphi\equiv 0$ and so \eqref{e:existence-inequality} holds. Finally, it is known (see for instance \cite[Proposition~4.7]{deve}) that in the presence of the density estimate \eqref{e:exterior-density-intro}, we have that $H^1_0(\Omega)=\widetilde H^1_0(\Omega)$. Thus, the equality $u_\Omega=\widetilde u_\Omega$ follows by the fact that the minimizer of the functional in \eqref{e:def-J-f-intro} is unique.
\end{proof}	

\subsection{Existence of optimal measurable sets in $\R^d$}
For all measurable set $\Omega\subset \R^d$, we set \[
\widetilde J(\Omega):=\int_\Omega -g(x)u_\Omega+Q(x)\,dx.
\]
We can prove the following existence result for the minimization of $\widetilde J$.
\begin{lemma}\label{l:existence-Rd}
If $d\ge 3$, let $f,g\in L^2(\R^d)$, while if $d=2$, let $f,g\in L^1(\R^2)\cap L^\infty(\R^2)$ and let $Q:\R^d\to\R$ be a measurable function bounded from below by a positive constant $c_Q>0$.
	Then, the following variational problem has a solution
	\begin{equation}\label{e:intro-shape-opt-pb-in-main-cor-2}
	\min \bigg\{\int_\Omega \Big(-g(x)\widetilde u_\Omega+Q(x)\Big)\,dx\ :\ \Omega\subset \R^d,\ \Omega\ \text{measurable},\ |\Omega|<+\infty\bigg\}\,.
	\end{equation}
\end{lemma}
\begin{proof}
	Let $d\ge 3$, let $\Omega_n$ be a minimizing sequence of sets of finite Lebesgue measure for \eqref{e:intro-shape-opt-pb-in-main-cor-2}. We set $u_n:=\widetilde u_{\Omega_n}$. By the Poincar\'e inequality, the equation for $u_n$ and the H\"older inequality, there is a dimensional constant $C_d$ such that
	$$\|u_n\|_{L^2(\Omega_n)}^2\le C_d|\Omega_n|^{\sfrac{2}{d}}\int_{\Omega_n}|\nabla u_n|^2\,dx=C_d|\Omega_n|^{\sfrac{2}{d}}\int_{\Omega_n}u_nf\,dx\le C_d|\Omega_n|^{\sfrac{2}{d}}\|f\|_{L^2(\R^d)}\|u_n\|_{L^2(\R^d)}.$$
	Thus,
	$$\|u_n\|_{L^2(\Omega_n)}\le C_d|\Omega_n|^{\sfrac{2}{d}}\|f\|_{L^2(\R^d)},$$
	and so (since we can suppose that $\widetilde J(\Omega_n)\le \widetilde J(\emptyset)=0$), we get
	\begin{equation*}
	0\ge \widetilde J(\Omega_n)=\int_{\Omega_n}\Big(-g(x)u_n+Q(x)\Big)\,dx\ge -C_d|\Omega_n|^{\sfrac{2}{d}}\|f\|_{L^2(\R^d)}\|g\|_{L^2(\R^d)}+c_Q|\Omega_n|,
	\end{equation*}
	which implies that the sequence of measures $|\Omega_n|$ is bounded
\begin{equation}\label{eq:boundmeas}
|\Omega_n|^{\frac{d-2}{d}}\le \frac{C_d}{c_Q}\|f\|_{L^2(\R^d)}\|g\|_{L^2(\R^d)}\,.
\end{equation}

If $d=2$, to obtain a bound on the measure analogous to~\eqref{eq:boundmeas}, we need to use a generalized Poincar\'e-Sobolev (or generalized Faber-Krahn) inequality (see for example~\cite[equation (1.2)]{brma}) instead of the classical one, namely, for $2<q<+\infty$, there exists a dimensional constant $\widetilde C_{2}$ such that (using also H\"older inequality)
\[
\begin{split}
\|u_n\|^2_{L^q(\Omega_n)}&\le \widetilde C_{2}|\Omega_n|^{1+\tfrac{2-q}{q}}\int_{\Omega_n}|\nabla u_n|^2\,dx=C_{d=2}|\Omega_n|^{1+\tfrac{2-q}{q}}\int_{\Omega_n}u_nf\,dx\\
&\le \widetilde C_{2}|\Omega_n|^{1+\tfrac{2-q}{q}}\|f\|_{L^{q'}(\R^2)}\|u_n\|_{L^q(\Omega_n)}.
\end{split}
\]
We immediately deduce that 
\begin{equation}\label{eq:boundintermedio}
\|u_n\|_{L^q(\Omega_n)} \widetilde C_{2}|\Omega_n|^{1+\tfrac{2-q}{q}}\|f\|_{L^{q'}(\R^2)},
\end{equation}
and again supposing that $\widetilde J(\Omega_n)\le \widetilde J(\emptyset)=0$, we obtain, using~\eqref{eq:boundintermedio}
	\begin{equation*}
	0\ge \widetilde J(\Omega_n)=\int_{\Omega_n}\Big(-g(x)u_n+Q(x)\Big)\,dx\ge -\widetilde C_2|\Omega_n|^{1+\tfrac{2-q}{q}}\|f\|_{L^{q'}(\R^2)}\|g\|_{L^{q'}(\R^2)}+c_Q|\Omega_n|,
	\end{equation*}
	which, since $q>2$ implies that the sequence of measures $|\Omega_n|$ is bounded, namely
\begin{equation}\label{eq:boundmeas2}
|\Omega_n|^{\frac{q-2}{q}}\le \frac{\widetilde C_2}{c_Q}\|f\|_{L^{q'}(\R^2)}\|g\|_{L^{q'}(\R^2)}\,.
\end{equation}

Now the proof continues in the same way both for $d=2$ and $d\geq 3$. Using again the equation for $u_n$ we get that the sequence $u_n$ is bounded in $H^1(\R^d)$. By the concentration-compactness principle and the optimality of $\Omega_n$, we get that $u_n$ converges weakly in $H^1$ and strongly in $L^2$ (and, up to extracting a further subsequence, pointwise almost-everywhere in $\R^d$) to a function $u\in H^1(\R^d)$. Now, by the semicontinuity of the $H^1$ norm and the fact that
	$$|\{u\neq0\}|\le\liminf_{n\to\infty}|\{u_n\neq 0\}|\le \liminf_{n\to\infty}|\Omega_n|\,,$$
	we get that $\Omega:=\{u\neq0\}$ is a solution to \eqref{e:intro-shape-opt-pb-in-main-cor-2}.
\end{proof}

\begin{lemma}\label{l:existence-boundedness}
	Let $d\ge 2$, $f,g\in L^1(\R^d)\cap L^\infty(\R^d)$ be non-negative, and let $Q:\R^d\to\R$ be a measurable function bounded from below by a positive constant $c_Q>0$. If the measurable set $\Omega\subset\R^d$, $|\Omega|<+\infty$, is a solution to
	\eqref{e:intro-shape-opt-pb-in-main-cor-2}, then $\Omega$ is bounded.
\end{lemma}
\begin{proof}
	For any set $E\subset\R^d$ of finite measure and any function $h\in L^2(E)$, we will denote by $\widetilde R_E(h)$ the unique solution to the problem
	$$-\Delta u=h\quad\text{in}\quad E\ ,\qquad u\in \widetilde H^1_0(E)\,.$$
	We have that $\widetilde R_E$ is linear, $\widetilde R_E(h_1+h_2)=\widetilde R_E(h_1)+\widetilde R_E(h_2)$ and positive: if $h\ge 0$, then $\widetilde R_E(h)\ge 0$. Moreover, by the weak maximum principle, if $E_1\subset E_2$ and $h\ge 0$, then
	$\widetilde R_{E_2}(h)\ge \widetilde R_{E_1}(h).$
	
	Let $\omega$ be any measurable set contained in $\Omega$.
	Then, the minimality of $\Omega$ gives that
	\begin{equation*}
	\int_{\Omega}\Big(-g(x)\widetilde R_\Omega(f) +Q(x)\Big)\,dx\le \int_{\omega}\Big(-g(x)\widetilde R_\omega(f)+Q(x)\Big)\,dx\,.
	\end{equation*}
	So, rearranging the terms and using the positivity of $f$ and $g$, and the inequality
	$$0\le \widetilde R_\Omega(f)-\widetilde R_\omega(f)\le \|f\|_{L^\infty}\Big(\widetilde R_\Omega(1)-\widetilde R_\omega(1)\Big)\,,$$
	we get
	\begin{align*}
	c_Q|\Omega\setminus\omega|&\le \int_\Omega Q(x)\,dx-\int_\omega Q(x)\,dx\le \int_{\Omega}g(x)\Big(\widetilde R_\Omega(f)-\widetilde R_\omega(f)\Big)\,dx\\
	&\le \|f\|_{L^\infty}\int_{\Omega}g(x)\Big(\widetilde R_\Omega(1)-\widetilde R_\omega(1)\Big)\,dx= \|f\|_{L^\infty}\int_{\Omega}\Big(\widetilde R_\Omega(g)-\widetilde R_\omega(g)\Big)\,dx\\
	&\le \|f\|_{L^\infty}\|g\|_{L^\infty}\int_{\Omega}\Big(\widetilde R_\Omega(1)-\widetilde R_\omega(1)\Big)\,dx\,.
	\end{align*}
	Finally, rearranging the terms again, we get that
	$$-\frac12\int_{\Omega}\widetilde R_\Omega(1)\,dx+\frac{c_Q}{2\|f\|_{L^\infty}\|g\|_{L^\infty}}|\Omega|\le -\frac12\int_{\omega}\widetilde R_\omega(1)\,dx+\frac{c_Q}{2\|f\|_{L^\infty}\|g\|_{L^\infty}}|\omega|\,,$$
	so the set $\Omega$ is inwards minimizing (or, in terms of \cite{bucve}, a shape subsolution) for the functional
	$$\mathcal E(\Omega)=-\frac12\int_{\Omega}\widetilde R_\Omega(1)\,dx+\frac{c_Q}{2\|f\|_{L^\infty}\|g\|_{L^\infty}}|\Omega|.$$
	Thus, applying \cite[Theorem~3.13]{bucve}, for every
	$$0<\eta\le \frac{c_Q}{2\|f\|_{L^\infty}\|g\|_{L^\infty}},$$
	the set $\Omega$ is contained in an open set $A\subset\R^d$ obtained as a finite union of $N$ balls $B_\rho(x_i)$, $i=1,\dots,N$, with $N$ and $\rho$ depending only on the dimension $d$ and $\eta$. In particular, $A$ is bounded and the diameter of any connected component of $A$ is at most $N\rho$.
\end{proof}	

\subsection{Proof of \cref{cor.main}}
The existence of an optimal set is a consequence of \cref{p:cor-main-existence} below, while the regularity follows from \cref{thm.main}.
\begin{prop}\label{p:cor-main-existence}
	In $\R^d$, $d\ge 2$, let $f,g,Q:\R^d\to\R$ be non-negative functions.
	Suppose that:
	\begin{enumerate}[\rm(a)]
		\item $f,g\in L^\infty(\R^d)\cap L^1(\R^d)\cap C(\R^d)$ and that $f>0$ {and} $g>0$ on $\R^d$\,;
		\item there is a constant $c_Q>0$ such that $c_Q\le Q$ on $\R^d$.
	\end{enumerate}
	Then, there exists a solution $\Omega\subseteq \R^d$ to the shape optimization problem \eqref{e:intro-shape-opt-pb-in-main-cor}. Moreover, every solution to \eqref{e:intro-shape-opt-pb-in-main-cor} is also a solution to \eqref{e:intro-shape-opt-pb-in-main-cor-2}.
\end{prop}
\begin{proof}
	By \cref{l:existence-Rd}, there is a measurable set $\Omega$ (of finite measure) that minimizes \eqref{e:intro-shape-opt-pb-in-main-cor-2}. By \cref{l:existence-boundedness}, $\Omega$ is contained in some ball $B_R\subset\R^d$. Since $f$ and $g$ are continuous and strictly positive on $\overline B_R$, we can find positive constants $C_1,C_2$ such that
	$C_1g\le f\le C_2g$ on $\overline B_R$. Thus, reasoning as in \cref{prop.almost} and \cref{lemm.lip} we get that the set
	$$A:=\{\widetilde u_\Omega>0\}$$
	is open. Now, the optimality of $\Omega$ gives that $|\Omega\Delta A|=0$, so we have $\widetilde u_A=\widetilde u_\Omega.$ Moreover, by \cref{p:density}, $A$ satisfies the exterior density estimate from \cref{l:comparison-open-measurable}. Thus,
	$u_A=\widetilde u_A$. In order to check that the open set $A$ minimizes \eqref{e:intro-shape-opt-pb-in-main-cor}, we notice that, for any open set $E\subset \R^d$,
	\begin{align*}
	\int_A\Big(-g(x) u_A+Q(x)\Big)\,dx&=\int_A\Big(-g(x) \widetilde u_A+Q(x)\Big)\,dx\\
	&\le \int_E\Big(-g(x) \widetilde u_E+Q(x)\Big)\,dx\le\int_E\Big(-g(x) u_E+Q(x)\Big)\,dx\,.
	\end{align*}
	Moreover, by the same chain of inequalities, we obtain that if the open set $E$ is a solution to \eqref{e:intro-shape-opt-pb-in-main-cor}, then it also minimizes \eqref{e:intro-shape-opt-pb-in-main-cor-2}.
\end{proof}

\appendix

\section{An optimization problem in heat conduction}\label{s:heat-conduction}
As a consequence of our analysis from \cref{section.variations}, \cref{section.blowup} and \cref{section.stable}, we obtain an estimate on the dimension of the singular part of the boundary of the optimal sets arising in a heat conduction problem studied by Aguilera, Caffarelli and Spruck in \cite{agcasp}. 
\subsection{Statement of the problem and known results}\label{sub:heat-conduction-intro}
Let $D$ be a smooth bounded open set in $\R^d$. Let $\phi:\partial D\to\R$ be a smooth function on $\partial D$ bounded from below and above by positive constants $0<c\le C$.
For every open set $\Omega\subset D$, for which the set $K:= D\setminus \Omega$ is a compact subset of $D$, we consider the state function $u_\Omega\in H^1(D)$ solution to the problem 
\begin{equation}\label{e:heat-conduction-u-Omega}
\Delta u_\Omega=0\quad\text{in}\quad\Omega\,,\qquad u_\Omega=\phi\quad\text{on}\quad \partial D\,,\qquad u_\Omega\equiv 0\quad\text{on}\quad K=D\setminus\Omega\,,
\end{equation}
and we define the functional 
$$\mathcal F(\Omega)=\int_{\partial D}\frac{\partial u_\Omega}{\partial \nu}\,d\HH^{d-1}+\Lambda|\Omega|\,$$
where $\Lambda>0$ is a real number and $\nu$ is the outer unit normal to $\partial D$. In \cite{agcasp} it was shown that there is a solution to the following shape optimization problem: 
\begin{equation}\label{e:heat-conduction-problem}
\min\Big\{\mathcal F(\Omega)\ :\ \Omega\subset D;\ \Omega\,-\,\text{open};\ D\setminus\Omega\,-\,\text{compact subset of\ } D\Big\}\,,
\end{equation}
and that for every solution $\Omega$ to \eqref{e:heat-conduction-problem} the following holds. 
\begin{itemize}
\item The state function $u_\Omega$ is Lipschitz continuous in $D$ and smooth in a neighborhood of the fixed boundary $\partial D$.
\item $u_\Omega$ is non-degenerate in the sense that there is a constant $\eta>0$, for which
$$\frac{1}{r^{d-1}}\int_{\partial B_r(x)}u_\Omega\,d\HH^{d-1}\ge \eta\,r\quad\text{whenever}\quad x\in\overline \Omega\quad\text{and}\quad B_r(x)\subset D\,.$$
\item There is $\eps>0$ such that, for every $x$ on the boundary of the set $K:=D\setminus\Omega$, we have: 
$$\eps_0|B_r|\le |\Omega\cap B_r(x)|\le (1-\eps_0)|B_r|\quad\text{whenever}\quad B_r(x)\subset D\,.$$
\item There is a constant $\tilde M>0$ such that, for every $B_{2r}(x_0)\subset D$, we have the bound
\begin{equation*}
0\le \bigg|\int_{B_{2r}(x_0)}\nabla u_\Omega\cdot\nabla \varphi\,dx\bigg|\le \tilde{M} \|\varphi\|_{L^\infty(B_{2r}(x_0))}\quad\text{for every}\quad \varphi\in C^{0,1}_c(B_{2r}(x_0)),
\end{equation*}
where $C^{0,1}_c(B_{2r}(x_0))$ is the space of Lipschitz functions with compact support in $B_{2r}(x_0)$.
\end{itemize}	
Then, they showed that the reduced boundary $\partial^\ast\Omega\cap D$ is $C^\infty$ smooth and that 
\begin{equation}\label{e:heat-conduction-optimality-condition}
|\nabla u_\Omega||\nabla v_\Omega|=\Lambda\quad\text{on}\quad\partial^\ast\Omega\cap D\,,
\end{equation}
where $v_\Omega$ is the solution to the PDE
\begin{equation}\label{e:heat-conduction-v-Omega}
\Delta v_\Omega=0\quad\text{in}\quad\Omega\,,\qquad v_\Omega=1\quad\text{on}\quad \partial D\,,\qquad v_\Omega\equiv 0\quad\text{on}\quad K=D\setminus\Omega\,.
\end{equation}
In \cref{t:heat-conduction}, we will improve this result in low dimension ($d\ge 4$) by showing that the whole free boundary $\partial\Omega\cap D$ is smooth.
 
\begin{oss}
	In the problem originally considered by Aguilera, Caffarelli and Spruck in \cite{agcasp}, the minimization of $\mathcal F$ is among all domains with prescribed measure. Thus, all the results from \cite{agcasp} apply to \eqref{e:heat-conduction-problem} with the only difference that the constant $\Lambda$ from \eqref{e:heat-conduction-optimality-condition} is an unknown positive Lagrange multiplier. In this last Section of the present paper, we choose to work with the penalized version \eqref{e:heat-conduction-problem} since it allows to apply all the results from \cref{section.variations}, \cref{section.stable} and \cref{section.sing} directly, without the need to add technical details related to the Lagrange multiplier. 
\end{oss}
	
\subsection{First and second variation}\label{sub:heat-conduction-variations}
By an integration by parts, we can rewrite $\mathcal F(\Omega)$ as  
$$\mathcal F(\Omega):=\int_{\Omega}\nabla u_\Omega\cdot\nabla v_\Omega\,dx+\Lambda|\Omega|\,,$$
where $u_\Omega$ and $v_\Omega$ are given by \eqref{e:heat-conduction-u-Omega} and \eqref{e:heat-conduction-v-Omega}. Now, let $\xi\in C^\infty_c(D;\R^d)$ be a smooth compactly supported vector field in $D$; let $\Phi_t:D\to D$, $t\in\R$, be the flow associated to $\xi$ and let $\Omega_t:=\Phi_t(\Omega)$. Reasoning as in \cref{section.stable}, we get that 
\begin{align*}
\begin{aligned}
\frac{\partial}{\partial t}\bigg|_{t=0}\mathcal{F}({\Omega_t})&=\, \int_\O \Big(\nabla u_\Omega \cdot (\delta A)\nabla v_\Omega + \Lambda \dive\xi\Big) \, dx\,,\\
\frac12\frac{\partial^2}{\partial t^2}\bigg|_{t=0}\mathcal{F}({\Omega_t})
&=\,
\int_\O\Big(\nabla u_\Omega\cdot(\delta^2 A)\nabla v_\Omega-\nabla (\delta u)\cdot \nabla (\delta v)+\frac{\Lambda}2\big((\dive\xi)^2+\xi\cdot\nabla(\text{\rm div}\,\xi)\big)\Big) dx\,,
\end{aligned}
\end{align*}
where $\delta A$ and $\delta^2A$ are given by \eqref{e:first-and-second-variation-A}, and where $\delta u$ and $\delta v$ are the solutions to the PDEs
\begin{align*}
\begin{aligned}
-\Delta (\delta u)=\text{\rm div}\big((\delta A)\nabla u_\Omega\big)\quad\text{in}\quad \Omega\ ,\qquad \delta u\in  H^1_0(\Omega)\,;\\
-\Delta (\delta v)=\text{\rm div}\big((\delta A)\nabla v_\Omega\big)\quad\text{in}\quad \Omega\ ,\qquad \delta v\in H^1_0(\Omega)\,.
\end{aligned}
\end{align*}
Using the blow-up analysis from \cref{section.blowup} and the argument from \cref{l:stability-of-blow-up-limits}, we get 
\begin{lemma}
Let $\Omega$ be an optimal set for \eqref{e:heat-conduction-problem} and let $u:=u_\Omega$ and $v:=v_\Omega$ be the state functions from \eqref{e:heat-conduction-u-Omega} and \eqref{e:heat-conduction-v-Omega}.  Suppose that  $x_0\in\partial\Omega\cap D$ and that $r_k\to 0$ are such that 
$$\lim_{k\to\infty}u_{x_0,r_k}=u_0\qquad\text{and}\qquad \lim_{k\to\infty}v_{x_0,r_k}=v_0\,$$
where $u_0,v_0$ are non-negative Lipschitz functions on $\R^d$ and both limits are locally uniform in $\R^d$, and where, as usual, we set: 
$$u_{x_0,r_k}(x):=\frac{1}{r_k}u(x_0+r_kx)\qquad\text{and}\qquad v_{x_0,r_k}(x):=\frac{1}{r_k}v(x_0+r_kx)\,.$$
Then, $u_0$ and $v_0$ are proportional and there is a constant $\lambda>0$ such that $\lambda u_0$ is a global stable solution of the one-phase Bernoulli problem in the sense of \cref{def:global-stable-solutions}.
\end{lemma}
\begin{proof}
By the analysis of the blow-up sequences from \cref{section.blowup}, we get that $u_{x_0,r_k}$ and $v_{x_0,r_k}$ converge strongly in $H^1_{loc}(\R^d)$ respectively to $u_0$ and $v_0$, and also that the sets $\frac1{r_k}(\Omega-x_0)$ converge (in the local Hausdorff distance) to $\Omega_0=\{u_0>0\}=\{v_0>0\}$. Then, as in \cref{l:stability-of-blow-up-limits}, for every smooth compactly supported vector field $\xi\in C^\infty_c(\R^d;\R^d)$, the first and the second variation of $\mathcal F$ pass to the limit, that is, 
\begin{align}\label{e:heat-conduction-stability-blow-ups}
\begin{aligned}
&\int_{\Omega_0} \Big(\nabla u_0 \cdot (\delta A)\nabla v_0 + \Lambda \dive\xi\Big) \, dx=0\,,\\
&\int_{\Omega_0}\Big(\nabla u_0\cdot(\delta^2 A)\nabla v_0-\nabla (\delta u_0)\cdot \nabla (\delta v_0)+\frac{\Lambda}2\big((\dive\xi)^2+\xi\cdot\nabla(\text{\rm div}\,\xi)\big)\Big) dx\geq 0\,,
\end{aligned}
\end{align}
where $\delta A$ and $\delta^2A$ are as in \eqref{e:first-and-second-variation-A}, and where $\delta u_0$ and $\delta v_0$ solve the PDE (see \cref{sub:section-stable-general-pde-theory})
\begin{align*}
\begin{aligned}
-\Delta (\delta u_0)=\text{\rm div}\big((\delta A)\nabla u_0\big)\quad\text{in}\quad \Omega_0\ ,\qquad \delta u_0\in \dot H^1_0(\Omega_0)\,;\\
-\Delta (\delta v_0)=\text{\rm div}\big((\delta A)\nabla v_0\big)\quad\text{in}\quad \Omega_0\ ,\qquad \delta v_0\in \dot H^1_0(\Omega_0)\,.
\end{aligned}
\end{align*}
On the other hand, since by \cite{agcasp} the Boundary Harnack Principle holds on the optimal set $\Omega$, we get that the blow-up limits are proportional, that is, $v_0=Cu_0$ for some $C>0$. Finally, we notice that, up to multiplying $u_0$ by a constant, the inequalities \eqref{e:heat-conduction-stability-blow-ups} correspond precisely to the stability condition \eqref{e:global-stable-solutions-main} in \cref{def:global-stable-solutions}, which concludes the proof of the lemma.
\end{proof}	

\subsection{Main theorem}\label{sub:heat-conduction-main-result} Let now $\Omega$ be an optimal set for \eqref{e:heat-conduction-problem}. We define the regular and the singular parts of the free boundary $\partial\Omega\cap D$ as follows. The regular part, $\text{\rm Reg}(\partial\Omega)$, is the set of points $x_0\in\partial\Omega\cap D$ at which there exists a blow up limit $u_0,v_0:\R^d\to\R$ of the form 
$$u_0(x)=\alpha(x\cdot\nu)_+\quad\text{and}\quad v_0(x)=\beta(x\cdot\nu)_+\,,$$
for some unit vector $\nu\in\R^d$ and some constants $\alpha>0$ and $\beta>0$ such that 
$\alpha\beta=\Lambda$. The remaining part of the free boundary is the singular set:
$$\text{\rm Sing}(\partial\Omega):=(\partial\Omega\cap D)\setminus \text{\rm Reg}(\partial\Omega).$$ 
Then, in terms of the decomposition into $\partial\Omega\cap D=\text{\rm Reg}(\partial\Omega)\cup\text{\rm Sing}(\partial\Omega)$, we can rewrite the regularity theorem of Aguilera-Caffarelli-Spruck (for the penalized problem) as follows:
\begin{theo}[Aguilera-Caffarelli-Spruck \cite{agcasp}]
If $\Omega$ is a solution to \eqref{e:heat-conduction-problem} in a domain $D\subset\R^d$, then $\text{\rm Reg}(\partial\Omega)$ is a relatively open subset of $\partial\Omega\cap D$ and (locally) a $C^\infty$ manifold.
\end{theo}	
For what concerns the singular set, by reasoning as in \cref{t:dimension-of-sing} and by applying \cref{t:stable-critical-dimension} and \cref{t:stable-dimension-reduction}, we obtain the following result, which in particular implies that in the physically relevant dimension $d=3$ (and also in $d=2$ and $d=4$) the free boundary $\partial\Omega\cap D$ of a solution $\Omega$ to \eqref{e:heat-conduction-problem} is $C^\infty$ smooth. 
\begin{theo}[Dimension of the singular set for solutions to \eqref{e:heat-conduction-problem}]\label{t:heat-conduction}
If $\Omega$ is a solution to \eqref{e:heat-conduction-problem} in a domain $D\subset\R^d$, then the following holds:	
\begin{enumerate}[\rm (i)]
	\item if $d< d^\ast$, then $\text{\rm Sing}(\partial\Omega)=\emptyset$;
	\item if $d\ge d^\ast$, then the Hausdorff dimension of $\text{\rm Sing}(\partial\Omega)$ is at most $d-d^\ast$;
\end{enumerate}		
where $d^\ast$ is the critical dimension from \cref{def:global-stable-solution-critical-dimension}.
\end{theo}


\end{document}